\documentclass{article}
\RequirePackage{amsthm,amsmath,amsfonts,amssymb}
\usepackage{titlesec}
\usepackage{graphicx}
\usepackage[hidelinks]{hyperref}
\usepackage{float}
\usepackage{mathtools}
\usepackage{afterpage}
\usepackage[export]{adjustbox}
\usepackage{graphicx}
\usepackage{subcaption}
\usepackage{dsfont}
\usepackage{enumitem}

\usepackage{color}
\usepackage[english]{babel}
\titleformat{\chapter}[display]
  {\normalfont\LARGE\bfseries}
  {\titleline{}\vspace{5pt}\titleline{}\vspace{1pt}%
  \MakeUppercase{\chaptertitlename} \thechapter}
  {1pc}
  {\titleline{}\vspace{0.5pc}} 
\usepackage{mathtools}
\DeclarePairedDelimiter\abs{\lvert}{\rvert}

\makeatletter

\renewcommand\section{\@startsection {section}{1}{\z@}%
                               {-3.5ex \@plus -1ex \@minus -.2ex}%
                               {2.3ex \@plus.2ex}%
                               {\normalfont\large\bfseries}}
\renewcommand\subsection{\@startsection{subsection}{2}{\z@}%
                                 {-3.25ex\@plus -1ex \@minus -.2ex}%
                                 {1.5ex \@plus .2ex}%
                                 {\normalfont\bfseries}}

\makeatother

\newtheorem{theorem}{Theorem}[section]
\newtheorem{prop}[theorem]{Proposition}
\newtheorem{corollary}[theorem]{Corollary}
\newtheorem{lemma}[theorem]{Lemma}
\newtheorem{remark}{Remark}[section]

\newtheorem{assumption}{Assumption}
\newtheorem{definition}[theorem]{Definition}

\graphicspath{{img/}}

\AtBeginDocument{%
   \def\MR#1{}
}

\title{
\huge 
Regularity preservation in Kolmogorov equations for non-Lipschitz coefficients under Lyapunov conditions
}
\author{Martin Chak}
\date{\normalsize\today}

\begin{document}

\maketitle

\begin{abstract}
Given global Lipschitz continuity and differentiability of high enough order on the coefficients in It\^{o}'s equation, 
differentiability of associated semigroups, existence of twice differentiable solutions to Kolmogorov equations and weak convergence rates of numerical approximations are known. In this work and against the counterexamples of Hairer et al.~\cite{MR3305998}, the drift and diffusion coefficients having Lipschitz constants that are~$o(\log V)$ and~$o(\sqrt{\log V})$ respectively for a function~$V$ satisfying~$(\partial_t + L)V\leq CV$ 
is shown to be a generalizing condition in place of global Lipschitz continuity for the above.

\end{abstract}

\section{Introduction}

Let~$b:[0,\infty)\times\mathbb{R}^n\rightarrow\mathbb{R}^n$,~$\sigma:[0,\infty)\times\mathbb{R}^n\rightarrow\mathbb{R}^{n\times n}$ and let~$W_t$ be a standard Wiener process on~$\mathbb{R}^d$. Consider the stochastic differential equation (SDE) on $\mathbb{R}^n$ given by
\begin{equation}\label{sde0}
dX_t = b(t,X_t)dt + \sigma(t,X_t)dW_t.
\end{equation}
This paper concerns the case where the coefficients~$b$ and~$\sigma$ are not globally Lipschitz continuous in space, but are only locally Lipschitz. 
The main contributions in this work are the existence of twice differentiable-in-space solutions to the associated Kolmogorov equations \cite{MR1731794} 
and order one weak error estimates~\cite{MR1214374,MR1091544,https://doi.org/10.48550/arxiv.2112.15102} of suitable explicit numerical approximations to~\eqref{sde0}. 
These results are obtained by first proving moment bounds of derivatives of~$X_t$ with respect to initial condition. Subsequently, the estimates are used to validate an It\^{o}-Alekseev-Gr\"{o}bner formula~\cite{huddeiag} and differentiability of semigroups associated with~\eqref{sde0}, which are then used to prove the announced results. 
Similar moment bounds on the first and second derivative with respect to initial value in the non-globally Lipschitz case have recently appeared in \cite{hut1} under different assumptions. 
Related ideas for non-globally Lipschitz coefficients 
had appeared earlier in~\cite{https://doi.org/10.48550/arxiv.1309.5595,MR4079431,MR1305784} for obtaining local Lipschitz continuity in initial value, strong numerical convergence rates and strong \mbox{($p$-)completeness}. 

More specifically, in this paper we will show that the aforementioned results hold true under conditions where~$b,\sigma$ do not necessarily satisfy the globally monotonicity assumption~\cite[equation~(3)]{MR4079431}.
Our main assumptions are that higher derivatives of $b,\sigma$ are bounded by Lyapunov 
functions and loosely that~$b$ and~$\sigma$ admit local Lipschitz constants which are~$o(\log V)$ and~$o(\sqrt{\log V})$ respectively for a Lyapunov 
function $V$. 
The results are applicable to all of the example SDEs presented in~\cite[Section~4]{MR3766391}, with the exception of those in Section~4.7. In particular, for the first time, weak numerical convergence rates of order one are shown for 
these example SDEs. 
The convergence rates are obtained using the stopped increment-tamed Euler-Maruyama schemes of~\cite{MR3766391}. 

In contribution to regularity analysis of SDEs, the results provide new criteria for regularity of semigroups associated to solutions of~\eqref{sde0}. Previously, this regularity was known in cases of globally Lipschitz~\cite{MR1731794} or monotone coefficients~\cite{MR1840644}, or hypoellipticity~\cite[Proposition~4.18]{MR3305998}. On the other hand, Hairer et al.~\cite{MR3305998} presented remarkable counterexample SDEs, which do not have such regularity properties, even when the SDE has globally bounded smooth coefficients. 
More concretely, one counterexample is given by~\eqref{sde0} with
\begin{equation}\label{counter}
n=3,\quad b(t,x)=(\cos(x_3\cdot\exp(x_2^3)),0,0),\quad\sigma_{2,1} = \sqrt{2},\sigma_{i,j}=0\ \forall (i,j)\neq (2,1).
\end{equation}
Denoting~$X_t^x$ to be the unique (up to indistinguishability) solution to this SDE with~$X_0^x=x$, Theorem~3.1 in~\cite{MR3305998} asserts that there exists an infinitely differentiable and compactly supported~$\varphi:\mathbb{R}^3\rightarrow\mathbb{R}$ such that~$\mathbb{R}^3\ni x\mapsto\mathbb{E}[\varphi(X_t^x)]$ fails to even be locally H\"older continuous for any~$t>0$. The counterexamples stand in contrast to more classical results in the globally Lipschitz/monotone case as referenced above. Further counterexamples have also been recently established in~\cite{1531-3492_2021203}.

In the following Theorem~\ref{basicdiff}, 
we do not assume that the coefficients~$b$ and~$\sigma$ are globally bounded. 
Note however that, as announced, the coefficients are assumed to satisfy local Lipschitz bounds. 
Our basic result about semigroup differentiability can be summarized as in Theorem~\ref{basicdiff} below. 
\begin{theorem}\label{basicdiff}
Suppose there exists~$V:[0,T]\times\mathbb{R}^n\rightarrow \mathbb{R}$ twice continuously differentiable in space, continuously differentiable in time and constant~$C>0$ such that
\begin{equation}\label{basiclyacond}
\partial_tV(t,x) + \sum_{i=1}^n b_i(t,x)\partial_{x_i}V(t,x) + \frac{1}{2}\sum_{i,j = 1}^n (\sigma(t,x)\sigma(t,x)^{\top})_{ij}\partial_{x_i}\partial_{x_j} V(t,x) \leq CV(t,x)
\end{equation}
for all~$t\in[0,T]$,~$x\in\mathbb{R}^n$ and~$\lim_{\abs{x}\rightarrow\infty}V(t,x) = \infty$. Let~$f,c:[0,T]\times\mathbb{R}^n\rightarrow\mathbb{R}$,~$g:\mathbb{R}^n\rightarrow\mathbb{R}$ be measurable functions and~$p\in\mathbb{N}_0$. 
Suppose~$b(t,\cdot),\sigma(t,\cdot),f(t,\cdot),g,c(t,\cdot)\in C^p$. Moreover, suppose
\begin{itemize}
\item there exists measurable~$G:[0,T]\times\mathbb{R}^n\rightarrow\mathbb{R}$ such that~$G(t,\cdot):\mathbb{R}^d\rightarrow\mathbb{R}$ is continuous for any~$t$,~$G(t,\cdot) = o(\log V(t,\cdot))$ uniformly in~$t$ and such that
\begin{subequations}\label{basiclip}
\begin{align}
\abs{b(t,x) - b(t,y)} \leq (G(t,x) + G(t,y))\abs{x-y},\label{basiclip1}\\
\|\sigma(t,x) - \sigma(t,y)\|^2 \leq (G(t,x) + G(t,y))\abs{x-y}^2\label{basiclip2}
\end{align}
\end{subequations}
for all~$t\in[0,T]$,~$x\in\mathbb{R}^n$,
\item for every~$k>0$,~$h\in\{b,f,g,c\}$, there exists~$C'>0$ such that
\begin{equation}\label{basichda}
\abs*{\partial^{\alpha}h(t,\lambda x+(1-\lambda)y)} + \|\partial^{\beta}\sigma(t,\lambda x + (1-\lambda)y)\|^2  \leq C'(1+ V(t,x) + V(t,y))^{\frac{1}{k}}
\end{equation}
for all~$t\in[0,T]$,~$x,y\in\mathbb{R}^n$,~$\lambda\in[0,1]$ and multiindices~$\alpha,\beta$ with~$p_0\leq \abs{\alpha}\leq p$,~$2\leq \abs{\beta}\leq p$, where~$p_0 = 2$ if~$h = b$ and~$p_0 = 0$ otherwise.
\end{itemize}
For any~$s\in[0,T]$ and stopping time~$\tau\leq T-s$, the expectation of~$u(s,\tau,\cdot):\mathbb{R}^n\rightarrow\mathbb{R}$ given by
\begin{align}
&\mathbb{E}[u(s,\tau,x)]\nonumber\\
&\quad = \mathbb{E}\bigg[\int_0^{\tau} f(s+r,X_r^{s,x}) e^{-\int_0^r c(s+w,X_w^{s,x}) dw} dr + g(X_{\tau}^{s,x})e^{-\int_0^{\tau} c(s+w,X_w^{s,x})dw}\bigg],\label{basicu}
\end{align}
is continuously differentiable in~$x$ up to order~$p$, where for any~$s\in[0,T]$,~$x\in\mathbb{R}^n$,~$X_{\cdot}^{s,x}$ is the solution to~$X_t^{s,x} = x + \int_0^t b(s+r,X_v^{s,x}) dr + \int_0^t \sigma(s+r,X_r^{s,x})dW_r$ on~$t\in[0,T-s]$. Moreover, if~$p\geq 2$, then the function given by~$v(t,x) = \mathbb{E}[u(t,T-t,x)]$ is locally Lipschitz in~$t$ 
and satisfies
\begin{equation}\label{asfkeq0}
\partial_t v + a:D^2v + b\cdot \nabla v - cv + f = 0
\end{equation}
a.e. on $[0,T]\times\mathbb{R}^n$, where~$a=\frac{1}{2}\sigma\sigma^{\top}$,~$D^2$ denotes the Hessian matrix and~$a:D^2 = \textrm{Tr}(aD^2)$. If in addition~$b,\sigma$ are independent of~$t$ and~$f,c$ are jointly continuous in~$t,x$, then~$v$ is continuously differentiable in~$t$ and satisfies~\eqref{asfkeq0} on all of~$[0,T]\times\mathbb{R}^n$.
\end{theorem}


Theorem~\ref{basicdiff} follows as corollary to Theorems~\ref{fk2},~\ref{fk} and~\ref{classical}. 
For simplicity, the Lyapunov function~$V$ in Theorem~\ref{basicdiff} has been made independent of~$k$ appearing in~\eqref{basichda}. In the more detailed Theorems~\ref{fk2},~\ref{fk} and~\ref{classical}, this assumption is relaxed so that~$V$ may depend on~$k$. In particular, this allows us to easily determine that indeed we have a generalization of the globally Lipschitz case as in~\cite[Section~5.3]{MR1731794}: for example if~$n=1$ (higher dimensions following similarly), take~$V=V_k(x) = x^{2km}$ for some large enough~$m\in\mathbb{N}$. Since any globally Lipschitz~$b,\sigma$ are at most linearly growing, the Lyapunov property is readily verified. The rest of the conditions are then not stronger than those in~\cite{MR1731794}. A discussion of Theorem~\ref{basicdiff} with regard to the results of~\cite{MR3305998} is given in Section~\ref{introsub}. 

In addition, the case that~$b,\sigma$ are assumed to be time homogeneous with locally Lipschitz derivatives up to order~$p$ satisfying
\begin{align*}
\sum_i\abs{\partial_i b(x)} + \|\partial_i \sigma(x) \|^2 &\leq G(x),\\
\abs{\partial^{\alpha}h(x)} + \|\partial^{\alpha}\sigma (x)\|^2 &\leq C'(1+V(t,x))^{\frac{1}{k}}
\end{align*}
in place of~\eqref{basiclip1},~\eqref{basiclip2} and~\eqref{basichda} is considered in Section~\ref{altassumps}. In particular, Theorems~\ref{alters2},~\ref{alters3},~\ref{alters4} show that the conclusions of Theorem~\ref{basicdiff} hold 
under this setting. 
These results appear to be the only ones in the literature about twice differentiable-in-space solutions to Kolmogorov equations outside hypoelliptic, elliptic diffusion or globally Lipschitz/monotone settings.

In the same vein as the counterexamples for regularity preservation, the authors of~\cite{MR3305998} present a counterexample SDE where the Euler-Maruyama approximation suffers from the lack of polynomial convergence rates. 
Namely, Theorem~5.1 in~\cite{MR3305998} shows that there exists a globally bounded smooth pair~$b,\sigma$ such that~$\lim_{\delta\rightarrow0^+}\abs{\mathbb{E}[X_t] - \mathbb{E}[Y_t^{\delta}]}/\delta^{\alpha} = \infty$ for any~$\alpha>0$, where~$X_t$ denotes the solution to~\eqref{sde0} with~$X_0=0$ and~$Y_t^{\delta}$ denotes its Euler-Maruyama approximation with stepsize~$\delta$. 
The next result provides general conditions where numerically weak convergence rates of order~$1$ may be established outside the classical globally Lipschitz~\cite{MR1214374} and monotone~\cite{https://doi.org/10.48550/arxiv.2112.15102,MR4689832} cases. 

\begin{theorem}\label{wbth}
Let all of the assumptions in Theorem~\ref{basicdiff} hold with~$p\geq 3$. Suppose~$b,\sigma$ are independent of~$t$ and suppose~$V$ is of the form~$V(t,x) = e^{U(x)e^{-\rho t}}$ for~$U\in C^3(\mathbb{R}^n,[0,\infty))$,~$\rho\geq 0$, such that there exist~$c\geq 1$ satisfying
\begin{equation*}
\abs{x}^{\frac{1}{c}} + 
\abs{\partial^{\alpha}b(x)}^{\frac{1}{c}} + \|\partial^{\alpha}\sigma(x)\|^{\frac{1}{c}} + \abs{\partial^{\beta}U(x)} \leq c(1+U(x))^{1 - \frac{1}{c}},
\end{equation*}
for all~$x\in\mathbb{R}^n$, multiindices~$\alpha,\beta$ with~$0\leq \abs{\alpha}\leq 2$ and~$1\leq \abs{\beta}\leq 3$. If~$h\in C^3(\mathbb{R}^n,\mathbb{R})$ is such that
\begin{equation*}
\abs{\partial^{\alpha}h(x)} \leq c(1+\abs{x}^c)
\end{equation*}
for all~$x\in\mathbb{R}^n$ and multiindices~$\alpha$ with~$0\leq \abs{\alpha}\leq 3$, then there exists a constant~$C>0$ such that
\begin{equation}\label{weaklhs}
\abs{\mathbb{E}[h(X_T)] - \mathbb{E}[h(Y_T^{\delta})]}\leq C \delta,
\end{equation}
for all~$0<\delta<1$, where~$Y_{\cdot}^{\delta}:[0,T]\rightarrow\mathbb{R}^n$ is the approximation given by~$Y_0^{\delta} = X_0\in\mathbb{R}^n$ and
\begin{equation}\label{btapprox}
Y_t^{\delta} = Y_{k\delta}^{\delta} + \mathds{1}_{\{ \abs{Y_{k\delta}^{\delta}}\leq \exp(\abs{\log \delta}^{\frac{1}{2}}) \}} \bigg( \frac{b(Y_{k\delta}^{\delta})(t-k\delta) + \sigma(Y_{k\delta}^{\delta})(W_t-W_{k\delta})}{1+\abs{b(Y_{k\delta}^{\delta})(t-k\delta) + \sigma(Y_{k\delta}^{\delta})(W_t-W_{k\delta})}^3} \bigg)
\end{equation}
for all~$t\in[k\delta,(k+1)\delta]$,~$k\in\mathbb{N}_0\cap[0,\frac{T}{\delta})$.
\end{theorem}

Theorem~\ref{wbth} is corollary to Theorem~\ref{weakconvlya}, for which the full setting is given by Assumption~\ref{A4lya}, with comments in Remark~\ref{before51}. 
The numerical scheme~\eqref{btapprox} is the stopped increment-tamed Euler-Maruyama approximation from~\cite{MR3766391}. It has the key property of retaining exponential integrability properties of the continuous time SDE, which is used throughout the proof for Theorem~\ref{weakconvlya}. As is well documented~\cite{MR2795791}, the classical Euler-Maruyama scheme may diverge in both the strong and weak sense for superlinearly growing, non-globally Lipschitz coefficients without this property. The power~$3$ appearing in the denominator on the right-hand side of~\eqref{btapprox} is chosen purposefully: weak convergence rates of order one are only obtained for exponents larger than or equal to~$3$. 
The proof of Theorem~\ref{weakconvlya} uses the recently established It\^{o}-Alekseev-Gr\"{o}bner formula~\cite{huddeiag} in order to expand the left-hand side of~\eqref{weaklhs}, instead of the classical approach using~$C^{1,2}$ solutions to Kolmogorov equations as in~\cite{MR1214374}. 
Note that as a result, the requirement~$p\geq 3$ in Theorem~\ref{wbth} is slightly weaker than the typical fourth order continuous differentiability of~$b,\sigma$. 
In order to apply the formula, strong completeness of some derivative processes of~\eqref{sde0} is established first by using a result in~\cite{https://doi.org/10.48550/arxiv.1309.5595}. Some closely related properties for~\eqref{sde0} appeared recently in~\cite{hut1}, where the authors use a different approach and different assumptions. Although weak convergence without rates has been established by way of convergence in probability in~\cite[Corollary~3.7]{MR3766391} and~\cite[Corollary~3.19]{MR3364862}, weak rates of convergence 
(of order one) have thus far been an open problem for general non-globally monotone coefficients, see however for example~\cite{MR4177372,MR2177799} in this direction. 
On the other hand, strong convergence rates of order~$\frac{1}{2}$ have been established in even the non-globally monotone case~\cite{MR4079431}. The assumptions of Theorem~\ref{wbth} (and of the more detailed Theorem~\ref{weakconvlya}) do not include the globally Lipschitz setting as in~\cite[Theorem~14.5.1]{MR1214374}. However, some weakening of these assumptions that both includes the globally Lipschitz setting and is sufficient for the conclusions of Theorem~\ref{weakconvlya} to hold is discussed in Remark~\ref{remprof}. 

The proofs for the moment estimates underlying both Theorems~\ref{basicdiff} and~\ref{wbth} use directly the results of~\cite{MR4260476}, for which exponential integrability in continuous time as in~\cite{https://doi.org/10.48550/arxiv.1309.5595,MR4079431} is an important property that is accounted for in a crucial way by our local Lipschitz condition. 
The core argument for these estimates, which can be thought of as a combination of the approach in~\cite{MR1731794} with ideas of~\cite{https://doi.org/10.48550/arxiv.1309.5595,MR4079431}, is to consider for any $\kappa\in\mathbb{R}^n$ 
processes $X_{t(\kappa)}$ satisfying
\begin{equation*}
\sup_{t\in[0,T]}\abs*{\frac{X_t^{x+r\kappa}-X_t^x}{r} - X_{t(\kappa)}^x} \rightarrow 0
\end{equation*}
in probability as $r\rightarrow 0$, where $X_t^x$ denotes a solution to~\eqref{sde0} with $X_0^x = x$. Such processes exist \cite[Theorem 4.10]{MR1731794} for $b$, $\sigma$ continuously differentiable in space satisfying some local integrability assumption and $X_{t(\kappa)}^x$ satisfies the system resulting from a formal differentiation of~\eqref{sde0} (see~\eqref{firsdd}). If~$b$ and~$\sigma$ are independent of~$t$ and the derivatives of~$b$ and~$\sigma$ are locally Lipschitz, the processes $X_{t(\kappa)}^x$ are almost surely continuous derivatives in the classical sense as in \cite[Theorem~V.39]{MR2273672}. Higher derivatives exist for $b$ and $\sigma$ with higher orders of differentiability. The SDEs solved by the first order derivatives turn out to be just as considered for previous applications of the stochastic Gr\"{o}nwall inequality~\cite{MR4260476}, whereas those for higher order derivatives have only the term involving the derivative of the highest order on the right-hand side of the dynamics requiring serious control. For the latter, the stochastic Gr\"{o}nwall inequality together with 
our Assumption~\ref{A1lya} below and an induction argument are sufficient to control all of the terms. 
We use~$o(\log V)$ and~$o(\sqrt{\log V})$ Lipschitz constants in order to control the moments for large time~$T$, but the results follow for~$O(\log V)$ and~$O(\sqrt{\log V})$ Lipschitz constants if~$T$ is suitably small. 
In order to establish solutions to the Kolmogorov equation, a number of intermediary results following the strategy of~\cite{MR1731794} are given for the present case of local Lipschitz constants. 
In particular, it is shown by extending an argument from~\cite{MR4074703} that an Euler-type approximation converges to solutions of the SDE in probability and locally uniformly in initial time and space, that is, the SDE is regular~\cite[Definition~2.1]{MR1731794}.

The original motivation for this work is the Poisson equation for finding the asymptotic variance of ergodic averages associated to SDEs. In \cite{optfric}, a formula for the derivative of this variance with respect to a parameter in the dynamics is derived. In order to do so, the Poisson equation is interpreted as a PDE in the classical sense, which in turn made use of an appropriate solution to the Kolmogorov equation. 
In a setting where the coefficients are not globally Lipschitz, for example if the friction in the Langevin equation of \cite{optfric} is not restricted to be constant in space, the existence of such a solution to the backward Kolmogorov equation appears to be unavailable in the literature. The present work fills this gap. In addition, solutions to the Poisson equation furnishes central limit theorems for additive functionals themselves by way of~\cite{MR3069369}. The results here allow some arguments there to be established rigorously for hypoelliptic diffusions, more details are given in Section~\ref{langwithvarlya}.

\subsection{Loss of regularity}\label{introsub}
To conclude the introduction, let us discuss Theorem~\ref{basicdiff} in the context of~\cite[Theorems~3.1,~4.16, Proposition~4.18]{MR3305998}.

Theorem~4.16 in~\cite{MR3305998} asserts the existence of unique viscosity solutions to Kolmogorov equations 
given existence of an associated Lyapunov function~$V$, that is,~$V$ satisfying~\eqref{basiclyacond}. In that statement, it is assumed that~$c=f=0$ and that~$b,\sigma$ are time-homogeneous. 
Otherwise, their assumptions are strictly weaker than those in Theorems~\ref{basicdiff} and~\ref{wbth}. 
This viscosity solution has the representation~$(t,x)\mapsto\mathbb{E}[g(X_{T-t}^x)]$, where~$X_t^x$ denotes the solution to~\eqref{sde0} with~$X_0^x=x$, 
but it is in general not differentiable in contrast to in Theorem~\ref{basicdiff}. However, given enough regularity, it 
is an almost everywhere solution. 
In particular, this is the case if it belongs to the Sobolev space~$W_{\textrm{loc}}^{2,1,p}$ for some~$p>n+1$, see~\cite[Proposition~I.4, Remark~I.16]{MR709162}. 
Under the stronger assumptions here, our results on~\eqref{basicu} and its a.e. derivatives as implied by Theorem~\ref{fk2} verifies that this viscosity solution indeed belongs to~$W_{\textrm{loc}}^{2,1,p}$. 
These arguments form an alternate proof for the assertion about a.e. solutions to~\eqref{asfkeq0} in Theorem~\ref{basicdiff} in the case where~$f=c=0$ and~$b,\sigma$ are time-homogeneous.

In Proposition~4.18 in~\cite{MR3305998}, again in the setting where~$f=c=0$ and~$b,\sigma$ are time-homogeneous, the authors make use of Lemma~5.12 in~\cite{MR1731794} to obtain distributional solutions to the Kolmogorov equation~\eqref{asfkeq0} in the case of smooth coefficients. 
If in addition~$b,\sigma$ satisfy H\"{o}rmander's condition, their result implies for continuous bounded~$g$ that~$(t,x)\mapsto\mathbb{E}[g(X_{T-t}^x)]$ is a classical solution to the Kolmogorov equation.
In particular, there is a preservation, or even a gain, of regularity in the semigroup. H\"{o}rmander's condition appears to be neither strictly stronger nor strictly weaker than the main assumptions in the present work. For example in Section~\ref{stochdufflya}, we consider~\eqref{sde0} with~$n=2$,~$b(t,x)=(x_2,\alpha_1 x_1-\alpha_2 x_2 - \alpha_3 x_2x_1^2 - x_1^3)$ and~$(\sigma(t,x))_{1,1}=(\sigma(t,x))_{1,2}=0$,~$(\sigma(t,x))_{2,1} = \beta_1x_1$,~$(\sigma(t,x))_{2,2}=\beta_3$ for some constants~$\alpha_1,\alpha_2,\beta_1,\beta_3\in\mathbb{R}$,~$\alpha_3>0$. This SDE does not satisfy H\"{o}rmander's condition when~$\beta_3=0$ (which is the case studied in~\cite{MR1430980}). However, as demonstrated in Section~\ref{stochdufflya}, it does satisfy the main assumptions in the present work. On the other hand, for example in the case where~$\alpha_3=0$ and~$\beta_3\neq 0$, H\"{o}rmander's condition is satisfied, but it is not clear whether there exists a satisfactory Lyapunov function.

Lastly, Theorem~3.1 in~\cite{MR3305998} presents an instance of~\eqref{sde0} such that there exists smooth compactly supported~$\varphi$ satisfying that for any~$t\in(0,T]$, the function~$x\mapsto\mathbb{E}[\varphi(X_t^x)]$ is not locally H\"{o}lder continuous. In combination with Theorem~\ref{basicdiff}, this implies that, for~$b,\sigma$ as in~\cite[equation~(3.1)]{MR3305998}, it is impossible to find a Lyapunov function~$V$ such that~$b,\sigma$ have Lipschitz constants that are~$o(\log V)$ and~$o(\sqrt{\log V})$ respectively.

The paper is organised as follows. In Section~\ref{notandpremlya}, the setting, notation and various definitions 
are given. In Section~\ref{momentestimates}, moment estimates of the supremum over time on the derivative process and the difference processes in initial value are given. These results are used throughout for proving the other results in the paper. In Section~\ref{kollya}, results on the regularity of the semigroup associated to~\eqref{sde0} are presented, which are followed by results about twice differentiable-in-space solutions to the Kolmogorov equation. Section~\ref{weakrates} contains the results about weak convergence rates for the stopped increment-tamed Euler-Maruyama scheme on SDEs with non-globally monotone coefficients. In Section~\ref{lyaexamples}, new Lyapunov functions are given for the Langevin equation with variable friction and stochastic Duffing-van der Pol equation. In the case of the former, the associated Poisson equation is discussed.

\section{Notation and preliminaries}\label{notandpremlya}
Let $(\Omega,\mathcal{F},\mathbb{P})$ be a complete probability space,~$\mathcal{F}_t$, $t\in [0,\infty)$, be a filtration satisfying the usual conditions (see e.g.~\cite[p.3]{MR2273672}) and~$(W_t)_{t\geq 0}$ be a standard Wiener process on~$\mathbb{R}^n$ with respect to~$\mathcal{F}_t$,~$t\in[0,\infty)$. Unless otherwise stated, let~$T\in(0,\infty)$. Let~$\abs{v},\|M\|$ denote the Euclidean norm of a vector~$v$ and the Frobenius norm of a matrix~$M$ respectively. Let~$b:\Omega\times[0,\infty)\times\mathbb{R}^n\rightarrow\mathbb{R}^n$, $\sigma:\Omega\times[0,\infty)\times\mathbb{R}^n\rightarrow\mathbb{R}^{n\times n}$ be functions such 
that~$b(t,\cdot),\sigma(t,\cdot)$ are continuous for every\footnote{The requirement that the properties hold for every~$\omega\in\Omega$, which will appear throughout the paper, is consistent with the assumptions in~\cite{MR1731794}, so that we may reference results directly from~\cite{MR1731794}. As is common practice, we omit in the notation the dependence on~$\omega$ for functions of~$\omega$.}~$t$,~$\omega$,~$b(\cdot,x),\sigma(\cdot,x)$ are~$\mathcal{F}\otimes \mathcal{B}([0,\infty))$-measurable for every~$x$,~$b(t,x),\sigma(t,x)$ are~$\mathcal{F}_t$-measurable for every~$t,x$ and~$\int_0^T\sup_{\abs*{x}\leq R}(\abs*{b(t,x)} + \|\sigma(t,x)\|^2)dt < \infty$ for any~$R>0$,~$\omega\in\Omega$. Let~$O\subseteq\mathbb{R}^n$ be an open set and for any~$x\in O $,~$s\geq0$, let~$X_t^{s,x}$ be an~$\mathcal{F}_t$-adapted~$O$-valued process 
such that~$X_t^{s,x}$ is~$\mathbb{P}$-a.s. continuous satisfying 
for all~$t\in [0,T]$ that
\begin{equation}\label{sde}
X_t^{s,x} = x  +\int_0^t b(s+r,X_r^{s,x})dr + \int_0^t\sigma(s+r,X_r^{s,x})dW_r.
\end{equation}
Note for spatially locally Lipschitz~$b,\sigma$, the existence of a Lyapunov function (that is, a function satisfying~\eqref{basiclyacond}) that grows to infinity at infinity suffices for the existence of the processes~$X_t^{s,x}$. More precisely, for example for the SDEs in Section~\ref{lyaexamples}, Theorem~3.5 in~\cite{MR2894052} proves that the processes~$X_t^{s,x}$ exist. See also Theorem~1.2 in~\cite{MR1731794} for general conditions on~$b,\sigma$ for the existence of~$X_t^{s,x}$.
When the initial value~$x$ and time~$s$ are not important or are obvious from the context, simply $X_t$ and similarly~$X_t^x$ 
is written. 
For $f\in C^2( O)$ and for either~$b$,~$\sigma$ as above or~$(b_{\cdot}^x:\Omega\times[0,T]\rightarrow\mathbb{R}^n)_{x\in O}$,~$(\sigma_{\cdot}^x:\Omega\times[0,T] \rightarrow\mathbb{R}^{n\times n})_{x\in O}$ that are, for each~$x$,~$\mathcal{F}\otimes\mathcal{B}([0,t])$-measurable and~$\mathcal{F}_t$-adapted satisfying~$\mathbb{P}$-a.s. that~$\int_0^T(\abs{b_s^x} + \|\sigma_s^x\|^2)ds < \infty$, we denote
\begin{equation}\label{jen}
Lf = b\cdot \nabla f + a : D^2 f,
\end{equation}
where $a = \frac{1}{2}\sigma\sigma^\top$, $D^2$ denotes the Hessian and for matrices $M,N$, $M:N = \sum_{i,j}M_{ij}N_{ij}$. 
Throughout, $\hat{O}$ is used to denote the convex hull of $O$, 
$C_c^\infty((0,T)\times\mathbb{R}^n)$ denotes the set of compactly supported infinitely differentiable functions on~$(0,T)\times\mathbb{R}^n$,~$C_b(\mathbb{R}^n)$ denotes the set of bounded continuous function on~$\mathbb{R}^n$,~$C^{1,2}([0,T]\times\mathbb{R}^n)$ denotes the set of continuous functions of the form $[0,T]\times\mathbb{R}^n\ni (t,x) \mapsto f(t,x)$ that are once continuously differentiable in $t$ and twice so in $x$, $B_R(x)$ denotes the closed ball of radius $R>0$ around $x\in\mathbb{R}^n$, $B_R = B_R(0)$,~$e_i$ denotes the~$i^{\textrm{th}}$ Euclidean basis vector in~$\mathbb{R}^n$, 
and~$C>0$ denotes a generic constant that may change from line to line. 
The expression~$\mathds{1}_A$ denotes the indicator function on the set~$A$. We denote~$\Delta_T = \{(s,t):0\leq s\leq t\leq T\}$. 
The notation~$\partial_i Z_{t,T}^z = \partial_{z_i}Z_{t,T}^{\cdot}\vert_z$ 
is used and similarly for the higher order derivatives~$\partial^{\alpha} Z_{t,T}^z$ 
for multiindices~$\alpha$. Moreover, for a multiindex~$\alpha$, 
we denote~$\abs{\alpha}=\sum_i\alpha_i$ and
\begin{equation*}
\kappa_{\alpha} = (\overbrace{e_1,\dots,e_1}^{\alpha_1\textrm{ times}},e_2,\dots).
\end{equation*}
\begin{definition}\label{rlyadef}
A positive random function $V:\Omega\times [0,T]\times O\rightarrow (0,\infty)$ is referred to as a $(\tilde{b}_{\cdot}^{\cdot},\tilde{\sigma}_{\cdot}^{\cdot},\alpha_{\cdot},\beta_{\cdot},p^*, V_0)$-Lyapunov function 
if~$(\tilde{b}_{\cdot}^y:\Omega\times[0,T] \rightarrow\mathbb{R}^n)_{y\in O}$,~$(\tilde{\sigma}_{\cdot}^y:\Omega\times[0,T] \rightarrow\mathbb{R}^{n\times n})_{y\in O}$,~$\alpha_{\cdot},\beta_{\cdot}:\Omega\times [0,T]\rightarrow [0,\infty]$,
~$p^*\in [1,\infty)$ and~$V_0\in C^{1,2}([0,T]\times O)$ are~$\mathcal{F}\otimes\mathcal{B}([0,T])$-measurable and~$\mathcal{F}_t$-adapted processes where applicable and satisfy for all~$y\in O$ that there exist a~$\mathcal{F}\otimes\mathcal{B}([0,T])$-measurable,~$\mathcal{F}_t$-adapted process~$Y_{\cdot}^y:\Omega\times[0,T]\rightarrow O$ such that it is~$\mathbb{P}$-a.s. continuous, it holds~$\mathbb{P}$-a.s. that~$V(t,y) = V_0(t,Y_t^y)$ for all~$t\in[0,T]$ and for any stopping time~$\tau\leq T$, it holds~$\mathbb{P}$-a.s. that
\begin{align}
&\int_0^T(\abs{\tilde{b}_r^y} + \|\tilde{\sigma}_r^y\|^2 + \abs{\alpha_r})dr < \infty,\label{dummypre}\\
&Y_s^y = y + \int_0^s \tilde{b}_r^y dr + \int_0^s \tilde{\sigma}_r^y dW_r,\label{dummysde}\\
&(\partial_t + L)V_0(s,Y_s^y) + \frac{p^*-1}{2}\frac{\abs{(\tilde{\sigma}_s^y)^{\top} \nabla V_0(s,Y_s^y)}^2}{V_0(s,Y_s^y)} \leq \alpha_s V_0(s,Y_t^y) + \beta_s\label{lyapr2}
\end{align}
for 
all $s\in[0,T]$, 
where~$L$ is given by~\eqref{jen} with~$\tilde{b},\tilde{\sigma}$ replacing~$b,\sigma$.
\end{definition}
\begin{definition}\label{lyafunctions}
For~$\bar{T}\in(0,\infty)$,~$\tilde{n}\in\mathbb{N}$ and open~$\tilde{O}\subseteq\mathbb{R}^{\tilde{n}}$, a function $V:\Omega\times[0,\bar{T}]\times\tilde{O}\rightarrow(0,\infty)$ is referred to as a Lyapunov function if there exist a filtration and Wiener process as above, $p^*\in[1,\infty)$,~$\tilde{b}_{\cdot}^{\cdot}:\Omega\times[0,\bar{T}]\times \tilde{O}\rightarrow\mathbb{R}^{\tilde{n}}$, $\tilde{\sigma}_{\cdot}^{\cdot}:\Omega\times[0,\bar{T}]\times \tilde{O}\rightarrow\mathbb{R}^{\tilde{n}\times \tilde{n}}$,~$V_0\in C^{1,2}([0,\bar{T}]\times \tilde{O})$, 
along with some~$\alpha_{\cdot}$ and~$\beta_{\cdot}$ such that~$V$ is a~$(\tilde{b}_{\cdot}^{\cdot},\tilde{\sigma}_{\cdot}^{\cdot},\alpha_{\cdot}, \beta_{\cdot},p^*, V_0)$-Lyapunov function and 
\begin{equation}\label{lyap3}
\left\| e^{\int_0^{\bar{T}} \abs{\alpha_u} du} \right\|_{L^{\frac{p^*}{p^*-1}}(\mathbb{P})} dt + \int_0^{\bar{T}} \left\|\frac{\beta_v}{e^{\int_0^v \alpha_u du }}\right\|_{L^{p^*}(\mathbb{P})} dv dt < \infty.
\end{equation}
\end{definition}
\begin{remark}
\begin{enumerate}[label=(\roman*)]
\item Smooth functions~$V_0$ satisfying~$(\partial_t + L)V_0\leq CV_0$ for some constant~$C$ as in \cite[Theorem~3.5]{MR2894052} are Lyapunov functions with~$p^*=1$,~$\alpha_t = C$ and~$\beta_t=0$. 
In this case, note that~\eqref{dummysde} holds~$\mathbb{P}$-a.s. with~$Y_{\cdot}^{\cdot},\tilde{b}_{\cdot}^{\cdot},\tilde{\sigma}_{\cdot}^{\cdot}$ replaced for example by~$X_{\cdot}^{0,\cdot},b(\cdot,X_{\cdot}^{0,\cdot}),\sigma(\cdot,X_{\cdot}^{0,\cdot})$ respectively, which follows by Lemma~4.51 in~\cite{MR2378491}, our assumptions on~$\sigma$ and the fact that~$(\omega,t)\mapsto(\omega,t,X_t^{0,x}(\omega))$ is~$\mathcal{F}\otimes\mathcal{B}([0,T])$-measurable and~$\mathcal{F}_t$-adapted (with~$(\mathcal{F}\otimes\mathcal{B}([0,T]))\otimes \mathcal{B}(O)$,~$\mathcal{F}_t\otimes\mathcal{B}(O)$ as~$\sigma$-algebras in the respective ranges). 
\item To summarize loosely, Lyapunov functions as defined above satisfy firstly the main condition~(15) in~\cite{MR4260476} for the stochastic Gr\"{o}nwall inequality and secondly finiteness conditions on the associated processes. These are properties that will be used many times throughout the paper in the form of Proposition~\ref{huddethm} and its corollaries below. 
\end{enumerate}
\end{remark}
The following property allows control across families of Lyapunov functions.
\begin{definition}\label{locdef}
Let~$\bar{T}\in(0,\infty)$,~$\tilde{n}\in\mathbb{N}$,~$(\tilde{n}_s)_{s\in[0,T]}$,~$\tilde{O}$,~$(\tilde{O}_s)_{s\in[0,T]}$,~$V_0\in C^{1,2}([0,\infty)\times\tilde{O})$ be such that~$\tilde{O}\subseteq\mathbb{R}^{\tilde{n}}$ and~$\tilde{O}_s\subseteq\mathbb{R}^{\tilde{n}_s}$ are all open. A family of functions~$(\hat{W}_s:\Omega\times[0,\bar{T}]\times\tilde{O}_s\rightarrow(0,\infty))_{s\in[0,T]}$ is~$(\tilde{n},\tilde{O},V_0)$-local in~$s$ if~$\tilde{O}=\tilde{O}_s$ and there exists a constant~$C>0$ satisfying that for any~$s\in[0,T]$, there exist~$\tilde{b}^{s,T}$,~$\tilde{\sigma}^{s,T},\alpha_{\cdot}^{s,T}$,~$\beta_{\cdot}^{s,T}$,~$p^{s,T}$ such that~$\hat{W}_s$ is a~$(\tilde{b}^{s,T},\tilde{\sigma}^{s,T},\alpha_{\cdot}^{s,T},\beta_{\cdot}^{s,T}, p^{s,T}, V_0(s+\cdot,\cdot)\vert_{[0,\bar{T}]\times\tilde{O}})$-Lyapunov function 
and the corresponding bound~\eqref{lyap3} holds uniformly with bound~$C$, that is,
\begin{equation}\label{localseq}
\Big\| e^{\int_0^{\bar{T}}\abs*{\alpha_u^{s,T}} du} \Big\|_{L^{\frac{p^{s,T}}{p^{s,T}-1}}(\mathbb{P})} + \int_0^{\bar{T}} \bigg\|\frac{\beta_v^{s,T}}{e^{\int_0^v \alpha_u^{s,T} du}}\bigg\|_{L^{p^{s,T}}(\mathbb{P})} dv < C.
\end{equation}
We say that~$(W_s)_{s\in[0,T]}$ is local in~$s$ if there exist~$\tilde{n},\tilde{O},V_0$ such that~$(W_s)_{s\in[0,T]}$ is~$(\tilde{n},\tilde{O},V_0)$-local in~$s$.
\end{definition}
A family of Lyapunov functions being local in~$s$ allows terms of the form~$\mathbb{E}[W_s(t,X_t^{s,x})]$ to be bounded uniformly in~$s$ after applying the stochastic Gr\"{o}nwall inequality (stated as Proposition~\ref{huddethm} below).
This is an important property for twice differentiable solutions to Kolmogorov equations, since such solutions and many lemmatic terms depend on a time variable via the starting times~$s$. On the other hand, such a property is in all of the examples mentioned here easily satisfied.

In the rest of the section, some results from~\cite{MR4260476,huddeiag} 
are recalled for the convenience of the reader. With the exception of Corollary~\ref{huddecor0}, we refer to the corresponding statements in~\cite{MR4260476,huddeiag} for their proofs. 
The next Proposition~\ref{huddethm} is a special case of Theorem~2.4 in~\cite{MR4260476}.

\begin{prop}\label{huddethm}
Let~$\tau\leq T$ be a stopping time,~$p^*\in[1,\infty)$ and~$V_0\in C^{1,2}([0,T]\times O,[0,\infty))$. 
Moreover, let~$\hat{X}:\Omega\times[0,T]\rightarrow O$,~$\hat{b}:\Omega\times[0,T]\rightarrow \mathbb{R}^{n}$,~$\hat{\sigma}:\Omega\times[0,T]\rightarrow \mathbb{R}^{n\times n}$,~$\hat{\alpha}:\Omega\times[0,T]\rightarrow\mathbb{R}\cup\{-\infty,\infty\}$,~$\hat{\beta}:\Omega\times[0,T]\rightarrow\mathbb{R}\cup\{-\infty,\infty\}$ be~$\mathcal{F}\otimes\mathcal{B}([0,T])$-measurable and~$\mathcal{F}_t$-adapted processes such that~$\hat{X}$ has continuous sample paths and it holds~$\mathbb{P}$-a.s. that~$\int_0^{\tau}(\abs{\hat{b}_s}+\|\hat{\sigma}_s\|^2 + \abs{\hat{\alpha}_s})ds<\infty$,~$\hat{X}_{t\wedge\tau}=\hat{X}_0+\int_0^t\mathds{1}_{[0,\tau)}(s)\hat{b}_sds + \int_0^t\mathds{1}_{[0,\tau)}(s)\hat{\sigma}_sdW_s$ for all~$t\in[0,T]$ and it holds~$\mathbb{P}$-a.s. that for a.a.~$s\in[0,\tau)$,~\eqref{lyapr2} holds with~$LV_0= \hat{b}_{\cdot}\cdot\nabla V_0 + \frac{1}{2}(\hat{\sigma}_{\cdot}\hat{\sigma}_{\cdot}^{\top}):D^2V_0$ and~$Y_s^y,\tilde{\sigma}_s^y,\alpha_s,\beta_s$ replaced by~$\hat{X}_s,\hat{\sigma}_s,\hat{\alpha}_s,\hat{\beta}_s$ respectively. The following statements hold.
\begin{enumerate}[label=(\roman*)]
\item For~$q_1,q_2\in(0,\infty]$ satisfying~$\frac{1}{q_1} = \frac{1}{q_2} + \frac{1}{p^*}$, it holds that
\begin{align}
(\mathbb{E}[(V_0(\tau,\hat{X}_{\tau}))^{q_1}])^{\frac{1}{q_1}} &\leq \bigg(\mathbb{E}\bigg[\exp\bigg(q_2\int_0^{\tau}\hat{\alpha}_sds\bigg)\bigg]\bigg)^{\frac{1}{q_2}}\bigg((\mathbb{E}[(V_0(0,\hat{X}_0))^{p^*}])^{\frac{1}{p^*}}\nonumber\\
&\quad+ \int_0^T\bigg(\mathbb{E}\bigg[\bigg(\frac{\mathds{1}_{[0,\tau)}(s)\hat{\beta}_s}{\exp(\int_0^s\hat{\alpha}_rdr)}\bigg)^{p^*}\bigg]\bigg)^{\frac{1}{p^*}}ds\bigg).\label{hueq1}
\end{align}
\item\label{hp2} For~$q_1,q_2,q_3\in(0,\infty]$ satisfying~$q_3<p^*$ and~$\frac{1}{q_1}=\frac{1}{q_2} + \frac{1}{q_3}$, there exists a constant~$C>0$ depending only on~$q_3,p^*$ such that
\begin{align*}
&\bigg(\mathbb{E}\bigg[\bigg(\sup_{s\in[0,\tau]}V_0(s,\hat{X}_s)\bigg)^{q_1}\bigg]\bigg)^{\frac{1}{q_1}} \\
&\quad\leq 
C\bigg(\mathbb{E}\bigg[\exp\bigg(\sup_{t\in[0,\tau]}q_2\int_0^t\hat{\alpha}_sds\bigg)\bigg]\bigg)^{\frac{1}{q_2}}\cdot\bigg(\mathbb{E}\bigg[\bigg(V_0(0,\hat{X}_0)\\
&\qquad+\int_0^{\tau}\frac{\hat{\beta}_s}{\exp(\int_0^s\hat{\alpha}_rdr)}ds\bigg)^{q_3}\bigg]\bigg)^{\frac{1}{q_3}}.
\end{align*}
\end{enumerate}
\end{prop}
An application of Proposition~\ref{huddethm} on Lyapunov functions as defined above is given by the next Corollary~\ref{huddecor0}.
\begin{corollary}\label{huddecor0}
Let~$\tilde{n}\in\mathbb{N}$,~$\tilde{O}\in\mathbb{R}^{\tilde{n}}$,~$V:\Omega\times[0,T]\times\tilde{O}\rightarrow (0,\infty)$ be a~$(\tilde{b}_{\cdot}^{\cdot},\tilde{\sigma}_{\cdot}^{\cdot},\alpha_{\cdot},\beta_{\cdot},p^*,V_0)$-Lyapunov function for some~$\tilde{b}_{\cdot}^{\cdot},\tilde{\sigma}_{\cdot}^{\cdot},\alpha_{\cdot},\beta_{\cdot},p^*,V_0$. For any~$q_1,q_2\in(0,\infty]$ with~$\frac{1}{q_1} = \frac{1}{q_2} + \frac{1}{p^*}$, it holds that
\begin{equation}\label{hudcorv}
(\mathbb{E}[(V(t,y))^{q_1}])^{\frac{1}{q_1}} \leq C((\mathbb{E}[(V(0,y))^{p^*}])^{\frac{1}{p^*}} + 1)
\end{equation}
for all stopping times~$t\leq T$ and~$\tilde{O}$-valued~$\mathcal{F}_0$-measurable r.v.'s~$y$, where~$C$ is given by the maximum between the first factor and the last term in the last factor both on the right-hand side of~\eqref{hueq1} with~$\hat{\alpha},\hat{\beta},\tau$ replaced by~$\alpha,\beta,t$. In particular, the same statement holds with the right-hand side of~\eqref{hudcorv} replaced by~$C(\mathbb{E}[V(0,y)] + 1)$ for all deterministic~$y$.
\end{corollary}
\begin{proof}
By Definitions~\ref{rlyadef},~\ref{lyafunctions} and Proposition~\ref{huddethm}, it suffices to check that any~$\mathcal{F}\otimes\mathcal{B}([0,T])$-measurable,~$\mathcal{F}_t$-adapted,~$\mathbb{P}$-a.s. continuous process~$Y_{\cdot}^y$ satisfying~$\mathbb{P}$-a.s.~\eqref{dummypre},~\eqref{dummysde} and~$V(t,y)=V_0(t,Y_t^y)$ for all~$t\in[0,T]$ is such that for any~$t\in[0,T]$ and stopping time~$\tau\leq T$, it holds~$\mathbb{P}$-a.s. that~$Y_{t\wedge\tau}^y = y + \int_0^t\mathds{1}_{[0,\tau)}(s)\tilde{b}_s^yds + \int_0^t\mathds{1}_{[0,\tau)}(s)\tilde{\sigma}_s^ydW_s$. The only thing to check is that the stochastic integrals~$\int_0^{t\wedge\tau}\hat{\sigma}_r^ydW_r$ and~$\int_0^t\mathds{1}_{[0,\tau)}(r)\hat{\sigma}_r^ydW_r$ are equal~$\mathbb{P}$-almost surely. This may be verified by Proposition~2.10 and Remark~2.11 (see also the paragraph after Definition~2.23) all in~\cite{MR1121940}.
\end{proof}
In the particular case where~$V_0(t,x)=\abs{x}^2$, Proposition~\ref{huddethm} implies the following Corollary~\ref{huddecor1}, which is a special case of Corollary~2.5 in~\cite{MR4260476}.
\begin{corollary}\label{huddecor1}
Let the setting of Proposition~\ref{huddethm} hold with~$V_0=0$ and~$p^*\in[2,\infty)$. Suppose it holds~$\mathbb{P}$-a.s. that for any~$t\in[0,\tau)$, the process~$\hat{X}$ satisfies~$\hat{b}_t\cdot\hat{X}_t + \frac{1}{2}\|\hat{\sigma}_t\|^2 + \frac{1}{2}(p^*-2)\abs{\hat{\sigma}_t^{\top}\hat{X}_t}^2/\abs{\hat{X}_t}^2\leq \hat{\alpha}_t\abs{\hat{X}_t}^2 + \frac{1}{2}\abs{\hat{\beta}_t}^2$. For any~$q_1,q_2,q_3\in(0,\infty]$ with~$q_3<p^*$ and~$\frac{1}{q_1} = \frac{1}{q_2} + \frac{1}{q_3}$, there exists a constant~$C>0$ depending only on~$q_3,p^*$ such that
\begin{align*}
\bigg(\mathbb{E}\bigg[\bigg(\sup_{s\in[0,\tau]}\abs{\hat{X}_s}\bigg)^{p_1}\bigg]\bigg)^{\frac{1}{p_1}} &\leq C\bigg(\mathbb{E}\bigg[\exp\bigg(\sup_{t\in[0,\tau]}q_2\int_0^t\hat{\alpha}_sds\bigg)\bigg]\bigg)^{\frac{1}{q_2}}\cdot\bigg(\mathbb{E}\bigg[\bigg(\abs{\hat{X}_0}^2\\
&\qquad+\int_0^{\tau}\bigg\vert\frac{\hat{\beta}_s}{\exp(\int_0^s\hat{\alpha}_rdr)}\bigg\vert^{2}ds\bigg)^{\frac{q_3}{2}}\bigg]\bigg)^{\frac{1}{q_3}}.
\end{align*}
\end{corollary}
Another useful corollary of Proposition~\ref{huddethm} that will be used frequently in Section~\ref{weakrates} is stated next. Corollary~\ref{huddecor2} is a special case of Corollary~3.3 in~\cite{MR4260476}.
\begin{corollary}\label{huddecor2}
Let the setting of Proposition~\ref{huddethm} hold with~$V_0=0$ and~$O=\mathbb{R}^n$. Assume there exists Borel-measurable~$\bar{b}:\mathbb{R}^n\rightarrow \mathbb{R}^n$ and~$\bar{\sigma}:\mathbb{R}^n\rightarrow\mathbb{R}^{n\times n}$ such that for any~$t\in[0,T]$, it holds~$\mathbb{P}$-a.s. that~$\hat{b}_t = \bar{b}(\hat{X}_t)$ and~$\hat{\sigma}_t = \bar{\sigma}(\hat{X}_t)$. Let~$\bar{U}:\mathbb{R}^n\rightarrow\mathbb{R}$ be a Borel-measurable function satisfying~$\int_0^T\abs{\bar{U}(\hat{X}_s)}ds<\infty$, let~$U\in C^2(\mathbb{R}^n)$ and let~$\alpha^*\geq 0$. Assume~$LU + \frac{1}{2}\abs{\bar{\sigma}^{\top}\nabla U}^2 + \bar{U}\leq \alpha^* U$, where~$L$ is given by~\eqref{jen} with~$b,\sigma$ replaced by~$\bar{b},\bar{\sigma}$ respectively. It holds that
\begin{equation*}
\mathbb{E}\bigg[\exp\bigg(\frac{U(\hat{X}_{\tau})}{\exp(\alpha^*\tau)} + \int_0^{\tau}\frac{\bar{U}(\hat{X}_s)}{\exp(\alpha^*s)}ds\bigg)\bigg] \leq \mathbb{E}[\exp(U(0,\hat{X}_0))].
\end{equation*}
\end{corollary}
Lastly, a corollary of the It\^{o}-Alekseev-Gr\"{o}bner formula (Theorem~3.1 in~\cite{huddeiag}) is stated below as Proposition~\ref{iagthm}. The result will be used in Section~\ref{weakrates} to obtain our Theorem~\ref{wbth} on weak numerical convergence rates. 
Its proof is straight-forward given Theorem~3.1 in~\cite{huddeiag}, so it is omitted.
\begin{prop}\label{iagthm}
Let the setting of Corollary~\ref{huddecor2} hold with~$U=\bar{U}=0$. Moreover, let~$p^{\dagger}>4,q^{\dagger}\in[0,p^{\dagger}/2-2)$. 
Assume the filtration~$\mathcal{F}_t$ satisfies~$\mathcal{F}_t = \sigma(\mathcal{F}_0\cup \sigma(W_s: s\in[0,t]) \cup\{ A\in \mathcal{F}:\mathbb{P}(A) = 0\})$ and that~$\mathcal{F}_0$ and~$\sigma(W_s: s\in[0,T])$ are independent. Assume~$\bar{b},\bar{\sigma}$ are continuous. Let~$\bar{X}_{\cdot,\cdot}^{\cdot}:\Omega\times\Delta_T\times\mathbb{R}^n\rightarrow\mathbb{R}^n$ 
be such that it holds~$\mathbb{P}$-a.s. that for any~$(s,t)\in\Delta_T$, the map~$\mathbb{R}^n\ni x\mapsto \bar{X}_{s,t}^x\in\mathbb{R}^n$ is continuously differentiable in~$x$ up to order~$2$ and
the derivative
~$\Delta_T\times \mathbb{R}^n \ni ((s,t),x) \mapsto \partial^{\alpha}\bar{X}_{s,t}^x \in \mathbb{R}^n$ is continuous for all multiindices~$\alpha$ with~$0\leq\abs{\alpha}\leq 2$. Assume for all~$s\in[0,T]$,~$x\in\mathbb{R}^n$ that the process~$[s,T]\times\Omega\ni(t,\omega)\mapsto \bar{X}_{s,t}^x$ is~$\mathcal{F}_t$-adapted and assume that for all~$(s,t)\in\Delta_T$,~$x\in\mathbb{R}^n$, it holds~$\mathbb{P}$-a.s. that~$\bar{X}_{s,t}^x=x+\int_s^t\bar{b}(\bar{X}_{s,r}^x)dr + \int_s^t\bar{\sigma}(\bar{X}_{s,r}^x)dW_r$ and~$\bar{X}_{t,T}^{\bar{X}_{s,t}^x} = \bar{X}_{s,T}^x$. Let~$Y:\Omega\times[0,T]\rightarrow\mathbb{R}^n$,~$A:\Omega\times[0,T]\rightarrow\mathbb{R}^n$ and~$B:\Omega\times[0,T]\rightarrow\mathbb{R}^{n\times n}$ be~$\mathcal{F}\otimes\mathcal{B}([0,T])$-measurable functions such that~$\mathbb{E}[\int_0^T (\abs{Y_t}^{p^{\dagger}} + \abs{A_t}^{p^{\dagger}} + \abs{B_t}^{p^{\dagger}}) dt]<\infty$,~$Y$ has continuous sample paths,~$B$ has left-continuous sample paths,~$Y,B$ are both~$\mathcal{F}_t$-adapted and for all~$t\in[0,T]$, it holds~$\mathbb{P}$-a.s. that~$Y_t=Y_0+\int_0^tA_sds + \int_0^tB_sdW_s$. In addition, assume 
\begin{equation*}
\sup_{0\leq\abs{\alpha}\leq 2}\sup_{0\leq r\leq s \leq t \leq T} \mathbb{E}\bigg[\Big\vert \bar{b}\Big(\bar{X}_{s,t}^{Y_s}\Big)\Big\vert^{p^{\dagger}} + \Big\|\bar{\sigma}\Big(\bar{X}_{s,t}^{Y_s}\Big)\Big\|^{p^{\dagger}}
+ \bigg\vert\partial^{\alpha} \bar{X}_{t,T}^{\bar{X}_{r,s}^{Y_r}}\bigg\vert^{p_{\alpha}}\bigg] 
 < \infty,
\end{equation*}
where~$p_{\alpha}=p^{\dagger}$ if~$\abs{\alpha}=0$,~$p_{\alpha}=4p^{\dagger}/(p^{\dagger}-2(q^{\dagger}+2))$ if~$\abs{\alpha}=1$ and~$p_{\alpha} = 2p^{\dagger}/(p^{\dagger}-2(q^{\dagger}+2))$ if~$\abs{\alpha}=2$. If~$f\in C^2(\mathbb{R}^n)$ is such that there exists a constant~$C>0$ satisfying
\begin{equation*}
\max\bigg(\frac{\abs{f(x)}}{1+\abs{x}},\abs{\nabla f(x)},\Big\|D^2 f(x)\Big\|\bigg)\leq C\Big(1+\abs{x}^{q^{\dagger}}\Big)
\end{equation*}
for all~$x\in\mathbb{R}^n$, then 
it holds~$\mathbb{P}$-a.s. that
\begin{align*}
&\mathbb{E}\Big[f\Big(\bar{X}_{0,T}^{Y_0}\Big) - f(Y_T)\Big] \\
&\quad= \mathbb{E}\bigg[\int_0^T\Big(\Big(((\bar{b}(Y_t) - A_t)\cdot\nabla)\bar{X}_{t,T}^{Y_t}\Big)\cdot\nabla\Big)f\Big(\bar{X}_{t,T}^{Y_t}\Big) dt+\frac{1}{2}\int_0^T\sum_{i,j=1}^n\Big(\bar{\sigma}(Y_t)\bar{\sigma}(Y_t)^{\top}\\
&\qquad - B_tB_t^{\top}\Big)_{ij}\Big(\Big(\Big(\partial_i\bar{X}_{t,T}^{Y_t}\otimes\partial_j\bar{X}_{t,T}^{Y_t}\Big):D^2\Big)f\Big(\bar{X}_{t,T}^{Y_t}\Big) + \Big(\partial_{ij}^2\bar{X}_{t,T}^{Y_t}\cdot\nabla\Big)f\Big(\bar{X}_{t,T}^{Y_t}\Big)dt\bigg].
\end{align*}
\end{prop}

\section{Moment estimates on derivative processes}\label{momentestimates}

The following assumption 
states our main requirement on the Lyapunov function. Alternative, more local, assumptions for the main results 
are given in Theorem~\ref{alters1}. 
\begin{assumption}\label{A1lya}
There exists~$G:\Omega\times[0,T]\times O\rightarrow[0,\infty)$ such that~$G$ is~$\mathcal{F}\otimes\mathcal{B}([0,T])\otimes \mathcal{B}(O)$-measurable,~$G(t,\cdot)$ is~$\mathcal{F}_t\otimes\mathcal{B}(O)$-measurable for all~$t$, it holds~$\mathbb{P}$-a.s. that~$G(t,\cdot)$ is continuous for all~$t$, 
it holds~$\mathbb{P}$-a.s. that
\begin{align}
\abs*{b(t,x)-b(t,y)} &\leq (G(t,x)+G(t,y)) \abs*{x-y}, \label{a3}\\
\| \sigma(t,x) - \sigma(t,y) \|^2 &\leq (G(t,x)+G(t,y)) \abs*{x-y}^2, \label{a4}
\end{align}
for all $t\in [0,T]$, $x,y\in O$ and such that for any~$s\in[0,T]$, there exist finite sets~$I_0,I_0'\subset\mathbb{N}$,~$\tilde{n}_i\in\mathbb{N}$, open~$\tilde{O}_i\subseteq\mathbb{R}^{\tilde{n}_i}$ for all~$i\in I_0\cup I_0'$, locally bounded functions~$M:(0,\infty)\rightarrow(0,\infty),(\bar{x}_i:O\rightarrow\tilde{O}_i)_{i\in I_0\cup I_0'}$ and Lyapunov functions $(V_i:\Omega\times[0,T-s]\times\tilde{O}_i\rightarrow(0,\infty))_{i\in I_0\cup I_0'}$ 
satisfying for any~$m>0$,~$x\in O$ and stopping times~$t\leq T-s$ that it holds~$\mathbb{P}$-a.s. that
\begin{equation}\label{a5}
\int_0^tG(s+r,X_r^{s,x})dr\leq M(m)+ m \bigg(\sum_{i\in I_0} \int_0^t \log V_i(r,\bar{x}_i(x)) dr + \sum_{i'\in I_0'}\log V_{i'}(t,\bar{x}_{i'}(x))\bigg).
\end{equation}
\end{assumption}

In some cases, the process~$Y_t$ associated with Lyapunov functions can be thought of to be equal to~$X_t$. More precisely, we have in mind the case where the process~$X_t$ satisfies the conditions for~$Y_t$ in Definition~\ref{rlyadef} for the Lyapunov functions in Assumption~\ref{A1lya}. In particular, 
in the applications here, it is enough to take in place of~\eqref{a5} the condition
\begin{equation}\label{a52}
G(x)\leq m\log V_0(x) +M
\end{equation}
for~$V_0$ satisfying~$LV_0\leq CV_0$ for~$L$ given by~\eqref{jen}
; the generality is justified by a trick to increase the set of admissible Lyapunov functions, as exemplified by the inclusion of~$\bar{U}$ in Corollary~\ref{huddecor2}, 
see also~\cite[Theorem~2.24]{https://doi.org/10.48550/arxiv.1309.5595}. Assumption~\ref{A1lya} is strictly weaker than assuming globally Lipschitz coefficients, since polynomial Lyapunov functions are easily constructed in that case. In addition, throughout, whenever continuous differentiability up to some order~$m^*$ of~$b$ and~$\sigma$ is assumed, we also assume 
\begin{equation}\label{mdifffin}
\sum_{\theta\in \mathbb{N}_0^n;\abs*{\theta}\leq m^*}\int_0^T \sup_{\abs*{x}\leq R} (\abs{\partial^{\theta}b(t,x)}+\|\partial^{\theta}\sigma(t,x)\|)dt < \infty,\qquad\forall R>0.
\end{equation}

As briefly mentioned, in Section~\ref{altassumps}, it is shown that if~$b,\sigma$ are independent of~$\omega,t$ and admit locally Lipschitz derivatives, Assumption~\ref{A1lya} and in particular~\eqref{a3},~\eqref{a4} may be replaced by~$\sum_i(\abs{\partial_ib}+\|\partial_i\sigma\|^2)\leq G$ in obtaining our results on the Kolmogorov equation. 

For~$x\in O$,~$s\in[0,T]$, let~$X_{t(\kappa)}^{s,x}$ be the first~$t$-uniform derivatives in probability of $X_t^{s,x}$ with respect to the initial value in any direction~$\kappa\in\mathbb{R}^n$, that is, for any~$\epsilon>0$,~$T>0$,~$t\leq T-s$, it holds that
\begin{equation*}
\mathbb{P}\bigg(\sup_{t\in[0,T-s]}\bigg\vert \frac{X_t^{s,x+r\kappa} - X_t^{s,x}}{r} - X_{t(\kappa)}^{s,x}\bigg\vert > \epsilon \bigg) \rightarrow 0\\
\end{equation*}
as~$r\rightarrow 0$ with~$r\neq 0$,~$x+r\kappa\in O$. 
If~$b(t,\cdot)$ and~$\sigma(t,\cdot)$ are once continuously differentiable on~$O$ for all~$t\in[0,\infty)$,~$\omega\in\Omega$ and satisfy~\eqref{mdifffin} with~$m^* = 1$ for all~$\omega\in\Omega$, then by Theorem~4.10 in~\cite{MR1731794},~$X_{t(\kappa)}^{s,x}$ exists for any~$x\in O$,~$s\in[0,T]$ and satisfies the system obtained by formal differentiation of~\eqref{sde}, that is,
\begin{equation}
d X_{t(\kappa)}^{s,x} = (X_{t(\kappa)}^{s,x}\cdot\nabla) b(s+t,X_t^{s,x})  dt + (X_{t(\kappa)}^{s,x}\cdot\nabla) \sigma(s+t,X_t^{s,x}) dW_t.\label{firsdd}
\end{equation}
By induction, if for any~$\omega\in\Omega$,~$b(t,\cdot)$ and~$\sigma(t,\cdot)$ are continuously differentiable on~$O$ up to some order~$p$ for all~$t\in[0,\infty)$ and satisfy~\eqref{mdifffin} with~$m^* = p$, then the $p^{\textrm{th}}$-order~$t$-uniform derivative in probability of~$X_t^{s,x}$ with respect to the initial value in directions~$(\kappa_i)_{1\leq i\leq p}$,~$\kappa_i\in\mathbb{R}^n$,~$\abs*{\kappa_i} = 1$,~$1\leq i\leq p$ exists for any~$x\in O$,~$s\in[0,T]$ and satisfies the system obtained by a correponding~$p^{\textrm{th}}$-order formal differentiation of~\eqref{sde}.

First we state a straightforward application of the Lyapunov property to obtain an estimate of a time integral, which will be used later and is also demonstrative for many similar derivations in the following. Throughout and consistent with Assumption~\ref{A1lya}, we omit in the notation the dependence of~$V$,~$\bar{x}$ and~$M$ on~$s$.

\begin{lemma}\label{ergodfin}
Under Assumption~\ref{A1lya}, for any~$s\in[0,T]$,~$c>0$, there exists a constant~$C>0$ such that
\begin{align*}
& \mathbb{E} \Big[e^{c\int_0^{T_0-s}G(s+t,X_t^{s,x})dt}\Big]\\
&\quad \leq e^{c\hat{M}}\bigg(\frac{1}{T_0-s}\int_0^{T_0-s}\sum_{i\in I_0}\mathbb{E} [V_i(t,\bar{x}_i(x))] dt + \sum_{i'\in I_0'}\mathbb{E}[V_{i'}(T_0-s,\bar{x}_{i'}(x))] + 1\bigg)\\ 
&\quad \leq Ce^{c\hat{M}}\bigg(\sum_{i\in I_0\cup I_0'}\mathbb{E}[V_i(0,\bar{x}_i(x))] +1 \bigg)< \infty
\end{align*}
for all~$x\in O$ and~$T_0\in[s,T]$, where~$\hat{M} = M(m)$ for some~$m$. 
If in addition,~$V_i$ is local in~$s$ for all~$i\in I_0\cup I_0'$, then~$C$ is independent of~$s$.
\end{lemma}
\begin{proof}
The first inequality follows by applying~\eqref{a5}, then applying Jensen's inequality, setting a small enough~$m$ and applying Young's inequality. The last two inequalities follow by Corollary~\ref{huddecor0} 
with~$q_1 = 1$,~$q_2 = \frac{p_i^*}{p_i^*-1}$,~$p^*=p_i^*$ and the inequalities corresponding to~\eqref{lyap3}, where~$p_i^*$ is such that~$V_i$ is a~$(b^{(i)},\sigma^{(i)},\alpha^{(i)},\beta^{(i)},p_i^*,V_0^{(i)})$-Lyapunov function for some~$b^{(i)},\sigma^{(i)},\alpha^{(i)},\beta^{(i)},V_0^{(i)}$. 
\end{proof}

\begin{lemma}\label{firstlem}
Under Assumption~\ref{A1lya}, for any $k>0$,~$s\in [0,T]$, there exists 
$\rho>0$ such that
\begin{equation}\label{firncon0}
\mathbb{E}\bigg[\sup_{0\leq t\leq T_0-s}\abs*{X_{t(\kappa)}^{(r)}}^k\bigg] \leq \rho W(x,r\kappa)\abs*{r}^k
\end{equation}
for all~$x\in O$,~$T_0\in[s,T]$,~$r\in\mathbb{R}\setminus\{0\}$,~$\kappa\in\mathbb{R}^n$,~$\abs*{\kappa}= 1$, $x+r\kappa\in O$, where~$X_{t(\kappa)}^{(r)}:=X_t^{s,x+r\kappa} - X_t^{s,x}$ and $W(x,r\kappa) := 1 + \sum_{i\in I_0\cup I_0'}\mathbb{E}[V_i(0,\bar{x}_i(x+r\kappa))] + \mathbb{E}[V_i(0,\bar{x}_i(x))]$. If in addition it holds for any~$\omega\in\Omega$ that~$b(t,\cdot)$,~$\sigma(t,\cdot)$ are continuously differentiable for all~$t\geq 0$ and~\eqref{mdifffin} holds with $m^* = 1$, then 
\begin{align}
&\mathbb{E}\bigg[\sup_{0\leq t \leq T_0-s}\abs*{X_{t(\kappa)}^{s,x}}^k \bigg] \leq\rho W(x,0)\label{firncon}\\
\lim_{0\neq r\rightarrow 0}&\mathbb{E}\bigg[\sup_{0\leq t\leq T_0-s} \abs*{X_{t(\kappa)}^{s,x} - r^{-1}X_{t(\kappa)}^{(r)}}^k\bigg] = 0\label{prelcon}
\end{align}
for all~$x\in O$,~$T_0\in[s,T]$,~$\kappa\in\mathbb{R}^n$ with~$\abs*{\kappa} = 1$. 
If for each~$i\in I_0$,~$i'\in I_0'$, the functions~$V_i$,~$V_{i'}$ are local in~$s$ (as in Definition~\ref{locdef}), then~$\rho$ is independent of~$s$.
\end{lemma}
\begin{proof}
For any $r$, 
\begin{equation}\label{ipo}
dX_{t(\kappa)}^{(r)} = (b(s+t,X_t^{s,x+r\kappa}) - b(s+t,X_t^{s,x})) dt + (\sigma(s+t,X_t^{s,x+r\kappa}) - \sigma(s+t,X_t^{s,x}))dW_t.
\end{equation}

Since $X_t$ is almost surely continuous in $t$, for any~$0<t\leq T-s$, it holds~$\mathbb{P}$-a.s. that~$\int_0^t (G(s+u,X_u^{s,x+r\kappa}) + G(s+u,X_u^{s,x}))du \leq C(\int_0^t \sum_{i\in I_0}\log V_i(u,\bar{x}(x+r\kappa)) du+ \sum_{i'\in I_0'}\log V_{i'}(u,\bar{x}(x)) + 1) < \infty$, 
therefore 
Corollary~\ref{huddecor1} 
can be applied with
\begin{align*}
\hat{b}_t &= b(s+t,X_t^{s,x+r\kappa}) - b(s+t,X_t^{s,x}),\ \hat{\sigma}_t = \sigma(t,X_t^{s,x+r\kappa}) - \sigma(t,X_t^{s,x}),\\ 
\hat{\alpha}_t &= \bigg(\frac{1}{2} + k\vee 1\bigg)(G(s+t,X_t^{s,x+r\kappa}) + G(s+t,X_t^{s,x})),\\
p^* &= 2k\vee 2,\ \hat{\beta}_t = 0,\ q_1 = k,\ q_2 = 3k,\ q_3 = \frac{3k}{2},
\end{align*}
to obtain
\begin{align}
&\mathbb{E}\bigg[\sup_{0\leq t \leq T_0-s} \abs*{X_{t(\kappa)}^{(r)}}^k \bigg]\nonumber\\
&\quad\leq C \Big(\mathbb{E}\Big[e^{\int_0^{T_0-s} 3k(\frac{1}{2} + k\vee 1) (G(s+u,X_u^{s,x+r\kappa}) + G(s+u,X_u^{s,x})) du} \Big]\Big)^{\frac{1}{3}}\abs*{r}^k.\label{taig}
\end{align}
By 
Lemma~\ref{ergodfin}, the expectation on the right-hand side of~\eqref{taig} satisfies the bound
\begin{align*}
&\mathbb{E}\Big[e^{\int_0^{T_0-s} 3k(\frac{1}{2} + k\vee 1) (G(s+u,X_u^{s,x+r\kappa}) + G(s+u,X_u^{s,x})) du}\Big]\\
&\quad \leq \mathbb{E}\bigg[\frac{1}{2}e^{\int_0^{T_0-s} 6k(\frac{1}{2} + k\vee 1) G(s+u,X_u^{s,x+r\kappa})}du + \frac{1}{2}e^{\int_0^{T_0-s}6k(\frac{1}{2} + k\vee 1)G(s+u,X_u^{s,x})) du}\bigg]\\
&\quad\leq C
\mathbb{E}\bigg[1 + \sum_{i\in I_0\cup I_0'}V_i(0,\bar{x}_i(x + r\kappa)) + V_i(0,\bar{x}_i(x)) \bigg],
\end{align*}
which gives~\eqref{firncon0}. 

The statement for $X_{s(\kappa)}$ follows along the same lines, where 
instead $X_{s(\kappa)}$ satisfies~\eqref{firsdd} and Corollary~\ref{huddecor1} 
can be applied as above except with the corresponding~$\hat{b}_t,\hat{\sigma}_t$ and 
\begin{align}
\hat{\alpha}_t &= (1+2k\vee 2) G(s+t,X_t^{s,x}). \label{thisalpha}
\end{align}

Equation~\eqref{prelcon} is a known consequence; it is immediate from the definition of $X_{u(\kappa)}$, the previous bounds and
\begin{align}
\mathbb{E}[S^{k_1}] &\leq \epsilon\mathbb{P}( S^{k_1}\leq \epsilon) + \mathbb{E}[ \mathds{1}_{\{S^{k_1}> \epsilon \}} S^{k_1}]\nonumber\\
&\leq \epsilon\mathbb{P}( S^{k_1}\leq \epsilon) + \mathbb{E}[ \mathds{1}_{\{S^{k_1}> \epsilon \}}]\mathbb{E}[ S^k]^{\frac{k_1}{k}}\label{thatof}
\end{align}
with $S = \sup_{0\leq u \leq T_0-s}\abs{X_{u(\kappa)}^{s,x} - r^{-1}X_{u(\kappa)}^{(r)}}$.
The final assertion follows by noting that the constants~$C$ above are independent of~$s$ in case of local in~$s$ Lyapunov functions.
\end{proof}

The following Assumption~\ref{A2lya} states our requirements on the higher derivatives of~$b$ and~$\sigma$ for results on the higher derivatives of solutions to~\eqref{sde}.

\begin{assumption}\label{A2lya}
There exist $p\in\mathbb{N}_0$ such that $b(t,\cdot)\vert_{\hat{O}},\sigma(t,\cdot)\vert_{\hat{O}}\in C^p$ for all~$t\geq 0$,~$\omega\in\Omega$ and inequality~\eqref{mdifffin} holds with~$m^*=p$ for all~$\omega\in\Omega$. Moreover, for all~$s\in[0,T]$ and~$k\geq 2$, there exist~$\hat{n}_k\in \mathbb{N}$, open~$\hat{O}_k\subset \mathbb{R}^{\hat{n}_k}$, a mapping~$\hat{x}_k: O\rightarrow \hat{O}_k$, a constant (in particular in~$s$)~$M'>0$ and Lyapunov function~$\hat{V}_k^{s,T}:\Omega\times [0,T-s]\times \hat{O}_k\rightarrow (0,\infty)$ 
satisfying for any~$x,x'\in O$ and multiindices~$\alpha$ with~$2\leq \abs{\alpha}\leq p$ that it holds~$\mathbb{P}$-a.s. that
\begin{align}
&\abs{\partial^{\alpha}b(s+t,\lambda X_t^{s,x}+(1-\lambda)X_t^{s,x'})} + \|\partial^{\alpha}\sigma(s+t,\lambda X_t^{s,x} + (1-\lambda)X_t^{s,x'})\|^2\nonumber\\
&\qquad \leq M'(1+ \hat{V}_k^{s,T}(t,\hat{x}_k(x)) + \hat{V}_k^{s,T}(t,\hat{x}_k(x')))^{\frac{1}{k}}\label{intermid}
\end{align}
for all~$t\in[0,T-s]$,~$\lambda\in [0,1]$.
\end{assumption}

Similar to Assumption~\ref{A1lya}, in Section~\ref{altassumps}, it is shown that if~$b,\sigma$ are independent of~$\omega,t$ and admit locally Lipschitz derivatives, then Assumption~\ref{A2lya} and in particular~\eqref{intermid} may be replaced by~$\abs{\partial^{\alpha}b(X_t^{s,x})}+\|\partial^{\alpha}\sigma(X_t^{s,x})\|^2\leq M'(1+\hat{V}^{s,T}(t,\hat{x}_k(x)))^{1/k}$.

In the following, for~$\kappa = (\kappa_i)_{1\leq i\leq l}$,~$\kappa_i\in\mathbb{R}^n$, the~$l^{\textrm{th}}$ order~$t$-uniform derivatives in probability of a process $Z_t^x$ with respect to initial condition $x$ in the directions~$\kappa_1,\dots,\kappa_l\in\mathbb{R}^n$ is denoted by~$\partial^{(\kappa)}Z_t^x$.
\begin{theorem}\label{seclem}
Under Assumptions~\ref{A1lya} and~\ref{A2lya}, 
for any~$s\in[0,T]$, constants~$1\leq l\leq p-1$,~$k_1>0$, 
there exist~$i^*\in\mathbb{N}$,~$\nu \geq\frac{k_1}{2}$,~$\{l_i\}_{i\in\{1,\dots,i^*\}}\subset (0,\infty)$ and a finite order polynomial~$q_0$, the degree of which is independent of
~$s,V_i,\hat{V}_k^{s,T}$, 
such that
\begin{align}
&\mathbb{E}\bigg[\sup_{0\leq t\leq T_0-s}\abs*{\partial^{(\kappa)}X_t^{s,x+r\kappa_{l+1}}-\partial^{(\kappa)}X_t^{s,x}}^{k_1}\bigg] \leq (T_0-s)^{\nu} q(x,x+r\kappa_{l+1}) \abs*{r}^{k_1}\label{kappar1}\\
&\mathbb{E}\bigg[\sup_{0\leq t\leq T_0-s}\abs*{\partial^{(\bar{\kappa})}X_t^{s,x}}^{k_1}\bigg] \leq (T_0-s)^{\nu}  q(x,x)\label{kappar2}\\
\lim_{r\rightarrow 0}&\mathbb{E}\bigg[\sup_{0\leq t\leq T_0-s}\abs*{\partial^{(\bar{\kappa})}X_t^{s,x} - r^{-1}(\partial^{(\kappa)}X_t^{s,x+r\kappa_{l+1}} - \partial^{(\kappa)}X_t^{s,x})}^{k_1}\bigg] = 0 \label{kappar3}
\end{align}
for all initial condition $x\in O$,~$T_0\in[s,T]$, $r\in\mathbb{R}\setminus\{0\}$,
~$\kappa_i\in\mathbb{R}^n$,~$\abs*{\kappa_i} =1$, $1\leq i\leq l+1$, $x+r\kappa_{l+1}\in O$, where~$\kappa = (\kappa_i)_{1\leq i\leq l}$,~$\bar{\kappa} = (\kappa_i)_{1\leq i\leq  l+1}$ and $q: O\times O\rightarrow\mathbb{R}$ is given by 
\begin{align}
q(y,y') &= \mathbb{E}[q_0((V_i(0,\bar{x}_i(y)))_{i\in I_0\cup I_0'},(\hat{V}_{l_i}^{s,T}(0,\hat{x}_{l_i}(y)))_{i\in\{1,\dots,i^*\}},\nonumber \\
&\qquad (\hat{V}_{l_i}^{s,T}(0,\hat{x}_{l_i}(y')))_{i\in\{1,\dots,i^*\}})].\label{qpoly}
\end{align}
If~$V_i$ and~$\hat{V}_k^{s,T}$ are local in~$s$ (as in Definition~\ref{locdef}) for every~$i,k$, then the form of the polynomial~$q_0$ is independent of~$s$.
\end{theorem}

\begin{proof}
Fix~$k_1>0
$,~$s\in[0,T]$, let $J$ be the set of strictly increasing functions from $\mathbb{N}$ to itself and $D^{(\kappa)}b(s+t,X_t^{s,x})$ denote the formal derivative of $b(s+t,X_t^{s,x})$ with respect to $x$ in the directions indicated by $\kappa$. In particular,
\begin{align*}
D^{(\kappa)}b(s+t,X_t^{s,x}) &= \Big(\partial^{(\kappa)}X_t^x\cdot\nabla\Big) b(s+t,X_t^{s,x})\\
&\quad + 
q_{b,X_t^{s,x}}\bigg(\bigg(\prod_{1\leq i\leq l'}\partial^{(\kappa_{j(i)})}\bigg)X_t^{s,x} ,1\leq l'\leq l-1, j\in J\bigg),
\end{align*}
where the last term denotes a $\mathbb{R}^n$-valued polynomial taking arguments as indicated, for which exactly $l$ of the operators $\partial^{(\kappa_i)}$ appear in each term and coefficients are spatial derivatives between orders $2$ and $l$ of elements of $b$ evaluated at $(s+t,X_t^{s,x})$. 
In the same way, let~$D^{(\kappa)}\sigma(s+t,X_t^{s,x})$ denote the formal derivative of~$\sigma(s+t,X_t^{s,x})$ with respect to~$x$ in the directions indicated by~$\kappa$, again with the form above but with~$b$ replaced by~$\sigma$ everywhere. 
Denoting $x' = x+r\kappa_{l+1}$, the difference processes of the derivatives satisfy
\begin{align*}
d(\partial^{(\kappa)}X_t^{s,x'} - \partial^{(\kappa)} X_t^{s,x}) &= (D^{(\kappa)}b(s+t,X_t^{s,x'}) - D^{(\kappa)}b(s+t,X_t^{s,x}) )dt\\
&\quad + (D^{(\kappa)}\sigma(s+t,X_t^{s,x'}) - D^{(\kappa)}\sigma(s+t,X_t^{s,x}) )dW_t
\end{align*}
on~$t\in[0,T-s]$ for all~$x,x'\in O$,~$r\in\mathbb{R}\setminus\{ 0\}$,~$\kappa_i\in\mathbb{R}^n$,~$\abs*{\kappa_i} = 1$,~$1\leq i\leq l+1$.

We proceed by strong induction in $l$ for~\eqref{kappar1}. A base case has been established in Lemma~\ref{firstlem}. By the fundamental theorem of calculus on derivatives of~$b$ and~$\sigma$, inequalities~\eqref{intermid},~\eqref{a3} and~\eqref{a4}, 
it holds~$\mathbb{P}$-a.s. that
\begin{align*}
&\abs*{D^{(\kappa)}b(s+t,X_t^{s,x'}) - D^{(\kappa)}b(s+t,X_t^{s,x}) }\\
&\quad\leq \sum_i \abs*{(\partial^{(\kappa)} X_t^{s,x'} - \partial^{(\kappa)} X_t^{s,x})_i}  \abs*{\partial_i b(s+t,X_t^{s,x})} + H(t,X_t^{s,x},X_t^{s,x'})\hat{q}_t\\
&\quad\leq 2\abs*{\partial^{(\kappa)} X_t^{s,x'} - \partial^{(\kappa)} X_t^{s,x}} \ G(s+t,X_t^{s,x}) + H(t,X_t^{s,x},X_t^{s,x'})\hat{q}_t,\\
&\|D^{(\kappa)}\sigma(s+t,X_t^{s,x'}) - D^{(\kappa)}\sigma(s+t,X_t^{s,x}) \|^2 \\
&\quad\leq 2\sum_i \abs*{(\partial^{(\kappa)} X_t^{s,x'} - \partial^{(\kappa)} X_t^{s,x})_i}^2\|\partial_i \sigma(s+t,X_t^{s,x})\|^2 + (H(t,X_t^{s,x},X_t^{s,x'})\hat{q}_t)^2\\
&\quad\leq 4\abs*{\partial^{(\kappa)} X_t^{s,x'} - \partial^{(\kappa)} X_t^{s,x}}^2 G(s+t,X_t^{s,x}) + (H(t,X_t^{s,x},X_t^{s,x'})\hat{q}_t)^2,
\end{align*}
on~$t\in[0,T-s]$, where 
\begin{align}
H(t,X_t^{s,x},X_t^{s,x'}) &= M'\Big(1+ \hat{V}_{4k_1\vee 4}^{s,T}(t,\hat{x}_{4k_1\vee 4}(x)) + \hat{V}_{4k_1\vee 4}^{s,T}(t,\hat{x}_{4k_1\vee 4}(x'))\Big)^{\frac{1}{4k_1\vee 4}}\label{hxx}
\end{align}
and $\hat{q}_t$ denotes a polynomial with constant coefficients taking arguments from the set $S = S_1\cup S_2$, 
\begin{align*}
S_1&= \bigg\{\bigg\vert \bigg(\prod_{1\leq i \leq l'} \partial^{(\kappa_{j(i)})}\bigg) X_t\bigg\vert : 1\leq l' \leq l,j\in J, X_t\in\{X_t^{s,x'},X_t^{s,x}\}\bigg\} \\
S_2&= \bigg\{\bigg\vert\bigg(\prod_{1\leq i\leq l'}\partial^{(\kappa_{j(i)})}\bigg) (X_t^{s,x'}- X_t^{s,x})\bigg\vert:1\leq l'\leq l-1,j\in J\bigg\}\\
&\qquad \cup \{ \lvert X_t^{s,x'} - X_t^{s,x}\rvert \}, 
\end{align*}
for which exactly $l$ of the operators $\partial^{(\kappa_i)}$ appear in each term of $\hat{q}_s$ and a factor from $S_2$ appears exactly once in each term. 
Note for $p\geq 2$ and by Lemma~\ref{ergodfin}, it holds 
$\mathbb{P}$-a.s. that
\begin{align*}
&\int_0^{T-s}\log V_i(t,\bar{x}_i(x)) dt + \log V_{i'} (t,\bar{x}_{i'}(x)) \\
&\quad< \int_0^{T-s} V_i(t,\bar{x}_i(x)) dt + V_{i'}(t,\bar{x}_{i'}(x)) < \infty
\end{align*}
on~$t\in[0,T-s]$ for all~$i\in I_0$ and $i'\in I_0'$. 
Corollary~\ref{huddecor1} can then be applied with
\begin{align}
\hat{b}_t &= D^{(\kappa)}b(s+t,X_t^{s,x'}) - D^{(\kappa)}b(s+t,X_t^{s,x}),\nonumber\\
\hat{\sigma}_t &= D^{(\kappa)}\sigma(s+t,X_t^{s,x'}) - D^{(\kappa)}\sigma(s+t,X_t^{s,x}),\nonumber\\
\hat{\alpha}_t &= 
4(2k_1\vee2) G(s+t,X_t^{s,x}) + \frac{1}{2}>0,\label{thisalpha2}\\
\hat{\beta}_t &= \sqrt{4k_1\vee 4}H(t,X_t^{s,x},X_t^{s,x'})\hat{q}_t,\nonumber\\
p^* &= 4k_1 \vee 4,\ q_1 = k_1,\ q_2 = \bigg(\frac{1}{k_1} - \frac{1}{2k_1\vee 2}\bigg)^{-1},\ q_3 = 2k_1\vee 2\nonumber
\end{align}
to obtain
\begin{equation*}
\mathbb{E}\bigg[\sup_{0\leq t\leq T_0-s} \abs*{\partial^{(\kappa)}X_t^{s,x'} - \partial^{(\kappa)}X_t^{s,x}}^{k_1}\bigg]\leq CA_{T_0-s}^{(1)}A_{T_0-s}^{(2)},
\end{equation*}
where, using the notation~$q_2,q_3$ above,
\begin{align*}
A_{T_0-s}^{(1)} &:= \bigg(\mathbb{E}\bigg[\exp\bigg(q_2\int_0^{T_0-s}\bigg(4(1+2k_1\vee 2)G(s+u,X_u^{s,x}) + \frac{1}{2}\bigg)du\bigg)\bigg) \bigg]\bigg)^{\frac{k_1}{q_2}} \\
A_{T_0-s}^{(2)} &:= \bigg(\mathbb{E}\bigg[\bigg( \int_0^{T_0-s} 2(1+2k_1\vee 4) (H(u,X_u^{s,x},X_u^{s,x'}) \hat{q}_u)^2du\bigg)^{\frac{q_3}{2}}\bigg]\bigg)^{\frac{k_1}{q_3}}.
\end{align*}
Setting
\begin{equation*}
m = \frac{1}{8q_2(1+2k_1\vee2)((T_0-s)\vee 1)(\abs{I_0}+\abs{I_0'})}
,
\end{equation*}
with the effect that~$M(m)$ is bounded in~$T_0\in[s,T]$, 
and using Assumption~\ref{A1lya} as well as (the arguments in the proof of) Lemma~\ref{ergodfin}, the first expectation has the bound
\begin{align*}
A_{T_0-s}^{(1)} &\leq \bigg(\mathbb{E}\bigg[\frac{C}{T_0-s}\int_0^{T_0-s} \exp\bigg( 8q_2(1+2k_1\vee2)m \bigg((T_0-s)\sum_{i'\in I_0'}\log V_{i'}(u,\bar{x}_{i'}(x))\\
&\quad + \sum_{i\in I_0} \log V_i(T_0-s,\bar{x}_i(x)) \bigg)\bigg) du\bigg]\bigg)^{\frac{k_1}{q_2}} \\
&\leq \bigg( \mathbb{E}\bigg[ \frac{C}{T_0-s}\int_0^{T_0-s}  \bigg( 1+\!\!\sum_{i'\in I_0'}\!V_{i'}(u,\bar{x}_{i'}(x)) +\!\! \sum_{i\in I_0} V_i(T_0-s,\bar{x}_i(x)) \bigg)du \bigg]\bigg)^{\frac{k_1}{q_2}}\\
&\leq C\bigg(1+\sum_{i'\in I_0'}\mathbb{E}[V_{i'}(0,\bar{x}_{i'}(x))] + \sum_{i\in I_0} \mathbb{E}[V_i(0,\bar{x}_i(x))] \bigg)^{\frac{k_1}{q_2}},
\end{align*}
where~$C$ is, here and in the rest of the proof, bounded as a function of~$T_0\in[s,T]$ and also of~$s$ if~$V_i$ is local in~$s$ for all~$i\in I_0\cup I_0'$. 
On the other hand, by the inductive argument and the form of~$H$,~$\hat{q}_s$ and~$q_3$, it holds that
\begin{align*}
A_{T_0-s}^{(2)} &\leq C\bigg(\mathbb{E} \bigg[\bigg(\int_0^{T_0-s} H(u,X_u^{s,x},X_u^{s,x'})^2du\bigg)^{k_1\vee 1}\sup_{0\leq u\leq T_0-s}\hat{q}_u^{2k_1\vee 2}\bigg]\bigg)^{\frac{k_1}{2k_1\vee2}}\\
&\leq C\bigg( \mathbb{E}\bigg[\bigg(\int_0^{T_0-s} H(u,X_u^{s,x},X_u^{s,x'})^2du\bigg)^{2k_1\vee 2}\bigg] \bigg)^{\frac{k_1}{4k_1\vee 4}} \\
&\quad \cdot \bigg(\mathbb{E}\bigg[\sup_{0\leq u\leq T_0-s}\hat{q}_u^{4k_1\vee 4}\bigg]\bigg)^{\frac{k_1}{4k_1\vee 4}}\\
&\leq C\bigg((T_0-s)^{(2k_1\vee 2)-1}\int_0^{T_0-s}\mathbb{E}[1+\hat{V}_{4k_1\vee 4}^{s,T}(u,\hat{x}_{4k_1\vee 4}(x))\\
&\qquad + \hat{V}_{4k_1\vee 4}^{s,T}(u,\hat{x}_{4k_1\vee 4}(x')) ]du \bigg)^{\frac{k_1}{4k_1\vee 4}} 
 \tilde{q}(x,x')\abs*{r}^{k_1},
\end{align*}
where
\begin{align*}
\tilde{q}(x,x') &= \mathbb{E}[\tilde{q}_0((V_i(0,\bar{x}_i(x)))_{i\in I_0\cup I_0'},(\hat{V}_{l_i}^{s,T}(0,\hat{x}_{l_i}(x)))_{i\in\{1,\dots,\hat{i}^*\}},\\
&\qquad (\hat{V}_{l_i}^{s,T}(0,\hat{x}_{l_i}(x')))_{i\in\{1,\dots,\hat{i}^*\}})]
\end{align*}
for some $\hat{i}^*\in\mathbb{N}$, $\{l_i\}_{i\in\{1,\dots,\hat{i}*\}}\subset (0,\infty)$ and finite order polynomial~$\tilde{q}_0$ taking arguments as indicated. 
Therefore, by Corollary~\ref{huddecor0} 
with~$q_1 = 1$, it holds that
\begin{align*}
A_{T_0-s}^{(2)} &\leq C \Big( (T_0-s)^{(2k_1\vee 2)} (1 + \mathbb{E}[\hat{V}_{4k_1\vee 4}^{s,T}(0,\hat{x}_{4k_1 \vee 4}(x))]\\
&\quad + \mathbb{E}[\hat{V}_{4k_1\vee 4}^{s,T}(0,\hat{x}_{4k_1 \vee 4}(x'))]) \Big)^{\frac{k_1}{4k_1\vee 4}} 
\tilde{q}(x,x')\abs*{r}^{k_1},
\end{align*}
which concludes the proof for~\eqref{kappar1}. 
Inequality~\eqref{kappar2} follows along the same lines, therefore the argument is not repeated. Equation~\eqref{kappar3} holds by~\eqref{thatof} with
\begin{equation*}
S = \sup_{0\leq u \leq t} \abs*{\partial^{(\bar{\kappa})}X_u^{s,x} - r^{-1}(\partial^{(\kappa)}X_u^{s,x'} - \partial^{(\kappa)}X_u^{s,x})}. \tag*{\textrm{\qedhere}}
\end{equation*}
\end{proof}

\begin{remark}\label{replacingrem}
A way to prove weaker versions of Lemma~\ref{firstlem} and Theorem~\ref{seclem} is instead of using the stochastic Gr\"{o}nwall inequality, that is, 
Proposition~\ref{huddethm}, 
to use Lemma~4.2 in~\cite{MR1731794} and Theorem~3.5 in \cite{MR2894052}. For this, one works directly with the SDEs governing~$\abs{\partial^{(\kappa)}X_t}^{k_1}$ in the proof and inequality~\eqref{lyapr2} is to be replaced by~$(\partial_t + L)V_0 \leq C V_0$. 
\end{remark}

\section{Kolmogorov equations}\label{kollya}

Throughout this section, we assume that~$b$ and~$\sigma$ are nonrandom functions. In Section~\ref{semdiff}, the moment estimates from Section~\ref{momentestimates} are used to derive~$p^\textrm{th}$ differentiability of a Feynman-Kac semigroup~\eqref{basicu}. 
The functions~$f,c,g$ appearing in~\eqref{basicu} are only required to be bounded by Lyapunov functions. Although the results and many details in the proofs are new, the approach is from~\cite{MR1731794}, in which~$f,c,g$ and their derivatives are only required to be polynomially bounded. This regularity is then used to show that the semigroup solves the Kolmogorov equation in the almost everywhere sense in Section~\ref{tdkol}. 

\subsection{Semigroup differentiability}\label{semdiff}


A condition that will be imposed on the derivatives of~$f,c,g$ is first stated. 
\begin{definition}\label{lyadiff}
For~$p\in \mathbb{N}$,~$k>1$,~$h:\Omega\times[0,T]\times\hat{O}\rightarrow \mathbb{R}$ satisfying~$\mathbb{P}$-a.s. that~$h(t,\cdot)\in C^p(\hat{O})$ for all~$t\in[0,T]$, we say that~$h$ has~$(p,k)$-Lyapunov derivatives if there exist~$(V^{s,T})_{s\in[0,T]}$ local in~$s$ (as in Definition~\ref{locdef}), locally bounded\footnote{The domain of~$\tilde{x}$ is~$O$, but its codomain is unspecified, except that it is equal to the domain of~$V^{s,T}$.}~$\tilde{x}$ and constant~$N>0$ such that for any~$s\in[0,T]$ and multiindices~$\alpha$ with~$0\leq \abs{\alpha}\leq p$, it holds~$\mathbb{P}$-a.s. that
\begin{equation}\label{gbylya}
\abs{\partial^{\alpha}h(s+t,\lambda X_t^{s,x} + (1-\lambda)X_t^{s,x'})} \leq N(1+V^{s,T}(t,\tilde{x}(x) ) + V^{s,T}(t,\tilde{x}(x')))^{\frac{1}{k}}
\end{equation}
for all stopping times~$t\leq T-s$,~$x,x'\in O$ and~$\lambda\in[0,1]$.
\end{definition}
Similar to Assumptions~\ref{A1lya},~\ref{A2lya}, an alternative sense of Lyapunov derivatives is given in Section~\ref{altassumps}, where~\eqref{gbylya} is replaced by~$\abs{\partial^{\alpha}h(s+t,X_t^{s,x})}\leq N(1+V^{s,T}(t,\tilde{x}(x)))^{1/k}$. This will be used in Section~\ref{altassumps} along with alternative assumptions to Assumptions~\ref{A1lya},~\ref{A2lya} in order to obtain results on Kolmogorov equations similar to the ones obtained in the present section.
\begin{assumption}\label{Afcg}
The functions~$f:[0,T]\times\mathbb{R}^n\rightarrow\mathbb{R},c:[0,T]\times\mathbb{R}^n\rightarrow(0,\infty),g:\mathbb{R}^n\rightarrow\mathbb{R}$ satisfy that 
\begin{enumerate}
\item for all~$R>0$,~$h\in\{f,c,g\}$, there exists~$C_R>0$ such that~$\abs{h(t,x)-h(t,x')}\leq C_R\abs{x-x'}$ for all~$t\in[0,T]$,~$x,x'\in B_R$, 
\item for~$h\in\{f,g\}$, if~$(p,k)$ is such that~$h$ has~$(p,k)$-Lyapunov derivatives, then there exist Lyapunov functions~$V^{s,T}$ and mapping~$\tilde{x}$ satisfying the condition in Definition~\ref{lyadiff} such that there exists~$C>0$ such that for any~$s\in[0,T]$, it holds~$\mathbb{P}$-a.s. that
\begin{equation}\label{maiscu}
V^{s+\tau,T}(0,\tilde{x}(X_{\tau}^{s,x}))^{\frac{1}{k}}\leq C (1+V^{s,T}(\tau,\tilde{x}(x)))
\end{equation}
for all~$x\in\mathbb{R}^n$ and stopping times~$\tau\leq T-s$.
\end{enumerate}
\end{assumption}
In addition, the following assumption will be made for our assertions about the Kolmogorov equation.
\begin{assumption}\label{A3lya}
The functions~$b,\sigma$ are nonrandom and~$O=\mathbb{R}^n$. For each $R\geq 0$, there exists a Borel, locally integrable $K_{\cdot}(R): [0,\infty)\rightarrow [0,\infty)$ such that
\begin{equation*}
2\langle x-y,b(t,x)-b(t,y)\rangle + \| \sigma(t,x) - \sigma(t,y) \|^2 \leq K_t(R) \abs*{x-y}^2\\
\end{equation*}
for all~$t\geq 0$,~$x,y\in B_R$. 
For any $s \geq 0$, $T>0$, $x\in \mathbb{R}^n$, there exists a $\mathbb{P}$-a.s. continuous $\mathbb{R}^n$-valued unique up-to-indistinguishability solution~$X_t^{s,x}$ to~\eqref{sde} 
on $[0,T]$. Moreover, for any~$T>0$, 
there exist~$\tilde{n}\in\mathbb{N}$, open~$\tilde{O}\subseteq \mathbb{R}^{\tilde{n}}$,
~$V_0\in C^{1,2}([0,\infty)\times\tilde{O})$, $\tilde{x}:\mathbb{R}^n\rightarrow\tilde{O}$, $\hat{G}:[0,\infty)\times \mathbb{R}^n\rightarrow\mathbb{R}$, constant~$C\geq 0$ 
and~$k\geq 1$
such that
 
\begin{enumerate}[label=(\roman*)]
\item there exists a family of Lyapunov functions~$(V^{s,T}:\Omega\times[0,T]\times\tilde{O}\rightarrow(0,\infty))_{s\in[0,T]}$ that is~$(\tilde{n},\tilde{O},V_0)$-local in~$s$ (as in Definition~\ref{locdef}),
\item for any $s\geq 0$, it holds~$\mathbb{P}$-a.s. that~\eqref{maiscu} holds 
for all~$x\in \mathbb{R}^n$ and stopping times~$\tau\leq T$,
\item \label{Ginf} for any $s\geq 0$, it holds that~$\lim_{\abs*{x}\rightarrow \infty}\inf_{t\in[0,T]}\hat{G}(t,x) = \infty$ and $\mathbb{P}$-a.s. that
\begin{equation*}
\hat{G}(s+t,X_t^{s,x}) \leq V^{s,T}(t,\tilde{x}(x))
\end{equation*}
for all $t\in[0,T]$, $x\in \mathbb{R}^n$.
\end{enumerate}
\end{assumption}

\begin{remark}\label{lyarem}
The parts about~$V^{s,T}$ in Assumptions~\ref{A3lya},~\ref{Afcg} are satisfied by the Lyapunov functions considered for example in~\cite[Corollary~2.4]{https://doi.org/10.48550/arxiv.1309.5595}. More specifically, taking~$\alpha$ and the functions~$U$,~$\bar{U}$ from there, for~$\tilde{n} = n+1$, one may take~$V_0 = V_0(t,(x,y)) = e^{U(x)e^{-\alpha t} + y}$ and~$\tilde{x}=\tilde{x}(x)= (x,0)\in\mathbb{R}^{n+1}$, then
\begin{align*}
\tilde{b}(t,(x,y)) &= (b(t,x),\bar{U}(t,x)), \quad \tilde{\sigma}(t,(x,y)) = \begin{pmatrix}\sigma(t,x) &0\\ 0 &0\end{pmatrix},\\
\hat{G}(t,x) &= e^{U(x)e^{-\alpha t}}
\end{align*}
for~$t\geq 0$,~$x\in\mathbb{R}^n$,~$y\in\mathbb{R}$ and the latter statements of Assumption~\ref{A3lya},~\ref{Afcg} are satisfied given the conditions on~$U$ and~$\bar{U}$ if~$\lim_{\abs*{x}\rightarrow\infty} U(x) = \infty$ and~$\bar{U}\geq 0$ for some~$C\in\mathbb{R}$ everywhere.
\end{remark}

\begin{theorem}\label{fk2}
Let~$b,\sigma$ be nonrandom, let Assumptions~\ref{A1lya},~\ref{A2lya} hold and let~$f:\Omega\times[0,T]\times\mathbb{R}^n\rightarrow\mathbb{R}$,~$c:\Omega\times[0,T]\times\mathbb{R}^n\rightarrow [0,\infty)$,~$g:\Omega\times\mathbb{R}^n\rightarrow\mathbb{R}$ be such that~$f(\cdot,x),c(\cdot,x)$ are~$\mathcal{F}\otimes\mathcal{B}([0,T])$-measurable functions for every~$x\in \hat{O}$, satisfying for any~$\omega\in\Omega$ that~$\int_0^T \sup_{x\in B_R\cap \hat{O}}(\abs{c(t,x)} + \abs{f(t,x)})dt <\infty$ for every~$R>0$ and~$f(t,\cdot)\vert_{\hat{O}},c(t,\cdot)\vert_{\hat{O}},g\vert_{\hat{O}}\in C^p(\hat{O})$ for all~$t\in[0,T]$. 
Assume there exists~$k_2>1$ such that~$f$ and~$g$ have~$(p,k_2)$-Lyapunov derivatives. There exists~$K>1$ such that if for any~$1<k'<K$,
~$c$ has~$(p,k')$-Lyapunov derivatives, 
then the following statements hold.
\begin{enumerate}[label=(\roman*)]
\item\label{tstwi} For~$u$ given by
\begin{align}
u(s,t,x) &= \int_0^t f(s+r,X_r^{s,x}) e^{-\int_0^r c(s+w,X_w^{s,x}) dw} dr \nonumber \\
&\qquad+ g(X_t^{s,x})e^{-\int_0^t c(s+w,X_w^{s,x})dw}\label{etstw}
\end{align}
defined for $(s,x)\in [0,T]\times O$ and stopping times~$t\leq T-s$, the expectation~$\mathbb{E}[u(s,t,x)]$ is continuously differentiable in~$x$ up to order~$p$.
\item\label{tstwii} For every multiindex~$\beta$ with~$0\leq \abs{\beta}\leq p$, there exists a finite order polynomial~$q^*$, the degree of which is independent of all of the Lyapunov functions in Assumptions~\ref{A1lya},~\ref{A2lya} and of the Lyapunov derivatives, 
such that for~$(s,x)\in[0, T]\times O$ and all stopping times~$t\leq T-s$, it holds that
\begin{align}
\abs{\partial^{\beta}_x \mathbb{E}[u(s,t,x)] }&\leq \mathbb{E}[q^*((V_i(0,\bar{x}_i(x)))_{i\in I_0\cup I_0'},V^{s,T}(0,\tilde{x}(x)), \nonumber\\
&\qquad(\hat{V}_{l_i}^{s,T}(0,\hat{x}_{l_i}(x)))_{i\in I^*} )],\label{tstweq}
\end{align}
where~$I^*\subset\mathbb{N}$ is finite,~$l_i >0$ and~$\tilde{x}$,
~$V^{s,T}$
associated to the~$(\abs{\beta},k_2)$-Lyapunov derivatives of~$f,g$ and the~$(\abs{\beta},k')$-Lyapunov derivatives of~$c$ are representative across any and all of~$\{f,c,g\}$ 
and~$k'\in K_0\subset (1,K)$ for some finite~$K_0$. 
\item\label{tstwiii} 
Let Assumptions~\ref{A3lya} hold. 
Suppose~$f,c,g$ are nonrandom and that they satisfy Assumption~\ref{Afcg}. 
Suppose the families of Lyapunov functions in Assumptions~\ref{A1lya},~\ref{A2lya},~\ref{Afcg}  and for the Lyapunov derivatives of~$c$ 
are local in~$s$ (as in Definition~\ref{locdef}). For any multiindex~$\alpha$ with~$0\leq \abs{\alpha}\leq p$, the function~$\Delta_T\times \mathbb{R}^n\ni((s,t),x)\mapsto\abs{\partial^{\alpha}_x \mathbb{E}[u(s,t,x)]}$ is locally bounded 
and if~$p\geq 2$, then for any~$R>0$, there exists a constant~$N>0$ such that
\begin{equation}\label{cop1}
\abs{\mathbb{E}[u(s',T-s',x)] - \mathbb{E}[u(s,T-s,x)}] \leq N\abs{s'-s}
\end{equation}
for all~$s,s'\in(0,T)$ and~$x\in B_R$.
\end{enumerate}
\end{theorem}
We prove first a lemma that will used in the proof of Theorem~\ref{fk2}. Throughout the proofs of Theorem~\ref{fk2} and of Lemma~\ref{lemtstw} and consistent with the statement of the results, we omit in the notation any dependence of~$V^{s,T}$,~$\tilde{x}$ and~$k_2$ on~$k$ and~$h$. 
\begin{lemma}\label{lemtstw}
Let the assumptions of Theorem~\ref{fk2} hold. Suppose there exists~$\hat{k}_2>1$ such that~$c$ has~$(p,\hat{k}_2)$-Lyapunov derivatives. 
For any~$h\in\{f,c,g\}$,~$k_3 >0$,~$s\in[0,T]$,~$x\in O$,~$\kappa\in\mathbb{R}^n$,~$\lambda'\in[0,1]$, multiindex~$\alpha$ and stopping time~$t\leq T-s$, such that~$k_3 < k_2$ if~$h\in\{f,g\}$ and~$k_3 < \hat{k}_2$ if~$h = c$ as well as~$\abs*{\kappa} = 1$,~$0\leq \abs{\alpha}\leq p$, it holds that
\begin{align}
&\mathbb{E}\bigg[\int_0^t\abs*{\partial^{\alpha} h(s+r,\lambda' X_r^{s,x'} + (1-\lambda')X_r^{s,x}) - \partial^{\alpha} h(s+r,X_r^{s,x})}^{k_3} dr \bigg]\rightarrow 0\label{ftstw}\\
&\mathbb{E}\bigg[\bigg\vert\int_0^1 \partial^{\alpha} h(s+t,\lambda X_t^{s,x'} + (1-\lambda)X_t^{s,x}) d\lambda - \partial^{\alpha} h(s+t, X_t^{s,x}) \bigg\vert^{k_3}\bigg] \rightarrow 0\nonumber\\
&\mathbb{E}\bigg[\int_0^t\abs*{\int_0^1 \partial^{\alpha} h(s+r,\lambda X_r^{s,x'} + (1-\lambda)X_r^{s,x}) d\lambda - \partial^{\alpha} h(s+r, X_r^{s,x}) }^{k_3} dr\bigg] \rightarrow 0\nonumber
\end{align}
as~$x'\rightarrow x$, where the derivatives~$\partial^{\alpha}$ are in the spatial argument and~$g(t,\cdot) = g$.
\end{lemma}
\begin{proof}
For any~$\epsilon>0$,~$s\in[0,T]$ and stopping time~$t\leq T-s$, note that
\begin{equation*}
\mathbb{P}\bigg(\sup_{0\leq u\leq T-s}\abs{X_u^{s,x'} - X_u^{s,x}}\leq\epsilon\bigg)\leq \mathbb{P}(\abs{X_t^{s,x'} - X_t^{s,x}}\leq\epsilon),
\end{equation*}
so that for any~$\lambda\in[0,1]$, by Theorem~1.7 in~\cite{MR1731794}, it holds that~$\lambda X_t^{s,x'} + (1-\lambda)X_t^{s,x} - X_t^{s,x} = \lambda (X_t^{s,x'} - X_t^{s,x}) \rightarrow 0$ in probability as~$x'\rightarrow x$ (sequentially). Therefore for any multiindex~$\alpha$,~$\hat{J} := \partial^{\alpha} h(s+t,\lambda X_t^{s,x'} + (1-\lambda)X_t^{s,x}) - \partial^{\alpha} h(s+t,X_t^{s,x})\rightarrow 0$ in probability by Theorem~20.5 in~\cite{MR2893652}. Moreover, 
by~\eqref{gbylya} and Corollary~\ref{huddecor0}, 
it holds that
\begin{equation*}
\mathbb{E}[\vert\partial^{\alpha} h(s+t,\lambda X_t^{s,x'} + (1-\lambda)X_t^{s,x})\vert^{k_3}]\! \leq C\mathbb{E}[1+V^{s,T}(0,\tilde{x}(x)) + V^{s,T}(0,\tilde{x}(x'))],
\end{equation*}
so that alongside~\eqref{thatof} with~$k_1 = k_3$,~$k = k_2$ or~$k = \hat{k}_2$ and~$S = \abs{\hat{J}}$, one obtains~$\mathbb{E}[\abs{\hat{J}}^{k_3}] \rightarrow 0$ 
as~$x'\rightarrow x$. 
Since~$C$ here is independent of~$t$ and~$\lambda$, Jensen's inequality, Fubini's and dominated convergence theorem concludes the proof.
\end{proof}

\begin{proof}[Proof of Theorem~\ref{fk2}]
For~$x\in O$,~$s\in[0,T]$, stopping time~$t\leq T-s$,
~$\kappa\in\mathbb{R}^n$,~$r\in\mathbb{R}\setminus\{0\}$,~$\abs*{\kappa}=1$, let~$x':=x+r\kappa\in O$ and for~$h\in\{f,c,g\}$, denote
\begin{align*}
h_t' &:= \int_0^1 \nabla h(s+t,\lambda X_t^{s,x'} + (1-\lambda) X_t^{s,x}) d\lambda,\\
\tilde{h}(t,x) &:= h(s+t,X_t^{s,x}),\\ 
c_t^{\dagger} &:=\int_0^1e^{-\lambda\int_0^t c(s+u,X_u^{s,x'})du - (1-\lambda)\int_0^t c(s+u,X_u^{s,x})du}d\lambda,\\
\hat{c}(t,x)&:=e^{-\int_0^t c(s+u,X_u^{s,x}) du}, 
\end{align*}
where~$\nabla$ denotes the gradient in the spatial argument,~$g(s+t,\cdot) = g$ and the same for its derivatives. For~\ref{tstwi}, we show 
directional differentiability. Let~$h\in\{f,g\}$; it holds that
\begin{align}
&\bigg\vert\frac{\mathbb{E}[\tilde{h}(t,x')\hat{c}(t,x')] - \mathbb{E}[\tilde{h}(t,x)\hat{c}(t,x)]}{r} - \mathbb{E}\bigg[\nabla h(s+t,X_t^{s,x}) \cdot X_{t(\kappa)}^{s,x} \hat{c}(t,x)\nonumber\\ 
&\quad - \tilde{h}(t,x)\hat{c}(t,x)\int_0^t\nabla c(s+u,X_u^{s,x})\cdot X_{u(\kappa)}^{s,x}du \bigg] \bigg\vert\nonumber\\
&\qquad\leq \abs*{\frac{\mathbb{E}[\tilde{h}(t,x')\hat{c}(t,x')] - \mathbb{E}[\tilde{h}(t,x)\hat{c}(t,x')]}{r} - \mathbb{E}[h_t' \cdot X_{t(\kappa)}^{s,x} \hat{c}(t,x')]}\nonumber\\
&\qquad \quad + \bigg\vert\frac{\mathbb{E}[\tilde{h}(t,x)\hat{c}(t,x')] - \mathbb{E}[\tilde{h}(t,x)\hat{c}(t,x)]}{r} + \mathbb{E}\bigg[\tilde{h}(t,x) c_t^{\dagger} r^{-1}\bigg(\int_0^t (c(s+u,X_u^{s,x'})\nonumber\\
&\qquad \quad - c(s+u,X_u^{s,x}))du \bigg) \bigg]\bigg\vert + \Big\vert\mathbb{E}[h_t' \cdot X_{t(\kappa)}^{s,x} \hat{c}(t,x)]\nonumber\\
&\qquad \quad - \mathbb{E}[\nabla h(s+t,X_t^{s,x}) \cdot X_{t(\kappa)}^{s,x} \hat{c}(t,x)] \Big\vert\nonumber\\
&\qquad\quad + \bigg\vert\mathbb{E}\bigg[\tilde{h}(t,x)c_t^{\dagger}r^{-1}\bigg(\int_0^t (c(s+u,X_u^{s,x'}) - c(s+u,X_u^{s,x}))du \bigg)\bigg]\nonumber\\
&\qquad\quad - \mathbb{E}\bigg[\tilde{h}(t,x) \hat{c}(t,x) \int_0^t\nabla c(s+u,X_u^{s,x}) \cdot X_{u(\kappa)}^{s,x} du \bigg] \bigg\vert.\label{redun2}
\end{align}
The first three terms on the right-hand side of~\eqref{redun2} converge to~$0$ as~$r\rightarrow 0$ by the fundamental theorem of calculus,~\eqref{gbylya}, Lemma~\ref{firstlem} and Lemma~\ref{lemtstw}. 
For the last term, H\"{o}lder's inequality yields
\begin{align}
&\bigg\vert\mathbb{E}\bigg[\tilde{h}(t,x)c_t^{\dagger}r^{-1}\bigg(\int_0^t (c(s+u,X_u^{s,x'}) - c(s+u,X_u^{s,x}))du \bigg)\bigg]\nonumber\\
&\quad - \mathbb{E}\bigg[\tilde{h}(t,x) \hat{c}(t,x) \int_0^t\nabla c(s+u,X_u^{s,x}) \cdot X_{u(\kappa)}^{s,x} du \bigg] \bigg\vert\nonumber\\
&\qquad\leq \left\|\tilde{h}(t,x)\right\|_{L^{k_2}(\mathbb{P})} \bigg\| c_t^{\dagger} \bigg(\int_0^t \bigg(\frac{c(s+u,X_u^{s,x'}) - c(s+u,X_u^{s,x})}{r} \nonumber\\
&\qquad\quad - \nabla c(s+u,X_u^{s,x}) \cdot X_{u(\kappa)}^{s,x}\bigg) du \bigg)\nonumber\\
&\qquad\quad + (c_t^{\dagger} - \hat{c}(t,x))\int_0^t \nabla c(s+u,X_u^{s,x}) \cdot X_{u(\kappa)}^{s,x}du \bigg\|_{L^{k_2'}(\mathbb{P})} ,\label{refac}
\end{align}
where~$\frac{1}{k_2} + \frac{1}{k_2'} = 1$. By~\eqref{gbylya} and Corollary~\ref{huddecor0}, 
we have~$\mathbb{E}\vert\tilde{h}(t,x)\vert^{k_2} \leq C(1+V^{s,T}(0,\tilde{x}(x)) + V^{s,T}(0,\tilde{x}(x')))$. Moreover, H\"{o}lder's inequality yields
\begin{align}
&\mathbb{E}\bigg[\abs*{(c_t^{\dagger} - \hat{c}(t,x)) \int_0^t\nabla c(s+u,X_u^{s,x})\cdot X_{u(\kappa)}^{s,x} du }^{k_2'}\bigg] \nonumber\\
&\quad\leq \bigg(\mathbb{E}\bigg[\abs*{c_t^{\dagger}-\hat{c}(t,x)}^{2k_2'}\bigg]\bigg)^{\frac{1}{2}} \bigg(\mathbb{E}\bigg[\bigg\vert\int_0^t \nabla c(s+u,X_u^{s,x})\cdot X_{u(\kappa)}^{s,x} du\bigg\vert^{2k_2'}\bigg]\bigg)^{\frac{1}{2}}.\label{etstwfm}
\end{align}
For the first factor on the right-hand side, note that by~\eqref{ftstw} in Lemma~\ref{lemtstw}, we have~$\int_0^t c(s+u, X_u^{s,x'}) du \rightarrow \int_0^t c(s+u,X_u^{s,x}) du$ in probability as~$r\rightarrow 0$, so that
\begin{equation*}
\hat{S}_t := e^{-\lambda\int_0^t (c(s+u, X_u^{s,x'})-(1-\lambda)\int_0^t c(s+u,X_u^{s,x})) du} - e^{-\int_0^t  c(s+u,X_u^{s,x}) du} \rightarrow 0
\end{equation*}
in probability by the continuous mapping theorem and~$\mathbb{E}[\abs{c_t^{\dagger} - \hat{c}(t,x)}^{2k_2'}] \leq \int_0^1 \mathbb{E} [\abs{ \hat{S}_t }^{2k_2'}] d\lambda \rightarrow 0$ as~$r\rightarrow 0$ by~\eqref{thatof} with~$k_1 = 2k_2'$,~$k>2k_2'$ and~$S = \hat{S}_t$. By setting~$K>2k_2'$, the second factor on the right-hand side of~\eqref{etstwfm} is clearly bounded independently of~$r$ (and of~$t$) by H\"{o}lder's inequality, our assumption on the derivatives of~$c$ and Lemma~\ref{firstlem}.\\
For the remaining term in the second factor on the right-hand side of~\eqref{refac}, the triangle inequality on~$L^{k_2'}(\mathbb{P})$ yields
\begin{align}
&\left\|c_t^{\dagger}\int_0^t (r^{-1}(c(s+u,X_u^{s,x'}) - c(s+u,X_u^{s,x})) - \nabla c(s+u,X_u^{s,x}) \cdot X_{u(\kappa)}^{s,x})du \right\|_{L^{k_2'}(\mathbb{P})}\nonumber\\
&\quad\leq \left\| \int_0^t \bigg( c_u' \cdot r^{-1}(X_u^{s,x'} - X_u^{s,x}) - \nabla c(s+u,X_u^{s,x}) \cdot X_{u(\kappa)}^{s,x} \bigg) du \right\|_{L^{k_2'}(\mathbb{P})}\nonumber\\
&\quad\leq \left\| \int_0^t c_u' \cdot (r^{-1}(X_u^{s,x'} - X_u^{s,x}) - X_{u(\kappa)}^{s,x}) du \right\|_{L^{k_2'}(\mathbb{P})}\nonumber\\
&\qquad + \left\|\int_0^t (c_u' - \nabla c(s+u,X_u^{s,x})) \cdot X_{u(\kappa)}^{s,x} du \right\|_{L^{k_2'}(\mathbb{P})}.\label{tfowuh}
\end{align}
For the first term of the right-hand side of~\eqref{tfowuh}, by Jensen's inequality, 
Corollary~\ref{huddecor0}, 
setting~$K>2k_2'$ and our assumption about the derivatives of~$c$, we have
\begin{align}
&\mathbb{E}\bigg[\bigg\vert\int_0^t c_u'\cdot\bigg( \frac{X_u^{s,x'} - X_u^{s,x}}{r} - X_{u(\kappa)}^{s,x}\bigg)du \bigg\vert^{k_2'}\bigg] \nonumber\\
&\quad\leq T^{k_2'-1}\mathbb{E}\bigg[\int_0^{T-s}\bigg\vert c_u'\cdot\bigg( \frac{X_u^{s,x'} - X_u^{s,x}}{r} - X_{u(\kappa)}^{s,x}\bigg)\bigg\vert^{k_2'}du\bigg]\nonumber\\
&\quad\leq T^{k_2'-1} \bigg(\mathbb{E}\bigg[\int_0^{T-s}\!\abs*{c_u'}^{2k_2'} du\bigg]\bigg)^{\!\frac{1}{2}} \!\bigg(\mathbb{E}\bigg[\int_0^{T-s} \bigg\vert\frac{X_u^{s,x+r\kappa}-X_u^{s,x}}{r} - X_{u(\kappa)}^{s,x}\bigg\vert^{2k_2'}\! du\bigg]\bigg)^{\frac{1}{2}}\nonumber\\
&\quad\leq C\big(1 + V^{s,T}(0,\tilde{x}(x')) + V^{s,T}(0,\tilde{x}(x))\big)^{\frac{1}{2}} \nonumber\\
&\qquad \cdot \bigg(\mathbb{E}\bigg[\sup_{0\leq u\leq T-s}\bigg\vert\frac{X_u^{s,x+r\kappa}-X_u^{s,x}}{r} - X_{u(\kappa)}^{s,x}\bigg\vert^{2k_2'}\bigg]\bigg)^{\frac{1}{2}}\label{etstw2}
\end{align}
for~$C$ independent of~$t$, which converges to $0$ as $r\rightarrow 0$ by Lemma~\ref{firstlem}. For the second term on the right-hand side of~\eqref{tfowuh}, it holds that
\begin{align}
&\mathbb{E}\bigg[\bigg\vert\int_0^t (c_u' - \nabla c(s+u,X_u^{s,x})) \cdot X_{u(\kappa)}^{s,x}du\bigg\vert^{k_2'}\bigg]\nonumber\\
&\quad \leq C\bigg(\int_0^{T-s}\mathbb{E}[\abs*{c_u' - \nabla c(s+u,X_u^{s,x})}^{2k_2'}]du \bigg)^{\frac{1}{2}}\bigg(\mathbb{E}\bigg[\int_0^{T-s}\abs{X_{u(\kappa)}^{s,x}}^{2k_2'}du \bigg]\bigg)^{\frac{1}{2}}.\label{fluster}
\end{align}
The last factor in the right-hand side of~\eqref{fluster} is uniformly bounded in~$r$ by Lemma~\ref{firstlem} and the first factor converges to~$0$ as~$r\rightarrow 0$ by Lemma~\ref{lemtstw}.

Putting together the above in~\eqref{redun2} gives that $\mathbb{E}[g(X_t^{s,x})e^{\int_0^t c(s+u,X_u^{s,x})du}]$ is directionally differentiable in $x$. For the other term in the expectation of~\eqref{etstw}, it suffices to check that after integrating the inequality~\eqref{redun2} in~$t$ from~$0$ to~$T-s$, the same convergences hold as~$r\rightarrow 0$. This is true for the first three term on the right-hand side of~\eqref{redun2} by the same reasoning as before. It is true for the right-hand side of~\eqref{refac} by dominated (in~$t$) convergence, since the right-hand sides of~\eqref{etstwfm},~\eqref{etstw2} and~\eqref{fluster} are uniformly bounded in~$t\in[0,T-s]$ and~$r\in[0,\epsilon]$ for some~$\epsilon>0$. By induction and largely the same arguments as above, higher order directional derivatives in~$x$ of~$\mathbb{E}[\tilde{h}(t,x)\hat{c}(t,x)]$ exist and they are sums of expressions of the form
\begin{align}
&\mathbb{E}\bigg[\partial^{\beta_1}h(s+t,X_t^{s,x}) \hat{c}(t,x)\bigg(\prod_{\beta_2\in\hat{I}_2}(\partial^{(\beta_2)}X_t^{s,x})_{j_{\beta_2}}\bigg) \label{fform}\\
&\quad\cdot \prod_{\beta_3\in\hat{I}_3} \int_0^t\partial^{\beta_3} c(s+u,X_u^{s,x}) \prod_{\beta_4\in\hat{I}_{\beta_3}}(\partial^{(\beta_4)}X_u^{s,x})_{j_{\beta_4}} du\bigg],\nonumber
\end{align}
where~$h\in\{f,g\}$,~$\beta_1$ is a multiindex with~$0\leq \abs{\beta_1}\leq p$,~$\hat{I}_2,\hat{I}_3,\hat{I}_{\beta_3}$ are some finite sets of multiindices each with absolute value less than or equal to~$p$ and~$j_{\beta_2},j_{\beta_4}\in\{1,\dots,n\}$. 

For differentiability of the expectation of~\eqref{etstw} in~$x$, 
note that Theorem~1.2 in~\cite{MR1731794} may be applied on~\eqref{firsdd} due to (by Assumption~\ref{A1lya} and the same for~$\sigma$)
\begin{align*}
&\int_0^{T-s}\abs{\partial_i b(s+r,X_t^{s,x})}dr \\
&\quad\leq C\bigg(1+\sum_{i\in I_0}\int_0^{T-s}\log V_i(r,\bar{x}_i(x))dr + \sum_{i'\in I_0'} \log V_{i'}(T-s,\bar{x}_{i'}(x)) \bigg)\\
&\quad\leq C\bigg(1+\sum_{i\in I_0}\int_0^{T-s}V_i(r,\bar{x}_i(x)) dr + \sum_{i'\in I_0'} V_{i'}(T-s,\bar{x}_{i'}(x))\bigg)
\end{align*}
and Lemma~\ref{ergodfin}, so that the derivatives in probability~$X_{t(\kappa)}$ are unique solutions to~\eqref{firsdd} for the initial condition~$\kappa$. Therefore the first directional derivatives from the left-hand side of~\eqref{redun2} indeed form a linear map. 
The same arguments apply for expressions of the form~\eqref{fform} that are directionally differentiable, where additionally Assumption~\ref{A2lya}, Lemma~\ref{firstlem} and Theorem~\ref{seclem} are to be used to control~$K_t(1)$ from Theorem~1.2 in~\cite{MR1731794}. Next, we show continuity in~$x$ of expressions of the form~\eqref{fform} (for multiindices with absolute values bounded by~$p$). 
Note first that~$\mathbb{P}(\sup_{0\leq u\leq T-s}\abs{\partial^{\beta}X_u^{s,x'} - \partial^{\beta}X_u^{s,x}}\leq\epsilon)\leq \mathbb{P}(\abs{\partial^{\beta}X_t^{s,x'} - \partial^{\beta}X_t^{s,x}}\leq\epsilon)$, therefore~$\partial^{\beta}X_t^{s,x}$ is continuous in probability w.r.t. to~$x$ by Theorem~4.10 in~\cite{MR1731794}. Consequently the product w.r.t.~$\beta_2$ in~\eqref{fform} and~$\partial^{\beta_1}h(s+t,X_t^{s,x})$ are sequentially continuous in probability by Theorem~20.5 in~\cite{MR2893652}. Lemma~\ref{lemtstw} and continuous mapping theorem yield that~$\hat{c}(t,x)$ is continuous in probability w.r.t.~$x$. For the remaining factors in~\eqref{fform}, for~$1<k<K$, we have
\begin{align*}
&\int_0^t\bigg\vert\partial^{\beta_3} c(s+u,X_u^{s,x'})\prod_{\beta_4\in\hat{I}_{\beta_3}}(\partial^{(\beta_4)}X_u^{s,x'})_{j_{\beta_4}}\\
&\quad - \partial^{\beta_3} c(s+u,X_u^{s,x})\prod_{\beta_4\in\hat{I}_{\beta_3}}(\partial^{(\beta_4)}X_u^{s,x})_{j_{\beta_4}}\bigg\vert du\\
&\quad \leq\int_0^{T-s} \bigg\vert\partial^{\beta_3} (c(s+u,X_u^{s,x'}) - c(s+u,X_u^{s,x})) \prod_{\beta_4\in\hat{I}_{\beta_3}}(\partial^{(\beta_4)}X_u^{s,x'})_{j_{\beta_4}}\bigg\vert du\\ 
&\qquad + \int_0^{T-s} \bigg\vert\partial^{\beta_3} c(s+u,X_u^{s,x})\prod_{\beta_4\in\hat{I}_{\beta_3}}(\partial^{(\beta_4)}(X_u^{s,x'} - X_u^{s,x}))_{j_{\beta_4}}\bigg\vert du\\
&\quad\leq \int_0^{T-s}\bigg\vert\partial^{\beta_3} (c(s+u,X_u^{s,x'}) - c(s+u,X_u^{s,x}))\bigg\vert du\!\! \prod_{\beta_4\in\hat{I}_{\beta_3}}\!\!\sup_{0\leq u\leq T-s} \abs*{\partial^{(\beta_4)}X_u^{s,x'} }\\
&\qquad + \int_0^{T-s}\abs{\partial^{\beta_3}c(s+u,X_u^{s,x})}du \prod_{\beta_4\in\hat{I}_{\beta_3}}\sup_{0\leq u\leq T-s} \abs*{\partial^{(\beta_4)}X_u^{s,x'}-\partial^{(\beta_4)}X_u^{s,x} }
\end{align*}
By H\"older's inequality, Lemma~\ref{firstlem} and Theorem~\ref{seclem}, the first term on the right-hand side converges to zero in mean, therefore to zero in probability, as~$x'\rightarrow x$. By Theorem~4.10 in~\cite{MR1731794} (and the continuous mapping theorem), the second term on the right-hand side also converges to zero in probability. Therefore the left-hand side converges to zero in probability. By the continuous mapping theorem, the term inside the square bracket in~\eqref{fform} is sequentially continuous in probability. Consequently, by~\eqref{thatof} with~$k_1 = 1$,~$k = \frac{1+k_2}{2}$,~$S = 
\abs{J(x')-J(x)}$, where~$J(x)$ is equal to the term inside the square brackets in~\eqref{fform}, together with H\"older's inequality, inequality~\eqref{gbylya}, our assumption on the derivatives of~$c$ with a large enough~$K$, 
Corollary~\ref{huddecor0}, 
Lemma~\ref{firstlem} and Theorem~\ref{seclem}, expectations of the form~\eqref{fform} are continuous functions w.r.t.~$x$ and so are their integrals in~$t$ by dominated convergence, which concludes the proof for~\ref{tstwi}.

Using the same arguments and denoting the expression~\eqref{fform} by~$\hat{u}$, it holds that 
\begin{align*}
\hat{u} &\leq C(1+V^{s,T}(0,\tilde{x}(x)))^{\frac{1}{k_2}}\bigg(\mathbb{E}\bigg[\sup_{0\leq u\leq T-s}\abs{\partial^{(\beta_2)}X_u^{s,x}}^{2k_2'}\bigg]\bigg)^{\frac{1}{2k_2'}}\\
&\quad\cdot  \prod_{(\beta_3,\beta_4)\in\hat{I}}(1+V^{s,T}(0,\tilde{x}(x)))^{\frac{1}{c_{\beta_3}}} \bigg(\mathbb{E}\bigg[\sup_{0\leq u\leq T-s}\abs{\partial^{(\beta_4)}X_u^{s,x}}^{c_{\beta_4}}\bigg]\bigg)^{\frac{1}{c_{\beta_4}}}
\end{align*}
for some~$c_{\beta_3},c_{\beta_4}>0$,~$\beta_3,\beta_4\in\hat{I}$ and in particular for some constant~$C$ independent of~$t$. The proof for~\ref{tstwii} then concludes by Theorem~\ref{seclem}.

Assertion~\ref{tstwiii} then follows by Theorem~3.5(iii) in~\cite{MR1731794}, Lemma~\ref{tstwalem} and by noting that~$C$ above is independent of~$s$ given the Lyapunov functions are 
local in~$s$.
\end{proof}

\subsection{Twice spatially differentiable solutions}\label{tdkol}

In this section, we prove that the expectation of~\eqref{etstw} with~$t=T-s$ solves a Kolmogorov equation by the approach in~\cite{MR1731794}. The main ingredient beside differentiability of the associated semigroups, given in Theorem~\ref{fk2}, is that the SDE can be approximated in probability by an Euler-type approximation locally uniformly in initial time and space, which is given in Lemma~\ref{eulercan}. Throughout this section, we assume~$O = \mathbb{R}^n$ and as before that~$b,\sigma$ are nonrandom (this is enforced by Assumption~\ref{A3lya}).


\begin{lemma}\label{eulercan}
Let Assumption~\ref{A3lya} hold. 
For~$I = \{t_k\}_{k\in\mathbb{N}_0}\subset [0,\infty)$ with~$t_0 = 0$,~$t_{k+1}\geq t_k$,~$k\in\mathbb{N}$,~$t_k\rightarrow \infty$ as~$k\rightarrow \infty$,~$\sup_{k\geq 0} t_{k+1} - t_k <\infty$, $s\in[0,\infty)$, $x\in\mathbb{R}^n$, let~$X_t^{s,x}(I)$ denote the Euler approximation given by~$X_0^{s,x}(I) = x$ and
\begin{equation}\label{euap}
X_t^{s,x}(I) = X_{t_k}^{s,x}(I) + \int_{t_k}^t b(s+r,X_{t_k}^{s,x}(I)) dr + \int_{t_k}^t\sigma(s+r,X_{t_k}^{s,x}(I))dW_r, 
\end{equation}
on $t\in [t_k,t_{k+1}]$, $k\in\mathbb{N}$. For any $R',T' \geq 0$, $\epsilon>0$, it holds that
\begin{equation*}
\sup_{s\in[0,T']}\sup_{\abs*{x}\leq R'} \mathbb{P}\bigg[\sup_{t\in [0,T']} \abs*{X_t^{s,x} - X_t^{s,x}(I)}\geq \epsilon \bigg]\rightarrow 0
\end{equation*}
as $\sup_{k\geq 0} t_{k+1} - t_k \rightarrow 0$.
\end{lemma}
\begin{proof}
We extend the proof of Theorem~1 in~\cite{MR4074703} to obtain convergence that is uniform with respect to~$s\in[0,T]$ and~$x\in B_R$. Fix the numbers~$R',T'\geq 0$. 
For~$k\in\mathbb{N}$, let~$\varphi_k:\mathbb{R}^n\rightarrow [0,\infty)$ be smooth cutoff functions satisfying~$\varphi_k(x) = 1$ for~$x\in B_k$,~$\varphi_k(x) = 0$ for~$x\in \mathbb{R}^n\setminus B_{k+1}$ and let~$b^{(k)}:[0,\infty)\times \mathbb{R}^n\rightarrow\mathbb{R}^n$,~$\sigma^{(k)}:[0,\infty)\times \mathbb{R}^n\rightarrow\mathbb{R}^{n\times n}$ be given by~$b^{(k)} = b\varphi_k$ and~$\sigma^{(k)} = \sigma\varphi_k$.
Let~$Y_t^{s,x,k}(I)$ be the unique solutions to the corresponding SDE with drift~$b^{(k)}$ and diffusion coefficient~$\sigma^{(k)}$. The corresponding Euler approximation is given by~\eqref{euap} with~$Y_0^{s,x,k} = Y_0^{s,x,k}(I) = x$. Fix w.l.o.g.~$0< \epsilon \leq 1$. In the same way as in the proof of Theorem~1 in~\cite{MR4074703}, one obtains that for any~$s\in [0,T']$,~$x\in \mathbb{R}^n$ and~$k\geq R'+1$,
\begin{align*}
\mathbb{P}\bigg(\sup_{0\leq t\leq T'} \abs*{X_t^{s,x} - X_t^{s,x}(I)}> \epsilon \bigg)  &\leq \mathbb{P}\bigg(\sup_{0\leq t\leq T'}\abs*{Y_t^{s,x,k} - Y_t^{s,x,k}(I)} > \epsilon \bigg)\\
&\quad + \mathbb{P}(\tau_{k-1} \leq T'),
\end{align*}
where~$\tau_{k-1} = \inf\{t\geq 0: \abs*{X_t^{s,x}} > k-1\}$. By Markov's inequality, (the arguments of) Corollary~\ref{huddecor0} 
and Assumption~\ref{A3lya}\ref{Ginf}
, it holds 
that
\begin{align*}
&\mathbb{P}(\tau_{k-1} \leq T')\inf_{t\in[s,s+T'],\abs*{y}=k-1}\hat{G}(t,y)\\
&\quad\leq \mathbb{E}[\hat{G}(s+(\tau_{k-1}\wedge T'),X_{\tau_{k-1}\wedge T'}^{s,x})]\\
&\quad\leq \mathbb{E}[V^{s,T'}(\tau_{k-1}\wedge T',\tilde{x}(x))]\\
&\quad\leq \Big\|e^{\int_0^{(\tau_{k-1}\wedge T')} \alpha_u^{s,T'} du}\Big\|_{L^{\frac{p^{s,T'}}{p^{s,T'}-1}}(\mathbb{P})} \\
&\qquad\cdot\bigg(  V_0(s,\tilde{x}(x)) + \int_0^{T'} \bigg\|\frac{\mathds{1}_{[0,\tau_{k-1}\wedge T')} (v) \beta_v^{s,T'}}{e^{\int_0^v \alpha_u^{s,T'} du}}\bigg\|_{L^{p^{s,T'}}(\mathbb{P})} dv\bigg).
\end{align*}
For any~$0<\epsilon'<1$, by the assumption that~$V^{s,T'}$ is local in~$s$  
and continuity of~$V_0$, there exists~$k^*$ such that $\mathbb{P}(\tau_{k^*-1}\leq T') \leq \frac{\epsilon'}{2}$ for all~$s\in[0,T']$ and~$x\in B_{R'}$. In addition, for any~$R>0$, it holds that 
\begin{align*}
&2\langle x-y, b^{(k^*)}(t,x) - b^{(k^*)}(t,y) \rangle + \| \sigma^{(k^*)}(t,x) - \sigma^{(k^*)}(t,y)  \|^2\\
&\quad \leq 2\langle x-y, b(t,x) - b(t,y) \rangle \varphi_{k^*}(x) + 2\abs*{b(t,y)}\abs*{x-y} \abs*{\varphi_{k^*}(x) - \varphi_{k^*}(y)}\\
&\qquad + \| \sigma(t,x) - \sigma(t,y) \|^2 \varphi_{k^*}(x)^2 + \|\sigma(t,y)\|^2 \abs*{\varphi_{k^*}(x) - \varphi_{k^*}(y)}^2\\
&\quad \leq (K_t(R)+C\sup_{y'\in B_R}(\abs*{b(t,y')} + \|\sigma(t,y')\|^2))\abs*{x-y}^2 
\end{align*}
for all~$x,y\in B_R$ and
\begin{equation*}
2\langle x,b^{(k^*)}(t,x) \rangle + \|\sigma^{(k^*)}(t,x)\|^2 \leq 2(1+\abs*{x})\sup_{x'\in B_{k^*+1}}(\abs*{b(t,x')} + \|\sigma(t,x')\|^2)
\end{equation*}
for all~$x\in\mathbb{R}^n$. Therefore Corollary~5.4 in~\cite{MR1731794} can be applied to obtain
\begin{equation*}
\sup_{s\in[0,T']}\sup_{x\in B_R}\mathbb{P}\bigg(\sup_{0\leq t\leq T'}\abs*{Y_t^{s,x,k^*} - Y_t^{s,x,k^*}(I)} > \epsilon \bigg) \rightarrow 0
\end{equation*}
as~$\sup_{k\geq 0} t_{k+1} - t_k \rightarrow 0$, which concludes the proof.
\end{proof}

\begin{theorem}\label{fk}
Let the assumptions of Theorem~\ref{fk2} hold with~$p\geq 2$ and let Assumption~\ref{A3lya} hold. Let~$f,c,g$ 
satisfy Assumption~\ref{Afcg}. There exists~$K>1$ such that if~$c$ has~$(p,k')$-Lyapunov derivatives for any~$1<k'<K$ and the families of Lyapunov functions in Assumptions~\ref{A1lya},~\ref{A2lya},~\ref{Afcg} and for the Lyapunov derivatives of~$c$ are local in~$s$, then 
for~$v:[0,T]\times\mathbb{R}^n\rightarrow\mathbb{R}$ given by
\begin{equation}\label{actualetstw}
v(t,x) = \mathbb{E}[u(t,T-t,x)],
\end{equation}
with~$u$ as in~\eqref{etstw}, the equation~\eqref{asfkeq0} 
holds almost everywhere in $(0,T)\times\mathbb{R}^n$.
\end{theorem}


\begin{proof}
Theorem~\ref{fk2}, Theorem~3.6 in~\cite{MR1731794} applied on the SDE~\eqref{sde0} appended by~\eqref{addingonapp} and Lemma~\ref{tstwalem} yield~$(\partial_t v  + b \cdot \nabla v + a: D^2 v - c v + f)e^{-x'} = 0$ almost everywhere.
\end{proof}
The assumptions in Theorems~\ref{fk2} and~\ref{fk} remain strictly weaker than those in~\cite[Lemma~5.10]{MR1731794}, since Lyapunov functions that are positive polynomials can easily be obtained under the global Lipschitz conditions there. 

Alternative to Theorem~\ref{fk}, 
under slightly stronger assumptions, we may use the approach as in~\cite[Theorem~1.6.2]{MR1840644},~\cite[Theorem~5.7.6]{MR1121940} in order to obtain unique classical solutions to Kolmogorov equations.
\begin{theorem}\label{classical}
Let the assumptions of Theorem~\ref{fk2} hold with~$p\geq 2$ and let Assumption~\ref{A3lya} hold. 
Assume~$b,\sigma$ are independent of~$t$, so that it holds~$\mathbb{P}$-a.s. that~$X_{\cdot}^{s,x} = X_{\cdot}^{0,x} =X_{\cdot}^x $ for all~$s$. 
Let~$f,c,g$ satisfy Assumption~\ref{Afcg} and be continuous. 
There exists~$\bar{K}>1$ such that if 
\begin{enumerate}[label=(\roman*)]
\item for any~$k'\in(1,\bar{K}]$,~$f,c,g$ have~$(p,k')$-Lyapunov derivatives,
\item\label{sepo} for any family of Lyapunov functions~$(\tilde{V}^{s,T})_{s\in[0,T]}$ and corresponding mappings~$(\hat{x})_{s\in[0,T]}$ in Assumptions~\ref{A1lya},~\ref{A2lya} and 
Definition~\ref{lyadiff}, it holds that~$(\tilde{V}^{s,T})_s$ is local in~$s$ (as in Definition~\ref{locdef}) and there exists a constant~$C>0$ such that for any~$s$, it holds~$\mathbb{P}$-a.s. that~$\tilde{V}^{s+t,T}(0,\hat{x}(X_t^x))\leq C(1+\tilde{V}^{s,T}(t,\hat{x}(x)))$ for all~$x\in\mathbb{R}^n$,~$t\in[0,T-s]$,
\end{enumerate}
then the function~$v:[0,T]\times\mathbb{R}^n\rightarrow\mathbb{R}$ given by~\eqref{actualetstw} and~\eqref{etstw} 
is the unique classical solution to~\eqref{asfkeq0} on~$[0,T]\times\mathbb{R}^n$, in the sense that~$v\in C^{1,2}$,~$\partial_t v$,~$\nabla\!_x v$,~$D_x^2v$ are continuous,~$v$ satisfies~\eqref{asfkeq0} and it is the only such function satisfying~$v(T,\cdot)=g$. 
\end{theorem}
\begin{proof}
For any~$s\geq 0$,~$x\in\mathbb{R}^n$,~$x',x''\in\mathbb{R}$, consider the solutions~$X_t^x$ to~\eqref{sde} appended with the corresponding~$\mathbb{R}$-valued solutions~$X_t^{(n+1),s,x'}$ and~$X_t^{(n+2),s,x''}$ to~\eqref{addingonapp} 
on~$[0,T]$, denoted~$\bar{X}_t^{s,y}=(X_t^x,X_t^{(n+1),s,x'},X_t^{(n+2),s,x''})$ for~$y=(x,x',x'')$.
Let~$\bar{g}:\mathbb{R}^{n+2}\rightarrow\mathbb{R}$ be given by~$\bar{g}(x,x',x'')=x'' + g(x) e^{-x'}$. 
By Lemma~\ref{regsde}, the joint equations~\eqref{sde},~\eqref{addingonapp} and their solutions are regular (Definition~2.1 in~\cite{MR1731794}). Therefore the Markov property as in Theorem~2.13 in~\cite{MR1731794} applies. In particular, since~$\mathbb{E}[\abs{\bar{g}(\bar{X}_{T-s}^{s,y})}]<\infty$ holds by our assumptions on~$f,c,g$, 
it holds by the usual decomposition into positive and negative parts that
\begin{equation}\label{markov0}
\mathbb{E}[\bar{g}(\bar{X}_{T-s}^{s,y})] = \int \int \bar{g}\Big(\bar{X}_{T-s-r}^{s+r,\bar{X}_r^{s,y}(\omega)}(\omega')\Big) d\mathbb{P}(\omega')d\mathbb{P}(\omega) = \mathbb{E}[(\mathbb{E}[\bar{g}(\bar{X}_{T-s-r}^{s+r,\cdot})])(\bar{X}_r^{s,y})]
\end{equation}
for all~$s\in[0,T]$,~$r\in[0,T-s]$. On the other hand, by Theorem~\ref{fk2}\ref{tstwi} and~$p\geq 2$,~$\mathbb{R}^{n+2}\ni y\mapsto\mathbb{E}[\bar{g}(\bar{X}_{T-s}^{s,y})]$ is twice continuously differentiable. Therefore
by It\^{o}'s lemma and~\eqref{markov0}, it holds for any~$s\in[0,T]$,~$h>0$ with~$s-h\in[0,T]$,~$x\in\mathbb{R}^n$,~$y=(x,0,0)$ that
\begin{align}
&h^{-1}\big(\mathbb{E}\big[\bar{g}\big(\bar{X}_{T-s+h}^{s-h,y}\big)\big] - \mathbb{E}\big[\bar{g}\big(\bar{X}_{T-s}^{s,y}\big)\big]\big)\nonumber\\
&\quad = h^{-1}\big(\mathbb{E}\big[\big(\mathbb{E}\big[\bar{g}\big(\bar{X}_{T-s}^{s,\cdot}\big)\big]\big)\big(\bar{X}_h^{s-h,y}\big)\big] - \mathbb{E}\big[\bar{g}\big(\bar{X}_{T-s}^{s,y}\big)\big]\big)\nonumber\\
&\quad= \mathbb{E}\bigg[\frac{1}{h}\int_0^h \bigg((b\cdot\nabla\!_x \mathbb{E}[\bar{g}(\bar{X}_{T-s}^{s,\cdot})])(\bar{X}_r^{s-h,y}) + \frac{1}{2}\Big(\Big(\sigma\sigma^{\top}\Big)\!\!:\!D_x^2\mathbb{E}[\bar{g}(\bar{X}_{T-s}^{s,\cdot})]\Big)(\bar{X}_r^{s-h,y}) \nonumber\\
&\qquad - c(s-h+r,X_r^x) \Big((\mathbb{E}[\bar{g}(\bar{X}_{T-s}^{s,\cdot})])(\bar{X}_r^{s-h,y}) - X_r^{(n+2),s-h,0}\Big)\nonumber\\
&\qquad+ f(s-h+r,X_r^x)e^{-X_r^{(n+1),s-h,0}}\bigg)dr\bigg].\label{markov1}
\end{align}
We show that the right-hand side of~\eqref{markov1} is well-behaved as~$h\rightarrow0$, that is, it converges to
\begin{equation}\label{bslim}
\bigg(b(x)\cdot\nabla\!_x + \frac{1}{2}\Big(\Big(\sigma\sigma^{\top}\Big)(x)\Big):D_x^2 -c(s,x)\bigg)\mathbb{E}[\bar{g}(\bar{X}_{T-s}^{s,y})] + f(s,x).
\end{equation}
For this, H\"{o}lder's inequality and dominated convergence theorem may be used along with Theorem~\ref{fk2}\ref{tstwii}. 
In order to obtain a good enough bound from~\eqref{tstweq}, set~$\bar{K} = 2\max(k_2,K)\textrm{degree}(q^*)$, where~$k_2,K$ are constants from Theorem~\ref{fk2} and~$q^*$ is the polynomial in Theorem~\ref{fk2}\ref{tstwii}. 
Since~$f,c,g$ have~$(p,k')$-Lyapunov derivatives for all~$k'\in(1,\bar{K})$ by assumption, for any~$h\in\{f,c,g\}$ and~$k'\in\{2k_2\textrm{degree}(q^*),2K\textrm{degree}(q^*)\}$, there exist a Lyapunov function~$V$, locally bounded~$\tilde{x}$ and constant~$N>0$ satisfying that for any multiindices~$\alpha$ with~$0\leq \abs{\alpha}\leq p$, it holds~$\mathbb{P}$-a.s. that
\begin{align*}
&\abs{\partial^{\alpha}h(\lambda X_t^x+(1-\lambda)X_t^{x'})}\\
&\quad\leq N(1+V(t,\tilde{x}(x)) + V(t,\tilde{x}(x')))^{\frac{1}{k'}} \\
&\quad\leq 3N\max\Big(1,(V(t,\tilde{x}(x)))^{\frac{1}{2\textrm{degree}(q^*)}},(V(t,\tilde{x}(x')))^{\frac{1}{2\textrm{degree}(q^*)}}\Big)^{\frac{2\textrm{degree}(q^*)}{k'}}\\
&\quad\leq 3N\Big(1+(1+V(t,\tilde{x}(x)))^{\frac{1}{2\textrm{degree}(q^*)}} + (1+V(t,\tilde{x}(x')))^{\frac{1}{2\textrm{degree}(q^*)}}\Big)^{\frac{2\textrm{degree}(q^*)}{k'}}
\end{align*}
for all~$t\in[0,T]$,~$x,x'\in\mathbb{R}^n$ and~$\lambda\in[0,1]$. 
In particular,~$(1+V)^{1/(2\textrm{degree}(q^*))}$ is itself a Lyapunov function (more precisely it is a~$(\tilde{b}_{\cdot}^{\cdot},\tilde{\sigma}_{\cdot}^{\cdot},\alpha_{\cdot},\beta_{\cdot},p^*,(1+V_0)^{1/(2\textrm{degree}(q^*))})$-Lyapunov function if~$V$
is a~$(\tilde{b}_{\cdot}^{\cdot},\tilde{\sigma}_{\cdot}^{\cdot},\alpha_{\cdot},\beta_{\cdot},p^*,V_0)$-Lyapunov function) that satisfies the conditions in the Definition~\ref{lyadiff} for the~$(p,k'/(2\textrm{degree}(q^*)))$-Lyapunov derivatives of~$h$. Therefore, we may consider~$(1+V)^{1/(2\textrm{degree}(q^*))}$ to be equal to~$V^{s,T}$ appearing in~\eqref{tstweq}. Similar statements can be made for~$V_i,\hat{V}_{l_i}^{s,T}$ in~\eqref{tstweq}. With such Lyapunov functions in~\eqref{tstweq}, the aforementioned program using dominated convergence theorem for the right-hand side of~\eqref{markov1} may be carried out. More specifically, we may apply Proposition~\ref{huddethm}\ref{hp2} and the arguments in the proof of Corollary~\ref{huddecor0} for the dominating function (the supremum in~$r$ of the integrand on the right-hand side of~\eqref{markov1}), to obtain the limit~\eqref{bslim} as~$h\rightarrow 0$. By Theorem~\ref{fk2}\ref{tstwiii} and an induction argument for the derivatives in~$x$, the limit~\eqref{bslim} is continuous in~$s$. 
In particular, for any~$x$, the function~$[0,T]\ni s\mapsto\mathbb{E}[\bar{g}(\bar{X}_{T-s}^{s,(x,0,0)})]$ is continuous and has continuous left-derivative which converges to a limit as~$s\rightarrow0$. Therefore it is continuously differentiable, 
of course with derivative given by~\eqref{bslim}. By (joint) continuity of~\eqref{bslim},~$(s,x)\mapsto\mathbb{E}[\bar{g}(\bar{X}_{T-s}^{s,(x,0,0)})]$ is a classical solution to~\eqref{asfkeq0} on~$[0,T]\times\mathbb{R}^n$. \\
For uniqueness, let~$\hat{v}:[0,T]\times\mathbb{R}^n$ be a classical solution to~\eqref{asfkeq0} with~$\hat{v}(T,\cdot)=g$. 
Applying It\^{o}'s rule 
and taking expectations, it holds for~$(t,x)\in[0,T]\times\mathbb{R}^n$ that
\begin{align*}
&\mathbb{E}\bigg[\hat{v}(T,X_{T-t}^x)\exp\bigg(-\int_0^{T-t} c(t+r,X_r^x)dr \bigg)\bigg] -\hat{v}(t,x)\\
&\quad= \mathbb{E}\bigg[\int_0^{T-t}(\partial_t + L-c)\hat{v}(t+s,X_s^x)\exp\bigg(-\int_0^s c(t+r,X_r^x)dr \bigg)ds\bigg]\\
&\quad= \mathbb{E}\bigg[\int_0^{T-t}-f(t+s,X_s^x)\exp\bigg(-\int_0^s c(t+r,X_r^x)dr \bigg)ds\bigg],
\end{align*}
which concludes after substituting~$\hat{v}(T,\cdot)=g$.
\end{proof}

\section{Alternative assumptions for time-independent, nonrandom coefficients}\label{altassumps}

In the following, we restrict to the case where~$b$ and~$\sigma$ are nonrandom and time-independent, so that we may use Theorem~V.39 in~\cite{MR2273672} in order to get rid of the need for bounds on function values on line segments in terms of the endpoint values. In particular, more local conditions in place of~\eqref{a3},~\eqref{a4},~\eqref{intermid} and~\eqref{gbylya} are obtained. These conditions are stated precisely after first giving Lemma~\ref{classprob} where we use the aforementioned reference. 

\begin{lemma}\label{classprob}
Let~$p\in\mathbb{N}$,~$b,\sigma$ be independent of~$\omega,t$ and suppose they are continuously differentiable up to order~$p$ with locally Lipschitz derivatives. 
There exists~$\Omega\times\Delta_T\times\mathbb{R}^n\ni (\omega,t,x)\mapsto \hat{X}_t^x\in\mathbb{R}^n$ that is~$\mathbb{P}$-a.s. continuously differentiable in~$x$ up to order~$p$ and is for any~$x$ indistinguishable from the corresponding derivatives in probability of~$X_{\cdot}^x$.
\end{lemma}
\begin{proof}
By Theorem~V.38 and~V.39 in~\cite{MR2273672}, continuously differentiable~$\hat{X}_t^x$  up to order~$p$ exists. Moreover, it satisfies~\eqref{sde} and
~$\hat{X}_{\cdot}^x$ is indistinguishable from~$X_{\cdot}^x$. 
The partial derivatives of~$\hat{X}_{\cdot}^x$ satisfy the systems given by formal differentiation of~\eqref{sde}. 
On the other hand, derivatives in probability of~$X_t^x$ as in~\cite[Theorem~4.10]{MR1731794} and Theorem~3.3 above satisfy the same system. 
Therefore by uniqueness in the aforementioned references\footnote{Alternatively, since these systems have terms on right-hand sides that are continuous functions of the partial derivatives and are in particular at most linear in the highest order derivative (see the beginning of proof for Theorem~\ref{seclem}), uniqueness holds by continuity of~$X_t^x$ in~$t$,~\eqref{firncon} in Lemma~\ref{firstlem}, induction in the number of derivatives and Theorem~1.2 in~\cite{MR1731794} with~$K_t(R) = K_t(1)$ constant in~$t$.}, it holds that~$\partial^{\alpha}\hat{X}_\cdot^x$ are the unique solutions to their respective systems for all time and are therefore indistinguishable from the corresponding derivatives in probability~$\partial^{(\kappa_{\alpha})}X_{\cdot}^x$ for every~$x$ and multiindex~$\alpha$ with~$0\leq\abs{\alpha}\leq p$. 
\end{proof}

The precise assumptions and conditions considered in this section in place of Assumptions~\ref{A1lya},~\ref{A2lya} and Definition~\ref{lyadiff} are as follows.
\begin{assumption}\label{A1prime}
The functions~$b,\sigma$ are independent of~$\omega,t$, they admit locally Lipschitz first derivatives and it holds that~$O=\mathbb{R}^n$. There exist continuous~$G:\mathbb{R}^n\rightarrow[0,\infty)$ such that for any~$s\in[0,T]$, there exist finite sets~$I_0,I_0'\subset\mathbb{N}$,~$\tilde{n}_i\in\mathbb{N}$, open~$\tilde{O}_i\subseteq\mathbb{R}^{\tilde{n}_i}$ for all~$i\in I_0\cup I_0'$, locally bounded mappings~$M:(0,\infty)\rightarrow(0,\infty)$,~$(\bar{x}_i:\mathbb{R}^n\rightarrow\tilde{O}_i)_{i\in I_0\cup I_0'}$ and Lyapunov functions~$(V_i:\Omega\times[0,T]\times\tilde{O}_i\rightarrow(0,\infty))_{i\in I_0\cup I_0'}$ satisfying that
\begin{equation}\label{a3a4rep}
\sum_i(\abs{\partial_ib(x)} + \|\partial_i\sigma(x)\|^2) \leq G(x)
\end{equation}
for all~$x\in\mathbb{R}^n$ and for any~$m>0$,~$x\in\mathbb{R}^n$, stopping times~$t\leq T-s$, inequality~\eqref{a5} holds~$\mathbb{P}$-almost surely.
\end{assumption}
Assumption~\ref{A1prime} will be used in Lemma~\ref{alters1} and Theorem~\ref{alters2},~\ref{alters3},~\ref{alters4} as a replacement for Assumption~\ref{A1lya}. The only substantial difference in this Assumption~\ref{A1prime} is that~\eqref{a3},~\eqref{a4} are replaced by the bound~\eqref{a3a4rep} above. Moreover, the following Assumption~\ref{A2prime} will used in place of Assumption~\ref{A2lya}.
\begin{assumption}\label{A2prime}
There exists~$p\in\mathbb{N}_0$ such that~$b,\sigma\in C^p$. In addition, for all~$s\in[0,T]$ and~$k\geq 2$, there exist~$\hat{n}_k\in\mathbb{N}$, open~$\hat{O}_k\subset\mathbb{R}^{\hat{n}_k}$, a mapping~$\hat{x}_k:O\rightarrow\hat{O}_k$, a constant (in particular in~$s$)~$M'>0$ and Lyapunov function~$\hat{V}_k^{s,T}:\Omega\times[0,T-s]\times\hat{O}_k\rightarrow(0,\infty)$ satisfying for any~$x,x'\in O$ and multiindices~$\alpha$ with~$2\leq \abs{\alpha}\leq p$ that it holds~$\mathbb{P}$-a.s. that
\begin{equation*}
\abs*{\partial^{\alpha}b(X_t^{s,x})} + \|\partial^{\alpha}\sigma (X_t^{s,x})\|^2 \leq M' (1+\hat{V}_k^{s,T}(t,\hat{x}_k(x)))^{\frac{1}{k}}
\end{equation*}
for all~$t\in[0,T-s]$.
\end{assumption}
The difference between Assumptions~\ref{A2lya} and~\ref{A2prime} is that~\eqref{intermid} is replaced by the inequality above. Lastly, the restriction for~$f,c,g$ to have Lyapunov derivatives (as in Definition~\ref{lyadiff}) may be relaxed. In Theorems~\ref{alters3},~\ref{alters4}, they will only be required to be Lyapunov derivatives in the following sense.
\begin{definition}\label{lyadiffprime}
For~$p\in \mathbb{N}$,~$k>1$,~$h:\Omega\times[0,T]\times\hat{O}\rightarrow \mathbb{R}$ satisfying~$\mathbb{P}$-a.s. that~$h(t,\cdot)\in C^p(\hat{O})$ for all~$t\in[0,T]$, we say that~$h$ has~$(p,k)'$-Lyapunov derivatives if there exist~$(V^{s,T})_{s\in[0,T]}$ local in~$s$ (as in Definition~\ref{locdef}), locally bounded~$\tilde{x}$ and constant~$N>0$ such that for any~$s\in[0,T]$ and multiindices~$\alpha$ with~$0\leq \abs{\alpha}\leq p$, it holds~$\mathbb{P}$-a.s. that
\begin{equation}
\abs{\partial^{\alpha}h(s+t,X_t^{s,x})} \leq N(1+V^{s,T}(t,\tilde{x}(x)))^{\frac{1}{k}}
\end{equation}
for all stopping times~$t\leq T-s$,~$x\in O$ and~$\lambda\in[0,1]$.
\end{definition}
The main results of this section are stated as Theorems~\ref{alters1},~\ref{alters2},~\ref{alters3},~\ref{alters4} as follows.
\begin{theorem}[Alternative assumptions to Theorem~\ref{seclem}
]\label{alters1}
Under Assumptions~\ref{A1prime},~\ref{A2prime}, for any~$s\in[0,T]$, constants~$1\leq l\leq p-1$,~$k_1>0$, there exists~$i^*\in\mathbb{N}$,~$\nu\geq\frac{k_1}{2}$,~$\{l_i\}_{i\in \{1,\dots,i^*\}} \subset (0,\infty)$ and a finite order polynomial~$q_0$, the degree of which is independent of~$s,V_i,\hat{V}_k^{s,T}$, such that~\eqref{kappar1},~\eqref{kappar2},~\eqref{kappar3} hold for all~$x\in\mathbb{R}^n$,~$T_0\in[s,T]$,~$r\in\mathbb{R}\setminus\{0\}$,~$\kappa_i\in\mathbb{R}^n$,~$\abs{\kappa_i}=1$,~$1\leq i\leq l+1$,~$x+r\kappa_{l+1}\in\mathbb{R}^n$, where~$\kappa = (\kappa_i)_{1\leq i\leq l}$,~$\bar{\kappa}=(\kappa_i)_{1\leq i\leq l+1}$ and~$q:\mathbb{R}^n\times\mathbb{R}^n\rightarrow\mathbb{R}$ is given by
\begin{align*}
q(y,y') &= \mathbb{E}\bigg[ q_0\bigg(\bigg(\int_0^1V_i(0,\bar{x}_i(\lambda y + (1-\lambda)y'))d\lambda\bigg)_{i\in I_0\cup I_0'} , \\
&\qquad\int_0^1\hat{V}_{l_1}^{s,T}(0,\hat{x}_{l_1}(\lambda y + (1-\lambda)y'))d\lambda ,(V_i(0,\bar{x}_i(y)))_{i\in I_0\cup I_0'}, \\
&\qquad (\hat{V}_{l_i}^{s,T}(0,\hat{x}_{l_i}(y)))_{i\in\{2,\dots, i^*\}} ,(\hat{V}_{l_i}^{s,T}(0,\hat{x}_{l_i}(y')))_{i\in\{2,\dots,i^*\}}\bigg)\bigg].
\end{align*}
If~$V_i$ and~$\hat{V}_k^{s,T}$ are local in~$s$ for every~$i,k$, then the form of the polynomial~$q_0$ is independent of~$s$.
\end{theorem}
\begin{proof}
The proof strategy is more or less the same as in the one used in the previous proof for Theorem~\ref{seclem}. The difference is encapsulated by following along the proof of Lemma~\ref{firstlem} using the same notation as before. Note first~\eqref{firncon} follows unperturbed. 
By Lemma~\ref{classprob}, classical derivatives are indistinguishable from derivatives in probability and we use the properties of both without changing the notation in the following. 
As consequence and in place of~\eqref{ipo}, it holds that
\begin{equation*}
dX_{t(\kappa)}^{(r)} = r\int_0^1 ( X_{t(\kappa)}^{x+\lambda r\kappa}\cdot\nabla) b(X_t^{x+\lambda r\kappa}) d\lambda dt + r\int_0^1(X_{t(\kappa)}^{x+\lambda r\kappa} \cdot\nabla)\sigma(X_t^{x+\lambda r\kappa})d\lambda dW_t.
\end{equation*}
Note that since for every~$t$ and~$\mathbb{P}$-almost all~$\omega$, the functions~$X_t^x$,~$X_{t(\kappa)}^x$ are continuous in $x$, the integrands on the right-hand side are $\mathcal{B}([0,T])\otimes \mathcal{F}\otimes \mathcal{B}([0,1])$-measurable by Lemma~4.51 in \cite{MR2378491} and the integrals (in~$\lambda$) themselves are adapted. For any $\hat{k}\geq 1$, by~\eqref{a3a4rep}, the drift coefficient satisfies
\begin{align*}
&2rX_{t(\kappa)}^{(r)} \cdot \int_0^1( X_{t(\kappa)}^{x+\lambda r\kappa}\cdot\nabla) b(X_t^{x+\lambda r\kappa}) d\lambda \\
&\quad \leq 2r\abs{X_{t(\kappa)}^{(r)}} \int_0^1 \abs{(X_{t(\kappa)}^{x+\lambda r\kappa}\cdot\nabla) b(X_t^{x+\lambda r\kappa})} d\lambda\\
&\quad \leq 2r \abs{X_{t(\kappa)}^{(r)}} \int_0^1 \abs{X_{t(\kappa)}^{x+\lambda r\kappa}} G(X_t^{x+\lambda r\kappa}) d\lambda\\
&\quad \leq \int_0^1 \Big( \abs{X_{t(\kappa)}^{(r)}}^2  +  r^2 \abs{X_{t(\kappa)}^{x+\lambda r \kappa}}^2 \Big) G(t,X_t^{x+\lambda r \kappa}) d\lambda
\end{align*}
and the diffusion coefficient satisfies
\begin{equation*}
\left\|r\int_0^1 (X_{t(\kappa)}^{x+\lambda r\kappa}\cdot\nabla) \sigma(X_t^{x+\lambda r\kappa}) d\lambda \right\|^2 \leq  r^2\int_0^1 \abs*{X_{t(\kappa)}^{x+\lambda r\kappa}}^2 G(t,X_t^{x+\lambda r\kappa}) d\lambda.
\end{equation*}
Consequently, Proposition~\ref{huddethm} 
can be applied with
\begin{align*}
& \hat{b}_t = r\int_0^1 (X_{t(\kappa)}^{x+\lambda r\kappa}\cdot \nabla) b(X_{t(\kappa)}^{x+\lambda r\kappa}) d\lambda,\ \hat{\sigma}_t = r\int_0^1 (X_{t(\kappa)}^{x+\lambda r\kappa}\cdot \nabla) \sigma(X_{t(\kappa)}^{x+\lambda r\kappa}) d\lambda,\\
& \hat{\alpha}_t = 
\int_0^1 G(t,X_t^{x+\lambda r \kappa}) d\lambda, \ \hat{\beta}_t = 4(k\vee 1)r^2 \int_0^1 \abs*{X_{t(\kappa)}^{x+\lambda r\kappa}}^2 G(t,X_t^{x+\lambda r \kappa}) d\lambda,\\
& p^* = 2k\vee 2,\ q_1 = \frac{k}{2},\ q_2 = \bigg(\frac{2}{k} - \frac{1}{k\vee1}\bigg)^{-1},\ q_3 = k\vee 1,\ V_0(t,x) = \abs*{x}^2,
\end{align*}
to obtain
\begin{align}
&\mathbb{E}\bigg[\sup_{0\leq u\leq t} \abs*{X_{u(\kappa)}^{(r)}}^k\bigg] \nonumber\\
&\quad\leq C r^k \bigg( \mathbb{E}\Big[e^{q_2
\int_0^t\int_0^1 G(s,X_s^{x+\lambda r \kappa}) d\lambda ds}\Big]\bigg)^{\frac{k}{2q_2}} \nonumber\\
&\qquad \cdot \bigg(\mathbb{E} \bigg[ 1+ 4(k\vee 1)\int_0^t \int_0^1 \abs*{X_{u(\kappa)}^{x+\lambda r\kappa}}^2 G(u,X_u^{x+\lambda r\kappa})d\lambda ds \bigg]^{k\vee 1}\bigg)^{\frac{k}{2(k\vee 1)}}\nonumber\\
&\quad\leq C r^k \bigg( \mathbb{E}\Big[e^{q_2
\int_0^1\int_0^t G(s,X_s^{x+\lambda r \kappa}) ds d\lambda }\Big]\bigg)^{\frac{k}{2q_2}} \nonumber\\
&\qquad \cdot \bigg(\mathbb{E}\bigg[1+\int_0^1 \sup_{0\leq u\leq t} \abs*{X_{u(\kappa)}^{x+\lambda r \kappa}}^{2k\vee2} \bigg(\int_0^t G(u,X_u^{x+\lambda r\kappa})du \bigg)^{k\vee 1} d\lambda\bigg] \bigg)^{\frac{k}{2(k\vee 1)}}\nonumber\\
&\quad\leq C r^k \bigg( \mathbb{E}\Big[e^{q_2
\int_0^1\int_0^t G(s,X_s^{x+\lambda r \kappa}) ds d\lambda }\Big]\bigg)^{\frac{k}{2q_2}}\nonumber\\
&\qquad\cdot\bigg( 1 + \bigg(\int_0^1\mathbb{E}\bigg[\sup_{0\leq u\leq t}\abs*{X_{u(\kappa)}^{x+\lambda r\kappa}}^{4k\vee 4}\bigg] d\lambda \bigg)^{\frac{1}{2}}\nonumber\\
&\qquad\cdot\bigg(\mathbb{E}\bigg[ \int_0^1 \bigg(\int_0^t G(u,X_u^{x+\lambda r\kappa})du\bigg)^{2k\vee2} d\lambda \bigg]\bigg)^{\frac{1}{2}}\bigg)^{\frac{k}{2(k\vee 1)}}.\label{crga1}
\end{align}
After Jensen's inequality, the first expectation on the right-hand side can be dealt by Lemma~\ref{ergodfin} (the same as in the proof of Lemma~\ref{firstlem}). By~\eqref{firncon}, the second expectation has the bound
\begin{equation}\label{crga2}
\int_0^1\mathbb{E}\bigg[\sup_{0\leq u\leq t}\abs*{X_{u(\kappa)}^{x+\lambda r\kappa}}^{4k\vee 4}\bigg] d\lambda \leq \int_0^1 \rho \bigg(1+\sum_{i\in I_0\cup I_0'}\mathbb{E}[V_
i(0,\bar{x}_i(x+\lambda r \kappa))]\bigg)  d\lambda 
\end{equation}
and, by~\eqref{a5} and Lemma~\ref{ergodfin}, the third and last expectation has the bound
\begin{align}
&\mathbb{E}\bigg[\int_0^1 \bigg(\int_0^t G(u,X_u^{x+\lambda r\kappa})du \bigg)^{2k\vee2} d\lambda\bigg] \nonumber\\ 
&\quad\leq C\int_0^1 \mathbb{E} \bigg[\bigg(1 + \sum_{i\in I_0}\int_0^t \log V_i(u,\bar{x}_i(x + \lambda u \kappa)) du \nonumber\\
&\qquad + \sum_{i'\in I_0'}\log V_{i'}(t,\bar{x}_{i'}(x+\lambda r\kappa)) \bigg)^{\!2k\vee 2}\bigg]\! d\lambda\nonumber\\
&\quad\leq C\int_0^1 \bigg( 1 + \sum_{i\in I_0}t^{2k\vee 2-1}\int_0^t \mathbb{E} [V_i(u,\bar{x}_i(x+\lambda r\kappa))]du \nonumber\\
&\qquad + \sum_{i'\in I_0'} \mathbb{E} [V_{i'}(t,\bar{x}_{i'}(x+\lambda r\kappa))] \bigg) d\lambda\nonumber\\
&\quad\leq C \int_0^1 \bigg(1+\sum_{i\in I_0\cup I_0'}\mathbb{E}[V_i(0,\bar{x}_i(x+\lambda r\kappa))] \bigg) d\lambda.\label{crga3}
\end{align}
Gathering Lemma~\ref{ergodfin},~\eqref{crga2} and~\eqref{crga3}, the bound~\eqref{crga1} becomes
\begin{equation*}
\mathbb{E}\bigg[\sup_{0\leq u \leq t} \abs{X_{u(\kappa)}^{(r)}}^k\bigg] \leq C r^k \bigg(\int_0^1 \bigg( 1 + \sum_{i\in I_0\cup I_0'}\mathbb{E}[V_i(0,\bar{x}_i(x+\lambda r\kappa))] \bigg) d\lambda \bigg)^{\frac{k}{2q_2} + \frac{k}{2(k\vee 1)}},
\end{equation*}
which, by definition of~$q_2$, proves that the conclusion of Lemma~\ref{firstlem} holds except with~$W(x,r\kappa) = 1+\sum_{i\in I_0\cup I_0'}\int_0^1\mathbb{E}[V_i(0,\bar{x}_i(x+\lambda r\kappa))]d\lambda$. The same type of arguments may be applied in the proof of Theorem~\ref{seclem} to obtain the assertions here. In particular, the main point is to use that expressions of the form~$h(s+u,X_u^{s,x'}) - h(s+u,X_u^{s,x}) = \int_0^1\nabla h(s+u,\lambda X_u^{s,x'} + (1-\lambda) X_u^{s,x}) \cdot (X_u^{s,x'} - X_u^{s,x}) d\lambda$ may be replaced by~$\int_0^1 \nabla h(t,X_u^{x+\lambda r\kappa}) \cdot r X_{u(\kappa)}^{x+\lambda r\kappa} d\lambda$. The detailed arguments are omitted.
\end{proof}
\begin{theorem}[Alternative assumptions to Theorem~\ref{fk2}
]\label{alters2}
Let Assumptions~\ref{A1prime},~\ref{A2prime} hold. Let~$f:\Omega\times[0,T]\times\mathbb{R}^n\rightarrow\mathbb{R}$,~$c\in\Omega\times[0,T]\times\mathbb{R}^n\rightarrow[0,\infty)$,~$g:\Omega\times\mathbb{R}^n\rightarrow\mathbb{R}$ be such that~$f(\cdot,x),c(\cdot,x)$ are~$\mathcal{F}\otimes\mathcal{B}([0,T])$-measurable functions for every~$x\in\mathbb{R}^n$, satisfying for any~$\omega\in\Omega$ that~$\int_0^T\sup_{x\in\mathbb{R}^n}(\abs{c(t,x)}+\abs{f(t,x)})dt<\infty$ for every~$R>0$ and~$f(t,\cdot),c(t,\cdot),g\in C^p$ for all~$t\in[0,T]$. Assume there exists~$k_2>1$ such that~$f$ and~$g$ have~$(p,k_2)'$-Lyapunov derivatives. There exists~$K>1$ such that if for any~$1<k'<K$,~$c$ has~$(p,k')'$-Lyapunov derivatives, then the following statements hold.
\begin{enumerate}[label=(\roman*)]
\item For~$u$ given by~\eqref{etstw}, defined for~$(s,x)\in[0,T]\times \mathbb{R}^n$ and stopping times~$t\leq T-s$, the expectation~$\mathbb{E}[u(s,t,x)]$ is continuously differentiable in~$x$ up to order~$p$.
\item For every multiindex~$\beta$ with~$0\leq \abs{\beta}\leq p$, there exists a finite order polynomial~$q^*$, the degree of which is independent of all of the Lyapunov functions in Assumptions~\ref{A1prime},~\ref{A2prime} and of the Lyapunov derivatives, such that for~$(s,x)\in[0,T]\times \mathbb{R}^n$ and all stopping times~$t\leq T-s$, it holds that
\begin{align*}
&\abs{\partial_x^{\beta}\mathbb{E}[u(s,t,x)]} \\
&\quad\leq q^*\bigg(\int_0^1 V(0,\bar{x}(\lambda x + (1-\lambda)x'))d\lambda),\int_0^1 V^{s,T}(0,\tilde{x}(\lambda x + (1-\lambda)x')) d\lambda,\\
&\qquad \int_0^1 \hat{V}_{l_1}^{s,T}(0,\hat{x}_{l_1}(\lambda x+(1-\lambda) x')) d\lambda, V(0,\bar{x}(x)), V^{s,T}(0,\tilde{x}(x)),\\
&\qquad \hat{V}_{l_i}^{s,T}(0,\hat{x}_{l_i}(x)): i \in I^*\bigg),
\end{align*}
where~$I^*\subset\mathbb{N}$ is finite,~$l_i>0$ and~$\tilde{x}$,~$V^{s,T}$ associated to the~$(\abs{\beta},k_2)'$-Lyapunov derivatives of~$f,g$ and the~$(\abs{\beta},k')'$-Lyapunov derivatives of~$c$ are representative across any and all of~$\{f,c,g\}$ and~$k'\in K_0\subset (1,K)$ for some finite~$K_0$.
\item Let Assumption~\ref{A3lya} hold. Suppose~$f,c,g$ are nonrandom and satisfy Assumption~\ref{Afcg}. Suppose that the families of Lyapunov functions in Assumptions~\ref{A1prime},~\ref{A2prime},~\ref{Afcg} and for the Lyapunov derivatives of~$c$ are local in~$s$. For any multiindex~$\alpha$ with~$0\leq\abs{\alpha}\leq p$, the function~$\Delta_T\times\mathbb{R}^n\ni ((s,t),x)\mapsto\abs{\partial_t^{\alpha}\mathbb{E}[u(s,t,x)]}$ is locally bounded and if~$p\geq 2$, then for any~$R>0$, there exists a constant~$N>0$ such that~\eqref{cop1} holds for all~$s,s'\in(0,T)$ and~$x\in B_R$.
\end{enumerate}
\end{theorem}
For the proof of Theorem~\ref{alters2}, Lemma~\ref{lemtstw} can easily be modified using what has already been mentioned in the proof of Theorem~\ref{alters1}, so that Theorem~\ref{fk2}\ref{tstwi} holds. Proofs for the other assertions of Theorem~\ref{fk2} follow in very similar ways. The same applies for the following Theorems~\ref{alters3},~\ref{alters4}. The precise arguments are therefore omitted.
\begin{theorem}[Alternative assumptions to Theorem~\ref{fk}]\label{alters3}
Let the assumptions of Theorem~\ref{alters2} hold with~$p\geq 2$ and let Assumption~\ref{A3lya} hold. Let~$f,c,g$ be nonrandom and satisfy Assumption~\ref{Afcg} (with~$(p,k)'$-Lyapunov derivatives replacing~$(p,k)$-Lyapunov derivatives). There exists~$K>1$ such that if~$c$ has~$(p,k')'$-Lyapunov derivatives for any~$1<k'<K$ and the families of Lyapunov functions in Assumptions~\ref{A1prime},~\ref{A2prime},~\ref{Afcg} and for the Lyapunov derivatives of~$c$ are local in~$s$, then for~$v:[0,T]\times\mathbb{R}^n\rightarrow\mathbb{R}$ given by~\eqref{actualetstw} with~$u$ as in~\eqref{etstw}, the equation~\eqref{asfkeq0} holds almost everywhere in~$(0,T)\times\mathbb{R}^n$.
\end{theorem}
\begin{theorem}[Alternative assumptions to Theorem~\ref{classical}]\label{alters4}
Let the assumptions of Theorem~\ref{alters2} hold with~$p\geq 2$ and let Assumption~\ref{A3lya} hold. Assume~$b,\sigma$ are independent of~$t$. Let~$f,c,g$ satisfy Assumption~\ref{Afcg} and be continuous. There exists~$\bar{K}>1$ such that if
\begin{enumerate}
\item for any~$k'\in(1,\bar{K}]$,~$f,c,g$ have~$(p,k')'$-Lyapunov derivatives,
\item for any family of Lyapunov functions~$(\tilde{V}^{s,T})_s$ and corresponding mappings~$(\hat{x})_s$ in Assumptions~\ref{A1prime},~\ref{A2prime} and Definition~\ref{lyadiffprime}, it holds that~$(\tilde{V}^{s,T})_s$ is local in~$s$ (as in Definition~\ref{locdef}) and there exists a constant~$C>0$ such that for any~$s$, it holds~$\mathbb{P}$-a.s. that~$\tilde{V}^{s,T}(0,\hat{x}(X_t^x))\leq C(1+\tilde{V}^{s,T}(t,\hat{x}(x)))$ for all~$x\in\mathbb{R}^n$,~$t\in[0,T]$,
\end{enumerate}
then the function~$v$ given by~\eqref{actualetstw} and~\eqref{etstw} for all~$(t,x)\in[0,T]\times\mathbb{R}^n$ is the unique classical solution to~\eqref{asfkeq0} on~$[0,T]\times\mathbb{R}^n$, in the sense that~$v\in C^{1,2}$,~$\partial_tv$,~$\nabla\!_x v$,~$D_x^2v$ are continuous,~$v$ satisfies~\eqref{asfkeq0} and it is the only such function satisfying~$v(T,\cdot)=g$.
\end{theorem}

\section{Weak convergence rates for approximations under Lyapunov conditions}\label{weakrates}
Here, the results in Section~\ref{momentestimates} are used with the exponential integrability property of stopped increment-tamed Euler-Maruyama schemes from~\cite{MR3766391} in order to establish weak convergence rates for SDEs with non-globally monotone coefficients. The well-known proof (see~\cite[Theorem~14.5.2]{MR1214374}), establishing weak rates for the Euler-Maruyama scheme approximating~\eqref{sde} with globally Lipschitz coefficients, requires bounds on derivatives of the expectation~\eqref{actualetstw}, the Kolmogorov equation~\eqref{asfkeq0} and moment bounds on the discretization. Although analogous requirements have mostly (beside continuous differentiability of~\eqref{actualetstw} in~$t$) been shown to be met in the setting here, 
the It\^{o}-Alekseev-Gr\"{o}bner formula of~\cite{huddeiag} is used (in the form of Proposition~\ref{iagthm}) for a more direct proof, which uses moment estimates on derivative processes as the main prerequisites. 
Along the way, strong completeness (see e.g.~\cite{MR2857244} for a definition) of the derivative SDEs as in~\eqref{firsdd} (and its higher order analogues) are shown in Lemma~\ref{weakconvlya0} using a result of~\cite{https://doi.org/10.48550/arxiv.1309.5595}. The same assertions as those in Lemma~\ref{weakconvlya0} up to order~$2$ have appeared recently in~\cite{hut1} under different assumptions. The approach here uses the results in~\cite{MR2273672} for continuous differentiability in initial condition as a starting point and consequently requires (at least at face value) the underlying space to be all of~$\mathbb{R}^n$. 
Before the aforementioned strong completeness result, a local H\"older continuity in time result in the strong~$L^p(\mathbb{P})$ sense for derivatives to our SDE is shown in Lemma~\ref{firstlembel}.

We begin by stating the numerical scheme and assumptions from~\cite{MR3766391} (amongst which is a Lyapunov-type condition) used for its exponential integrability. Assumptions about the relationship between the Lyapunov(-type) functions there and those in Assumptions~\ref{A1lya},~\ref{A2lya} are stated alongside, as well as some assumptions from Proposition~\ref{iagthm}. Lemma~\ref{weakconvlya0} serves to verify the rest of the assumptions in Proposition~\ref{iagthm} 
for use in proving the main Theorem~\ref{weakconvlya}. More specifically, Lemma~\ref{weakconvlya0} verifies the continuous differentiability conditions and the finiteness in expectation conditions as assumed in Proposition~\ref{iagthm}.

\begin{assumption}\label{A4lya}
\begin{enumerate}[label=(\roman*)]
\item \label{jgi1} 
The filtration~$\mathcal{F}_t$ satisfies~$\mathcal{F}_t = \sigma(\mathcal{F}_0\cup \sigma(W_s: s\in[0,t]) \cup\{ A\in \mathcal{F}:\mathbb{P}(A) = 0\})$ and that~$\mathcal{F}_0$ and~$\sigma(W_s: s\in[0,T])$ are independent. It holds that~$O=\mathbb{R}^n$ and~$b,\sigma$ are independent of $\omega,t$. 


\item \label{jgi2} There exist~$\gamma,\rho\geq 0$,~$\gamma',c'>0$,~$\xi,c >  1$,~$\bar{U}_0\in\mathbb{R}$,~$U\in C^2(\mathbb{R}^n,[0,\infty))$,~$\bar{U}\in C(\mathbb{R}^n)$ such that~$\bar{U} > \bar{U}_0$,~$U(x)\geq c'\abs{1+x}^{\gamma'}$ 
and
\begin{align}
&\sup_{\substack{\kappa_1,\dots,\kappa_j\in\mathbb{R}^n\setminus\{0\}:\\
                  \abs{\kappa_1} =\dots = \abs{\kappa_j} = 1}}
 \abs*{\sum_{i_1,\dots,i_j = 1}^n \partial_{i_1} \dots\partial_{i_j}(U(x) - U(y)) (\kappa_1)_{i_1}\dots(\kappa_j)_{i_j}}\nonumber\\
&\qquad\leq c\abs{x-y}\bigg(1+\sup_{\lambda\in[0,1]} \abs{U(\lambda x + (1-\lambda) y)}\bigg)^{(1-\frac{j+1}{\xi})\vee 0},\nonumber\\
& \abs{\partial^{\alpha}b(x)} + \|\partial^{\alpha}\sigma(x)\| + \abs{\bar{U}(x)} \leq c(1+U(x))^{\gamma}, \label{bubound}\\
&\frac{\abs{\bar{U}(x) - \bar{U}(y)}}{\abs{x-y}} \leq c(1+\abs{U(x)}^{\gamma} + \abs{U(y)}^{\gamma}),\nonumber\\
&L U(x) + \frac{1}{2}\|\sigma^{\top} \nabla U(x) \|^2 + \bar{U}(x) \leq \rho U(x).\nonumber
\end{align}
for all~$x,y\in\mathbb{R}^n$,~$j \in\{0,1,2\}$ and multiindices~$\alpha$ with~$0\leq \abs{\alpha}\leq 2$. 
\item \label{partition}For any~$\theta \in \Theta := \{\theta = (t_0,\dots,t_{n^*}):n^*\in\mathbb{N}, t_k\in[0,T], t_k < t_{k+1},k\in\{1,\dots,n^*-1\}, t_0 = 0, t_{n^*} = T\}$, the function~$Y_{\cdot}^{\theta}:\Omega\times[0,T]\rightarrow\mathbb{R}^n$ is an~$\mathcal{F}_t$-adapted,~$\mathbb{P}$-a.s. continuous process satisfying~$\sup_{\theta\in\Theta}\mathbb{E}[e^{U(Y_0^{\theta})}] <\infty $ and~$\mathbb{P}$-a.s. that
\begin{align}
Y_t^{\theta} &= Y_{t_k}^{\theta} + \mathds{1}_{\{ y:\abs{y} < \exp(\abs{\log \sup_k t_{k+1} - t_k}^{\frac{1}{2}}) \}}(Y_{t_k}^{\theta})\nonumber\\
&\quad\cdot\bigg[\frac{b(Y_{t_k}^{\theta}) (t-t_k) + \sigma(Y_{t_k}^{\theta})(W_t - W_{t_k})}{1+\abs{b(Y_{t_k}^{\theta}) (t-t_k) + \sigma(Y_{t_k}^{\theta})(W_t - W_{t_k})}^{q'}}\bigg]\label{stotamapp}
\end{align}
on~$t\in(t_k,t_{k+1}]$ for each~$k\in\{0,\dots,n^*-1\}$, where~$q' \geq 3$.

\item \label{ltfirst} 
Assumptions~\ref{A1lya} and~\ref{A2lya} hold with~$p\geq 3$. 
There exists~$0<l^*\leq1$ such that 
for any~$V'\in\{V_i,\hat{V}_k^{s,T}:i\in I_0\cap I_0, s\in[0,T],2\leq\abs{\alpha}\leq p-1,k\geq 2\}$, 
it holds~$\mathbb{P}$-a.s. that 
\begin{align}
V' (0,\tilde{x}'(X_{s,\tau}^y))^{l^*} &\leq C(1+V'(\tau-s,\tilde{x}'(y))), \nonumber\\
V'(0,\tilde{x}'(y))^{l^*} &\leq C(1 + e^{U(y)e^{-\rho T}})\nonumber
\end{align}
for all~$s \in[0, T]$, stopping times~$\tau \leq T-s$,
~$y\in \cup_{\theta\in\Theta}\textrm{Range}(Y_{\cdot}^{\theta})$, where~$\tilde{x}' = \bar{x}_i$ if~$V' = V_i$,~$\tilde{x}' =\hat{x}_k $ otherwise and~$X_{s,t}^y$ solves
\begin{equation}\label{stxsys}
X_{s,t}^y = y + \int_s^t b(X_{s,u}^y) du + \int_s^t \sigma(X_{s,u}^y) dW_u.
\end{equation}

\end{enumerate}
\end{assumption}

\begin{remark}\label{before51}
\begin{enumerate}[label=(\alph*)]
\item Assumption~\ref{A4lya}\ref{jgi2} implies that the mapping~$(t,x,y)\mapsto e^{U(x)e^{-\rho t} + y}$ is a Lyapunov function in the sense of~\cite[Theorem~3.5]{MR2894052} for an extended system (see the proof of Corollary~3.3 in~\cite{MR4260476}), so that for all~$s\in[0,T]$,~$x\in\mathbb{R}^n$, there exists a unique up to distinguishability,~$\mathcal{F}_t$-adapted,~$\mathbb{P}$-a.s. continuous solution to~\eqref{stxsys} 
and for~$t\in[s,T]$ it holds~$\mathbb{P}$-a.s. that~$X_{t,T}^{X_{s,t}^x} = X_{s,T}^x$. 
\item In Assumption~\ref{A4lya}\ref{jgi1}, the assertions about~$\mathcal{F}_t$ are essentially from~\cite{huddeiag}. We set~$O$ to be the whole space and fix~$b$ and~$\sigma$ to be time-independent and nonrandom in order to use continuous differentiability in initial value from~\cite{MR2273672} and to use the exponential integrability results of~\cite{MR3766391}. 
\item Assumptions~\ref{A4lya}\ref{jgi2} and~\ref{A4lya}\ref{partition} also follow closely the assumptions in~\cite{MR4079431,MR3766391}. Here, two things are of note. Firstly,~$q'$ is assumed to be greater than or equal to~$3$ rather than~$1$ in the denominator of the expression for~$Y_t^{\theta}$; this assumption is made in order to ensure well-behavedness of some higher order terms in the It\^{o}-Alekseev-Gr\"{o}bner expansion such that weak convergence rate of order~$1$ is attained. Secondly, the lower bound~$U\geq c'\abs{1+x}^{\gamma'}$ and~\eqref{bubound} are not strictly necessary. This is useful for determining that the main assumptions genuinely generalize the globally Lipschitz case, where the Lyapunov functions are polynomial and~$U$ grows like~$\log x$. More precise generalizing assumptions are given in Remark~\ref{remprof}. Note that these in turn relax the regularity conditions on~$b,\sigma$ in~\cite{MR1214374} for order one weak convergence rates.

\item The Lipschitz estimate on~$U$ with~$j = 0$ in~\ref{jgi2} easily gives that~$U$ is polynomially bounded, so that the set under the indicator function in~\ref{partition} indeed satisfies the assumptions in~\cite{MR3766391}, as used in~\cite{MR4079431,MR3766391}. 
\item Item~\ref{ltfirst} (and in general Assumption~\ref{A4lya}) are easily satisfied by the examples mentioned here and in particular if all of the Lyapunov functions have~$V_0 = e^{U(x)e^{-\rho t}}$ (as in Theorem~\ref{wbth}) or~$V_0 = e^{U(x)e^{-\rho t} + y}$ (as in again proof of Corollary~3.3 in~\cite{MR4260476}). 

\end{enumerate}
\end{remark}

In the following, for any~$s\in[0,T]$, we extend the definition of any process~$Z_t$ defined on~$[s,T]$ to~$[0,T]$ by setting~$Z_t = Z_s$ for~$t\in[0,s)$. In the proofs, many computations are close in spirit to those in Lemma~\ref{firstlem}, Theorem~\ref{seclem} and so are compressed.

\begin{lemma}\label{firstlembel}
Under Assumption~\ref{A4lya}, for any~$k_1 > 2(n+1)$,~$R>0$, there exist constants~$C>0$,~$n+1<\nu_1\leq k_1$ such that 
\begin{equation*}
\mathbb{E}\bigg[\sup_{u\in[s,t]}\abs{\partial^{(\kappa)}X_{s,u}^x - \partial^{(\kappa)}X_{s,s}^x }^{k_1} \bigg] < C\abs{t-s}^{\nu_1}
\end{equation*}
for all~$(s,t)\in\Delta_T$,~$x\in B_R$,~$\kappa \in \{ (\kappa_i)_{1\leq i\leq p_0}:\kappa_i\in\mathbb{R}^n, \abs*{\kappa_i} =1, 1\leq i\leq p, p_0\in \mathbb{N}_0\cap [0,p]\}$.
\end{lemma}
\begin{proof}

By~\eqref{kappar2} in Theorem~\ref{seclem} (with a time-shifted Wiener process and filtration) and using that~$\partial^{(\kappa)}X_{s,s}^x = 0$ (for~$\kappa$ in the following set), the existence of such constants have already been shown for~$\kappa \in\{ (\kappa_i)_{1\leq i\leq p_0}:\kappa_i\in\mathbb{R}^n, \abs*{\kappa_i} =1, 1\leq i\leq p, p_0\in \mathbb{N}_0\cap [2,p]\}$. Using Assumption~\ref{A4lya}\ref{jgi2}, 
Corollaries~\ref{huddecor1},~\ref{huddecor2} 
as well as Jensen's inequality, it holds that
\begin{align*}
\mathbb{E}\bigg[\sup_{u\in[s,t]}\abs{X_{s,u}^x - x}^{k_1}\bigg] &\leq Ce^{k_1 (t-s)} \bigg(  \mathbb{E}\bigg[ (t-s)^{k_1-1}\int_0^{t-s} e^{U(X_{s,s+u}^x)e^{-\rho u}-2k_1u} du \bigg] \bigg)^{\frac{1}{2}}\\
&\leq Ce^{k_1 (t-s)} \bigg( (t-s)^{k_1-1}\int_0^{t-s}  e^{U(x)}du \bigg)^{\frac{1}{2}}\\
&\leq C\abs{t-s}^{\frac{k_1}{2}} 
\end{align*}
for all~$(s,t)\in\Delta_T$,~$x\in B_R$. 
Using in addition Assumption~\ref{A1lya}, it holds that
\begin{align*}
&\mathbb{E}\bigg[\sup_{u\in[s,t]}\abs{\partial^{(\kappa_i)}X_{s,u}^x - \kappa_i}^{k_1}\bigg]\\
&\quad \leq C\Big(\mathbb{E}\Big[ \Big(e^{m(\sum_{i\in I_0}\int_0^{t-s} \log V_i(u,\bar{x}_i(x)) + \sum_{i'\in I_0'} \log V_{i'}(t-s,\bar{x}_{i'}(x)) } \Big)^{2k_1}\Big]\Big)^{\frac{1}{2}}\\
&\qquad\cdot \bigg(  \mathbb{E}\bigg[(t-s)^{k_1-1}\int_0^{t-s} e^{U(X_{s,s+u}^x)e^{-\rho u}} du \bigg]  \bigg)^{\frac{1}{2}}\\
&\quad\leq C \bigg(1+\sum_{i\in I_0\cup I_0'}V_i(0,\bar{x}_i(x))\bigg)^{\frac{1}{2}} \bigg( (t-s)^{k_1-1}\int_0^{t-s} e^{U(x)} du \bigg)^{\frac{1}{2}}\\
&\quad\leq C\abs{t-s}^{\frac{k_1}{2}} 
\end{align*}
for all~$(s,t)\in\Delta_T$,~$x\in B_R$,~$\kappa_i\in\mathbb{R}^n$ with~$\abs{\kappa_i} = 1$.
\end{proof}

The following lemma verifies the corresponding assumptions in Proposition~\ref{iagthm} 
under Assumption~\ref{A4lya}. Moreover, it is shown that the estimates therein hold uniformly with respect to the discretization~$\theta\in\Theta$. 
\begin{lemma}\label{weakconvlya0}
Let Assumption~\ref{A4lya} hold. 
There exists a function~$\Omega\times\Delta_T\times\mathbb{R}^n\ni(\omega,(s,t),x)\mapsto\bar{X}_{s,t}^x(\omega)\in\mathbb{R}^n$ 
such that 
\begin{itemize}
\item it holds~$\mathbb{P}$-a.s. that for any~$(s,t)\in\Delta_T$,~$\mathbb{R}^n\ni x\mapsto \bar{X}_{s,t}^x\in\mathbb{R}^n$ is continuously differentiable in~$x$ up to order~$p-1$ and
the derivative
~$\Delta_T\times \mathbb{R}^n \ni ((s,t),x) \mapsto \partial^{\alpha}\bar{X}_{s,t}^x \in \mathbb{R}^n$ is continuous for all multiindices~$\alpha$ with~$0\leq\abs{\alpha}\leq p-1$,
\item for any~$s\in[0,T]$,~$x\in\mathbb{R}^n$, 
the function
~$\partial^{\alpha}\bar{X}_{s,\cdot}^x$ is indistinguishable from
~$\partial^{(\kappa_{\alpha})}X_{s,\cdot}^x$ for all multiindices~$\alpha$ with~$0\leq\abs{\alpha}\leq p-1$.
\end{itemize}
Moreover, for any~$p^{\dagger}>0$, 
it holds that
\begin{equation*}
\sup_{0\leq\abs{\alpha}\leq p-1}\sup_{\theta\in\Theta}\sup_{0\leq r\leq s \leq t \leq T} \mathbb{E}\bigg[\Big\vert b\Big(\bar{X}_{s,t}^{Y_s^{\theta}}\Big)\Big\vert^{p^{\dagger}} + \Big\|\sigma\Big(\bar{X}_{s,t}^{Y_s^{\theta}}\Big)\Big\|^{p^{\dagger}}
+ \bigg\vert\partial^{\alpha} \bar{X}_{t,T}^{\bar{X}_{r,s}^{Y_r^{\theta}}}\bigg\vert^{p^{\dagger}}\bigg] 
 < \infty.
\end{equation*}

\end{lemma}

\begin{proof}
By Lemma~\ref{classprob} (with time-shifted Wiener process and filtration), 
derivatives in probability~$\partial^{(\kappa_{\alpha})}X_{s,\cdot}^x$ are indistinguishable from classical derivatives~$\partial^{\alpha}\hat{X}_{s,\cdot}^x$. 
In order to use 
Corollary~3.10 in~\cite{https://doi.org/10.48550/arxiv.1309.5595}, we show that for each~$R>0$,~$k_1>2(n+1)$, it holds that
\begin{equation}\label{desprop}
\sup_{0\leq\abs{\alpha}\leq p-1}\sup_{x,x'\in B_R}\sup_{s,s' \in[0,T]} \frac{ \mathbb{E}[\sup_{t\in[0,T]} \abs{ \partial^{\alpha}\hat{X}_{s',t}^{x'} - \partial^{\alpha}\hat{X}_{s,t}^x}^{k_1}] }{(\abs{x'-x}^2 + \abs{s'-s}^2)^{\frac{\nu_1}{2}} } < \infty,
\end{equation}
where~$\nu_1$ is the same constant from Lemma~\ref{firstlembel}.
The marginal differences in~$x$ and~$s$ in the numerator are considered separately. By Lemma~\ref{firstlem} or Theorem~\ref{seclem}, the difference term in~$x$ in the numerator of~\eqref{desprop} has the bound
\begin{align*}
\mathbb{E}\bigg[\sup_{t\in[0,T]}\abs{\partial^{\alpha}\hat{X}_{s',t}^{x'} - \partial^{\alpha}\hat{X}_{s',t}^x}^{k_1}\bigg]\leq C\abs{x'-x}^{k_1}
\end{align*}
for all~$s\in[0,T]$,~$x,x'\in B_R$, which is the desired H\"older bound for~\eqref{desprop}. For the difference term in~$s$ in the numerator of~\eqref{desprop}, it holds that
\begin{align}
\mathbb{E}\bigg[\sup_{t\in[0,T]} \abs{ \partial^{\alpha} \hat{X}_{s',t}^x - \partial^{\alpha} \hat{X}_{s,t}^x }^{k_1} \bigg] & \leq \mathbb{E}\bigg[\sup_{t\in[s\wedge s',s\vee s']} \abs{ \partial^{\alpha}\hat{X}_{s\wedge s',t}^x - \partial^{\alpha} \hat{X}_{s\vee s',s\vee s'}^x }^{k_1} \bigg]\nonumber\\
&\quad + \mathbb{E}\bigg[\sup_{t\in[s\vee s',T]} \abs{ \partial^{\alpha} \hat{X}_{s',t}^x - \partial^{\alpha} \hat{X}_{s,t}^x }^{k_1} \bigg],\label{thatlastholder}
\end{align}
where the first term on the right-hand side has the desired H\"older bound for~\eqref{desprop} by Lemma~\ref{firstlembel}. For the second term, by Assumption~\ref{A4lya}\ref{ltfirst} and Lemma~\ref{eulercan}, combined with Theorem~5.3 in~\cite{MR1731794}, the joint system solved by~$(\partial^{\alpha} X_{s,t}^x)_{0\leq\abs{\alpha}\leq p-1}$ is regular~\cite[Definition~2.1]{MR1731794} and the same holds for the sum~$(\partial^{\alpha} X_{s',t}^x - \partial^{\alpha} X_{s,t}^x)_{0\leq\abs{\alpha}\leq p-1}$ by an easy argument; therefore the strong Markov property (Theorem~2.13 in~\cite{MR1731794} with Proposition~4.1.5 in~\cite{MR838085}) 
yields for any~$R'>0$ that
\begin{align}
&\mathbb{E}\bigg[\sup_{t\in[s\vee s',T]} \abs{ \partial^{\alpha} \hat{X}_{s',t}^x - \partial^{\alpha} \hat{X}_{s,t}^x }^{k_1}\wedge R'  \bigg]\nonumber\\
&\quad= \mathbb{E}\bigg[\mathbb{E}\bigg[\sup_{t\in[s\vee s',T]} \abs{ \partial^{\alpha} \hat{X}_{s',t}^x - \partial^{\alpha} \hat{X}_{s,t}^x }^{k_1} \wedge R' \bigg\vert \mathcal{F}_{s\vee s'}\bigg]\bigg]\nonumber\\
&\quad= \int \int \sup_{t\in[s\vee s',T]} \Big\vert \partial^{\alpha} \hat{X}_{s\vee s',t}^{(\partial^{\beta}\hat{X}_{s\wedge s',s\vee s'}^x(\omega))_{\beta}}(\omega')\nonumber\\
&\qquad - \partial^{\alpha}\hat{X}_{s\vee s',t}^x(\omega') \Big\vert^{k_1} \wedge R'\ d\mathbb{P}(\omega') d\mathbb{P}(\omega),\label{texpwp}
\end{align}
where~$\partial^{\alpha} \hat{X}_{s\vee s',t}^{(\partial^{\beta}\hat{X}_{s\wedge s',s\vee s'}^x(\omega))_{\beta}}(\omega')$ denotes the solution to the same (joint) system as~$\partial^{\alpha} \hat{X}_{s\vee s',t}^x(\omega')$ but with initial conditions~$\partial^{\beta}\hat{X}_{s\wedge s',s\vee s'}^x(\omega)$ for~$0\leq\abs{\beta}\leq p-1$ for each respective partial derivative in place of the initial conditions~$x$,~$e_i$ or~$0$. Then the proofs of Lemma~\ref{firstlem} and Theorem~\ref{seclem} may be slightly modified in order to obtain analogous statements for the expectation in~$\omega'$ in~\eqref{texpwp}; the modification is namely that the initial condition (fixed with respect to~$\omega'$) as mentioned can be added with no complications\footnote{Actually the~$T-s$ term is lost on the right-hand side of~\eqref{kappar1} but that's not important here.} when Corollary~\ref{huddecor1}
is applied. Given this, it holds that
\begin{align*}
&\mathbb{E}\bigg[\sup_{t\in[s\vee s',T]} \abs{ \partial^{\alpha} \hat{X}_{s',t}^x - \partial^{\alpha} \hat{X}_{s,t}^x }^{k_1}\wedge R'  \bigg] \\
&\quad\leq C \sum_{\beta = 0}^{\abs{\alpha}-1}\mathbb{E} \Big[\abs{\partial^{\beta}\hat{X}_{s\wedge s',s\vee s'}^x - \partial^{\beta}\hat{X}_{s\vee s',s\vee s'}^x }^{k_1}\Big],\\
&\quad= C \sum_{\beta = 0}^{\abs{\alpha}-1}\mathbb{E} \Big[\abs{\partial^{\beta}\hat{X}_{s\wedge s',s\vee s'}^x - \partial^{\beta}\hat{X}_{s\wedge s',s\wedge s'}^x }^{k_1}\Big],
\end{align*}
for all~$x\in B_R$,~$s,s'\in[0,T]$,~$0\leq \abs{\alpha}\leq p-1$, which, by Lemma~\ref{firstlembel} and dominated convergence in~$R'$, implies that the last term on the right-hand side of~\eqref{thatlastholder} has the desired H\"older bound for~\eqref{desprop}. Gathering the above and using the triangle inequality,~\eqref{desprop} holds. 
Consequently, using on the way Lemma~\ref{firstlem} and Theorem~\ref{seclem}, Corollary~3.10 in~\cite{https://doi.org/10.48550/arxiv.1309.5595} may be applied with~$\beta = \frac{\nu_1}{k_1}$,~$D = [0,T]\times\mathbb{R}^n$,~$E=F=C([0,T],\mathbb{R}^n)$,~$X = (\Omega\times [0,T]\times\mathbb{R}^n \ni (\omega,s,x)\mapsto \partial^{\alpha}\hat{X}_{s,\cdot}^x (\omega)\in C([0,T],\mathbb{R}^n))$ to obtain for~$0\leq\abs{\alpha}\leq p-1$ existence of an~$\mathcal{F}\otimes\mathcal{B}([0,T])\otimes\mathcal{B}(\mathbb{R}^n)$-measurable~$\Omega\times[0,T]\times\mathbb{R}^n\ni(\omega,s,x)\mapsto \overline{\partial^{\alpha} X}_{s,\cdot}^{\,x}(\omega) \in C([0,T],\mathbb{R}^n)$ such that for all~$\omega\in\Omega$, the function~$[0,T]\times\mathbb{R}^n\ni (s,x)\mapsto \overline{\partial^{\alpha} X}_{s,\cdot}^{\,x} \in C([0,T],\mathbb{R}^n)$ is continuous and for any~$(s,x)\in[0,T]\times\mathbb{R}^n$,~$\overline{\partial^{\alpha} X}_{s,\cdot}^{\,x}$ is indistinguishable from~$\partial^{\alpha}\hat{X}_{s,\cdot}^x$.

Since partial integrals of (jointly) continuous functions are still continuous, we may partially integrate~$\abs{\alpha}$ times each~$\Delta_T\times\mathbb{R}^n \ni ((s,t),x)\mapsto\overline{\partial^{\alpha} X}_{s,t}^{\,x}\in\mathbb{R}^n$ from~$0$ to~$x_i$ in order to obtain for each~$\alpha,\omega$ a continuous function~$\Delta_T\times\mathbb{R}^n\ni((s,t),x)\mapsto \bar{X}_{s,t}^{x,\alpha}\in\mathbb{R}^n$ (where at each integration, continuous functions of the form~$((s,t),x)\mapsto \overline{\partial^{\beta}X}_{s,t}^{\,(x_1,\dots,0,\dots,x_n)}$ and subsequently their integrals are added in line with the fundamental theorem of calculus, which have zero partial derivative
). For any~$(s,t)\in\Delta_T$ and~$\alpha$ with~$0\leq\abs{\alpha}\leq p-1$, by definition of~$\overline{\partial^{\alpha}X}_{s,t}^{\,x}$ and its continuity in~$x$, it holds~$\mathbb{P}$-a.s. that~$\overline{\partial^{\alpha}X}_{s,t}^{\,x} = \partial^{\alpha}\hat{X}_{s,t}^x$ for all~$x\in\mathbb{R}^n$, so that their partial integrals in~$x$ are also~$\mathbb{P}$-a.s. equal for all~$x\in\mathbb{R}^n$ and in particular it holds~$\mathbb{P}$-a.s. that~$\bar{X}_{s,t}^{x,\alpha} = \hat{X}_{s,t}^x$, 
for all~$x\in\mathbb{R}^n$. 
Therefore, by continuity in~$(s,t),x$, these functions coincide~$\mathbb{P}$-a.s. across~$\alpha$, that is, it holds~$\mathbb{P}$-a.s. that~$\bar{X}_{s,t}^{x,\alpha} = \bar{X}_{s,t}^{x,\alpha'}$ and thus~$\partial^{\beta}\bar{X}_{s,t}^{x,\alpha} = \partial^{\beta}\bar{X}_{s,t}^{x,\alpha'}$ for all~$(s,t)\in\Delta_T$,~$x\in\mathbb{R}^n$ and multiindices~$\alpha,\alpha',\beta$ with
~$\abs{\alpha},\abs{\alpha'},\abs{\beta} \in [0, p-1]$. 
Let this~$\mathbb{P}$-a.s. defined function be denoted by~$\bar{X}_{s,t}^x$, then the assertions about~$\bar{X}_{s,t}^x$ in the statement of the lemma have been shown.

For the last assertion, the Markov property (Theorem~2.13 in~\cite{MR1731794}) will be applied repeatedly without further mention. Since Assumption~\ref{A4lya}\ref{jgi2} implies in particular for any~$p^{\dagger}>0$ that
\begin{equation*}
\abs*{b(x)}^{p^{\dagger}} + \|\sigma (x)\|^{p^{\dagger}} \leq Ce^{U(x)e^{-\rho t} }
\end{equation*}
for all~$x\in\mathbb{R}^n$,~$t\in[0,T]$, 
by Corollary~\ref{huddecor2} 
and Assumption~\ref{A4lya}\ref{jgi2}, it holds that
\begin{align*}
&\sup_{\theta\in\Theta}\sup_{0\leq s\leq t\leq T}\mathbb{E}[\abs{b(\bar{X}_{s,t}^{Y_s^{\theta}})}^{p^{\dagger}} + \|\sigma(\bar{X}_{s,t}^{Y_s^{\theta}})\|^{p^{\dagger}}  ]\\
&\quad\leq C\sup_{\theta\in\Theta}\sup_{0\leq s\leq t\leq T}\mathbb{E}\Big[e^{U(\bar{X}_{s,t}^{Y_s^{\theta}})e^{-\rho (t-s)} }\Big]\\
&\quad\leq C\sup_{\theta\in\Theta}\sup_{0\leq s\leq t\leq T}\mathbb{E}\Big[e^{U(\bar{X}_{s,t}^{Y_s^{\theta}})e^{-\rho (t-s)} + \int_s^t \bar{U}(\bar{X}_{s,u}^{Y_s})e^{-\rho (u-s)} du } \Big]\\
&\quad\leq C\sup_{\theta\in\Theta}\sup_{0\leq s\leq T}\mathbb{E}\Big[e^{U(Y_s^{\theta})}\Big],
\end{align*}
which is finite by Theorem~2.9 in~\cite{MR3766391}. For any~$p^{\dagger}>0$, by Assumption~\ref{A4lya}\ref{jgi2}, 
Corollary~\ref{huddecor2} 
and that~$e^{-\rho(s-r)},e^{-\rho r} < 1$, it holds that
\begin{align*}
&\sup_{0\leq r\leq s\leq t\leq T} \mathbb{E}\bigg[\Big\vert \bar{X}_{t,T}^{\bar{X}_{r,s}^{Y_r^{\theta}}}\Big\vert^{p^{\dagger}}\bigg]\\
&\quad \leq C\sup_{0\leq r\leq s\leq t\leq T} \mathbb{E} \bigg[ \exp(U(\bar{X}_{t,T}^{\bar{X}_{r,s}^{Y_r^{\theta}}}) e^{-\rho (T-t)}e^{-\rho(s-r)}e^{-\rho r}\\
&\qquad + \int_t^T \bar{U}(\bar{X}_{t,u}^{\bar{X}_{r,s}^{Y_r^{\theta}}})e^{-\rho (u-t)}e^{-\rho(s-r)}e^{-\rho r} du ) \bigg] \\
&\quad\leq C \sup_{0\leq r\leq s \leq T} \mathbb{E} \bigg[e^{U(\bar{X}_{r,s}^{Y_r^{\theta}})e^{-\rho(s-r)}e^{-\rho r}} \bigg]\\
&\quad\leq C \sup_{0\leq r\leq s \leq T} \mathbb{E} \bigg[e^{U(\bar{X}_{r,s}^{Y_r^{\theta}})e^{-\rho(s-r)}e^{-\rho r} + \int_r^s \bar{U}(\bar{X}_{r,u}^{Y_r^{\theta}})e^{-\rho (u-r)}e^{-\rho r} du } \bigg]\\
&\quad\leq C\sup_{0\leq r \leq T} \mathbb{E} \Big[e^{U(Y_r^{\theta})e^{-\rho r}}\Big],
\end{align*}
for all~$\theta\in\Theta$, which is finite uniformly in~$\theta$ by Theorem~2.9 in~\cite{MR3766391}.

For the higher derivatives, first note that for~$V_0$ satisfying~\eqref{lyapr2} and~$0<l<1$,~\eqref{lyapr2} is also satisfied with~$(V_0+1)^l$ in place of~$V_0$. Moreover, the respective Lyapunov functions they generate satisfy Assumptions~\ref{A1lya} and~\ref{A2lya}. Therefore, for any~$\tilde{I}\in\mathbb{N}\cap [1,p-1]$,~$\kappa \in \{(\kappa_i)_{i = 1,\dots,\tilde{I}}: \kappa_i\in\mathbb{R}^n,\abs{\kappa_i}=1 \}$, we may choose~$l = \frac{(l^*)^2}{\textrm{degree}(q_0)}$, with~$q_0$ from Theorem~\ref{seclem}, so that for~$\tilde{p}^{\dagger}>0$, by Lemma~\ref{firstlem} or Theorem~\ref{seclem}, Young's inequality, Assumptions~\ref{A4lya}\ref{jgi2}\ref{ltfirst} and 
Corollary~\ref{huddecor0}, 
it holds that
\begin{align*}
&\sup_{0\leq r\leq s\leq t\leq T} \mathbb{E}\bigg[\Big\vert \partial^{(\kappa)}X_{t,T}^{X_{r,s}^{Y_r^{\theta}}}\Big\vert^{\tilde{p}^{\dagger}}\bigg]\\
&\quad\leq C \sup_{0\leq r\leq s \leq T} \mathbb{E} \bigg[1+ \sum_{i\in I_0\cup I_0'}V_i(0,\bar{x}_i( X_{r,s}^{Y_r^{\theta}} ))^{(l^*)^2} + \sum_{i=1}^{i^*}\hat{V}_{l_i}^{0,T}(0,\hat{x}_{l_i}(X_{r,s}^{Y_r^{\theta}}) )^{(l^*)^2}\bigg]\\
&\quad\leq C \sup_{0\leq r\leq s \leq T} \mathbb{E} \bigg[1+ \sum_{i\in I_0\cup I_0'}V_i(s-r,\bar{x}_i( Y_r^{\theta} ))^{l^*} + \sum_{i=1}^{i^*}\hat{V}_{l_i}^{0,T}(s-r,\hat{x}_{l_i}(Y_r^{\theta}) )^{l^*}\bigg]\\
&\quad\leq C \sup_{0\leq r \leq T} \mathbb{E} \bigg[1+ \sum_{i\in I_0\cup I_0'}V_i(0,\bar{x}_i( Y_r^{\theta} ))^{l^*} + \sum_{i=1}^{i^*}\hat{V}_{l_i}^{0,T}(0,\hat{x}_{l_i}(Y_r^{\theta}) )^{l^*}\bigg]\\
&\quad\leq C \sup_{0\leq r \leq T} \mathbb{E} \Big[1+ e^{U(Y_r^{\theta})e^{-\rho T}}\Big],
\end{align*}
where~$C$ is in particular independent of~$\kappa \in \{(\kappa_i)_{i = 1,\dots,\tilde{I}}: \kappa_i\in\mathbb{R}^n,\abs{\kappa_i}=1 \}$ and~$\theta\in\Theta$, so that the right-hand side is finite uniformly in~$\theta$ by Theorem~2.9 in~\cite{MR3766391} and also uniformly in~$\tilde{I}$.
\end{proof}

The main theorem of this section about weak convergence of order~$1$ for the stopped increment-tamed Euler-Maruyama scheme is as follows.
\begin{theorem}\label{weakconvlya}
Let Assumption~\ref{A4lya} hold. 
For~$f\in C^3(\mathbb{R}^n,\mathbb{R})$, if there exist constants~$q^{\dagger},C_f>0$ such that
\begin{equation}\label{theoneonf}
\abs{\partial^{\alpha}f(x)} 
\leq C_f\Big(1+\abs{x}^{q^{\dagger}}\Big)
\end{equation}
for all~$x\in\mathbb{R}^n$ and multiindices~$\alpha$ with~$0\leq \abs{\alpha} \leq 3$, then there exists a constant~$C>0$ such that
\begin{equation*}
\Big\vert\mathbb{E}\Big[f\Big(X_{0,T}^{Y_0^{\theta}}\Big)\Big] - \mathbb{E}\Big[f\Big(Y_T^{\theta}\Big)\Big]\Big\vert \leq C\sup_{k\in\mathbb{N}_0\cap[0,n^*)}(t_{k+1}-t_k)
\end{equation*}
for all~$\theta\in\Theta$, where~$\theta = (t_0,\dots, t_{n^*})$.
\end{theorem}

\begin{proof}
Throughout the proof, we write~$D_{\abs{\theta}} = \{ y:\abs{y} < \exp(\abs{\log \sup_k t_{k+1} - t_k}^{\frac{1}{2}}) \}$. 
To begin, we rewrite the approximation~$Y_t^{\theta}$ as the solution of a SDE. For every~$k\in\mathbb{N}_0\cap[0,n^*-1]$,~$\theta = (t_0,\dots,t_{n^*})\in\Theta$, consider
\begin{align}
Z_t^{\theta,k} &= \begin{cases}
0 & \textrm{if }t \leq t_k\\
b(Y_{t_k}^{\theta}) (t-t_k) + \sigma(Y_{t_k}^{\theta}) (W_t - W_{t_k})& \textrm{if }t_k< t \leq t_{k+1}\\
b(Y_{t_k}^{\theta}) (t_{k+1}-t_k) + \sigma(Y_{t_k}^{\theta}) (W_{t_{k+1}} - W_{t_k})& \textrm{if }t_{k+1}< t
\end{cases}\nonumber\\
&= \int_0^t \mathds{1}_{(t_k,t_{k+1}]}(u)b(Y_{t_k}^{\theta}) du  + \int_0^t \mathds{1}_{(t_k,t_{k+1}]}(u)\sigma(Y_{t_k}^{\theta}) dW_u,\label{Zbssde1}
\end{align}
defined for all~$t\in[0,T]$, then~$Y_t^{\theta}$ solves
\begin{equation}
Y_t^{\theta} = Y_0^{\theta} + \sum_{k = 0}^{n^*-1} \mathds{1}_{D_{\abs{\theta}}}(Y_{t_k}^{\theta})\frac{Z_t^{\theta,k}}{1+\abs{Z_t^{\theta,k}}^{q'}},\label{Zbssde2}
\end{equation}
where by It\^{o}'s rule, for~$\hat{f}:\mathbb{R}^n\rightarrow \mathbb{R}^n$ given by~$\hat{f}(z) = \frac{z}{1+\abs{z}^{q'}}$, it holds that 
\begin{align}
\frac{Z_t^{\theta,k}}{1+\abs{Z_t^{\theta,k}}^{q'}} &= \int_0^t \mathds{1}_{(t_k,t_{k+1}]}(u) (b(Y_{t_k}^{\theta}) + b^*(Y_{t_k}^{\theta},Z_u^{\theta,k}))du\nonumber\\
&\quad + \int_0^t \mathds{1}_{(t_k,t_{k+1}]}(u) (\sigma(Y_{t_k}^{\theta}) + \sigma^*(Y_{t_k}^{\theta},Z_u^{\theta,k})) dW_u\label{Zbssde3}
\end{align}
and~$b^*:\mathbb{R}^n\times\mathbb{R}^n\rightarrow \mathbb{R}^n$ and~$\sigma^*:\mathbb{R}^n\times\mathbb{R}^n\rightarrow \mathbb{R}^{n\times n}$ are given by
\begin{align}
b^*(y,z) &= -b(y)\bigg(\frac{\abs{z}^{q'}}{1+\abs{z}^{q'}}\bigg) - q' z\bigg(z\cdot b(y) \frac{\abs{z}^{q' - 2}}{(1+\abs{z}^{q'})^2}\bigg)\nonumber\\
&\quad + \frac{1}{2}((\sigma\sigma^{\top}(y)):D^2)\hat{f}(z)\label{bstardef}\\
\sigma^*(y,z) &= -\sigma(y)\bigg(\frac{\abs{z}^{q'} }{1+\abs{z}^{q'}}\bigg) - q' z\bigg(z^{\top} \sigma(y) \frac{\abs{z}^{q' - 2}}{(1+\abs{z}^{q'})^2}\bigg).\label{sigstardef}
\end{align}
By assumption it holds that~$q' \geq 3$, therefore there exists a constant~$\nu_2\geq 2$ such that the second order derivatives satisfy~$\abs{\partial_{ij}^2\hat{f}(z)} \leq C\abs{z}^{\nu_2} $ for all~$z\in\mathbb{R}^n$,~$i,j\in\mathbb{N}\cap [1,n]$. 

By Lemma~\ref{weakconvlya0} and Proposition~\ref{iagthm}, 
for any~$\theta\in\Theta$, it holds that
\begin{align}
&\mathbb{E}\Big[f\Big(X_{0,T}^{Y_0^{\theta}}\Big)\Big]- \mathbb{E}\Big[f\Big(Y_T^{\theta}\Big)\Big]\nonumber\\
&\quad = \sum_{k=0}^{n^*-1}\mathbb{E}\bigg[\int_{t_k}^{t_{k+1}} \Big(\Big(\Big(\Big(b(Y_t^{\theta}) - \mathds{1}_{D_{\abs{\theta}}}(Y_{t_k}^{\theta})(b(Y_{t_k}^{\theta})\nonumber\\
&\qquad + b^*( Y_{t_k}^{\theta},Z_t^{\theta,k} ))\Big) \cdot \nabla\Big) \bar{X}_{t,T}^{Y_t^{\theta}} \Big) \cdot \nabla\Big) f(\bar{X}_{t,T}^{Y_t^{\theta}}) dt\bigg]\nonumber\\
&\qquad + \frac{1}{2}\mathbb{E}\bigg[\int_{t_k}^{t_{k+1}} \sum_{i,j = 1}^n\Big( \sigma(Y_t^{\theta})\sigma(Y_t^{\theta})^{\top} - \mathds{1}_{D_{\abs{\theta}}}(Y_{t_k}^{\theta})(\sigma(Y_{t_k}^{\theta}) \nonumber\\
&\qquad + \sigma^*(Y_{t_k}^{\theta},Z_t^{\theta,k}))(\sigma(Y_{t_k}^{\theta}) + \sigma^*(Y_{t_k}^{\theta},Z_t^{\theta,k}))^{\top}\Big)_{ij} \nonumber\\
&\qquad \cdot\Big( ((\partial_i \bar{X}_{t,T}^{Y_t^{\theta}} \otimes \partial_j \bar{X}_{t,T}^{Y_t^{\theta}} ) : D^2) f(\bar{X}_{t,T}^{Y_t^{\theta}})+ (\partial_{ij}^2 \bar{X}_{t,T}^{Y_t^{\theta}} \cdot \nabla) f(\bar{X}_{t,T}^{Y_t^{\theta}})\Big)dt\bigg].\label{weakexp}
\end{align}
For the first terms on the right-hand side of~\eqref{weakexp}, denoting
\begin{equation}\label{bhatstar1}
\hat{b}^*(y',y,z) = b(y') - \mathds{1}_{D_{\abs{\theta}}}(y) (b(y) + b^*(y,z)), 
\end{equation}
it holds that
\begin{align}
&\Big(\Big(\Big(\hat{b}^*(Y_t^{\theta},Y_{t_k}^{\theta},Z_t^{\theta,k})\cdot \nabla\Big) \bar{X}_{t,T}^{Y_t^{\theta}} \Big) \cdot \nabla\Big) f(\bar{X}_{t,T}^{Y_t^{\theta}})\nonumber\\
&\quad = \Big(\Big(\Big(\hat{b}^*(Y_t^{\theta},Y_{t_k}^{\theta},Z_t^{\theta,k}) \cdot \nabla\Big) \Big(\bar{X}_{t,T}^{Y_t^{\theta}} - \bar{X}_{t,T}^{Y_{t_k}^{\theta}} \Big) \Big) \cdot \nabla\Big) f(\bar{X}_{t,T}^{Y_t^{\theta}})\nonumber\\
&\qquad + \Big( \Big(\Big(\hat{b}^*(Y_t^{\theta},Y_{t_k}^{\theta},Z_t^{\theta,k}) \cdot \nabla\Big) \bar{X}_{t,T}^{Y_{t_k}^{\theta}} \Big) \cdot \nabla\Big) \Big( f(\bar{X}_{t,T}^{Y_t^{\theta}}) - f(\bar{X}_{t,T}^{Y_{t_k}^{\theta}}) \Big)\nonumber\\
&\qquad + \Big(\Big(\Big(\hat{b}^*(Y_t^{\theta},Y_{t_k}^{\theta},Z_t^{\theta,k}) \cdot \nabla\Big) \bar{X}_{t,T}^{Y_{t_k}^{\theta}} \Big) \cdot \nabla\Big) f(\bar{X}_{t,T}^{Y_{t_k}^{\theta}}).\label{dertria}
\end{align}
The first part of the factor involving~$b$ has the form
\begin{align}
&b(Y_t^{\theta}) - \mathds{1}_{D_{\abs{\theta}}}(Y_{t_k}^{\theta}) b(Y_{t_k}^{\theta})\nonumber\\
&\quad = \bigg[b(Y_t^{\theta})- b(Y_{t_k}^{\theta})\bigg] 
+ \bigg[b(Y_{t_k}^{\theta}) - \mathds{1}_{D_{\abs{\theta}}}(Y_{t_k}^{\theta}) b(Y_{t_k}^{\theta}) \bigg]\nonumber\\
&\quad = \int_{t_k}^t \mathds{1}_{D_{\theta}}(Y_{t_k})\Big( \Big((b(Y_{t_k}^{\theta}) + b^*(Y_{t_k}^{\theta},Z_u^{\theta,k}))\cdot\nabla \Big) b(Y_u^{\theta})\nonumber\\
&\qquad + \frac{1}{2} \Big(\Big((\sigma(Y_{t_k}^{\theta}) +  \sigma^*(Y_{t_k}^{\theta},Z_u^{\theta,k}) )(\sigma(Y_{t_k}^{\theta}) +  \sigma^*(Y_{t_k}^{\theta},Z_u^{\theta,k}) )^{\top}\Big):D^2 \Big) b(Y_u^{\theta}) \Big) du\nonumber\\
&\qquad+\int_{t_k}^t \mathds{1}_{D_{\theta}}(Y_{t_k})\Big( (\sigma(Y_{t_k}^{\theta}) + \sigma^*(Y_{t_k}^{\theta},Z_u^{\theta,k}))\cdot\nabla \Big) b(Y_u^{\theta}) dW_u \nonumber\\
&\qquad + b(Y_{t_k}^{\theta}) (1 - \mathds{1}_{D_{\abs{\theta}}}(Y_{t_k}^{\theta})),\label{dertria1}
\end{align}
where the integral w.r.t.~$u$ is uniformly bounded in~$\theta$ by~$C(t-t_k)$ in~$L^2(\mathbb{P})$ norm, the stochastic integral is uniformly bounded in~$\theta$ by~$C(t-t_k)^{\frac{1}{2}}$ in~$L^2(\mathbb{P})$ norm and the last term has the same property as the integral w.r.t.~$u$ (and in fact of arbitrary order in~$t-t_k$) by the calculation of inequalities~(47),~(48) in~\cite{MR4079431}. 
Using the definition~\eqref{bstardef} for~$b^*$ along with~$q'\geq 3$, there exists a constant~$\nu_2 \geq 2$ such that the remaining part of the factor involving~$b$ from~\eqref{dertria} has the bound
\begin{equation}\label{dertria2}
\abs{\mathds{1}_{D_{\abs{\theta}}}(Y_{t_k}^{\theta}) b^*(Y_{t_k}^{\theta},Z_t^{\theta,k})} \leq C\abs{b(Y_{t_k}^{\theta})} \abs{Z_t^{\theta,k}}^{\nu_2}
\end{equation}
for all~$\theta\in\Theta$. Putting~\eqref{dertria1} and~\eqref{dertria2} into the first term on the right-hand side of~\eqref{dertria} and using H\"older's inequality,~Assumptions~\ref{A4lya}\ref{jgi2}\ref{ltfirst}, equations~\eqref{Zbssde1}-\eqref{sigstardef}, 
Lemma~\ref{weakconvlya0}, Lemma~\ref{firstlem}, Theorem~\ref{seclem}, Markov property (Theorem~2.13 in~\cite{MR1731794}; see also justification in the proof of Lemma~\ref{weakconvlya0}), the fact that if~$V$ is a Lyapunov function then~$(V+1)^l$ with~$0 < l\leq 1$ is also one and exponential integrability for~$U$ as in Theorem~2.9 in~\cite{MR3766391} yield 
\begin{equation}\label{ttkbd}
\mathbb{E}\bigg[\bigg\vert\Big(\Big(\Big(\hat{b}^*(Y_t^{\theta},Y_{t_k}^{\theta},Z_t^{\theta,k}) \cdot \nabla\Big) \Big(\bar{X}_{t,T}^{Y_t^{\theta}} - \bar{X}_{t,T}^{Y_{t_k}^{\theta}} \Big) \Big) \cdot \nabla\Big) f(\bar{X}_{t,T}^{Y_t^{\theta}})\bigg\vert\bigg] \leq C(t-t_k)
\end{equation}
for all~$t\in[t_k,t_{k+1})$,~$\theta\in\Theta$. The same arguments can be used for the second term on the right-hand side of~\eqref{dertria}, along with the additional estimate
\begin{align*}
&\mathbb{E}[\abs{\partial_i f(\bar{X}_{t,T}^{Y_t^{\theta}}) - \partial_i f(\bar{X}_{t,T}^{Y_{t_k}^{\theta}})}^r]\\
&\quad \leq \mathbb{E}\bigg[\bigg\vert\int_0^1 
\nabla \partial_i f(\lambda \bar{X}_{t,T}^{Y_t^{\theta}} + (1-\lambda) \bar{X}_{t,T}^{Y_{t_k}^{\theta}} )d\lambda \cdot (\bar{X}_{t,T}^{Y_t^{\theta}} - \bar{X}_{t,T}^{Y_{t_k}^{\theta}} ) \bigg\vert^r\bigg] \\
&\quad \leq  C\Big(1 + \mathbb{E}\Big[\abs{\bar{X}_{t,T}^{Y_t^{\theta}}}^{2q^{\dagger}}\Big] + \mathbb{E}\Big[\abs{\bar{X}_{t,T}^{Y_{t_k}^{\theta}}}^{2q^{\dagger}}\Big]\Big)^{\frac{r}{2}}  \Big(\mathbb{E}\abs{\bar{X}_{t,T}^{Y_t^{\theta}} - \bar{X}_{t,T}^{Y_{t_k}^{\theta}} }^2\Big)^{\frac{r}{2}}\\
&\quad \leq C\bigg(\mathbb{E}\bigg[\exp\bigg(\frac{U(Y_t^{\theta})}{e^{\rho t}}\bigg)\bigg] + \mathbb{E}\bigg[\exp\bigg(\frac{U(Y_{t_k}^{\theta})}{e^{\rho t_k}}\bigg)\bigg]\bigg)\mathbb{E}\Big[\abs{Y_t^{\theta} - Y_{t_k}^{\theta}}^r\Big]\\
&\quad \leq C(t-t_k)^{\frac{r}{2}}
\end{align*}
where~$r>1$, in order to obtain the same right-hand bound as~\eqref{ttkbd}. For the last term on the right-hand side of~\eqref{dertria}, we rely more prominently on the Markov property. For any~$R>0$, it holds that
\begin{align*}
&\mathbb{E}\bigg[\Big(\Big(\Big(\hat{b}^*(Y_t^{\theta},Y_{t_k}^{\theta},Z_t^{\theta,k}) \cdot \nabla\Big) \bar{X}_{t,T}^{Y_{t_k}^{\theta}}  \Big) \cdot \nabla\Big) f\Big(\bar{X}_{t,T}^{Y_{t_k}^{\theta}}\Big)\wedge R\bigg]\\
&\quad= \mathbb{E}\bigg[\mathbb{E}\bigg[\Big(\Big(\Big(\hat{b}^*(Y_t^{\theta},Y_{t_k}^{\theta},Z_t^{\theta,k}) \cdot \nabla\Big) \bar{X}_{t,T}^{Y_{t_k}^{\theta}}  \Big) \cdot \nabla\Big) f\Big(\bar{X}_{t,T}^{Y_{t_k}^{\theta}}\Big) \wedge R \bigg\vert \mathcal{F}_t \bigg] \bigg]\\
&\quad= \sum_{i = 1}^n\mathbb{E}\bigg[\hat{b}_i^*(Y_t^{\theta},Y_{t_k}^{\theta},Z_t^{\theta,k}) \mathbb{E}\Big[ \Big(\partial_i \bar{X}_{t,T}^{Y_{t_k}^{\theta}}\ \cdot \nabla\Big) f\Big(\bar{X}_{t,T}^{Y_{t_k}^{\theta}}\Big) \wedge R \Big\vert \mathcal{F}_t \Big] \bigg]\\
&\quad= \sum_{i = 1}^n\mathbb{E}\bigg[\mathbb{E}\Big[\hat{b}_i^*(Y_t^{\theta},Y_{t_k}^{\theta},Z_t^{\theta,k})\Big\vert \mathcal{F}_{t_k}\Big] \mathbb{E}\Big[ \Big(\partial_i \bar{X}_{t,T}^{Y_{t_k}^{\theta}}\ \cdot \nabla\Big) f\Big(\bar{X}_{t,T}^{Y_{t_k}^{\theta}}\Big) \wedge R \Big\vert \mathcal{F}_{t_k}\Big]  \bigg],
\end{align*}
so that~\eqref{bhatstar1},~\eqref{dertria1} and~\eqref{dertria2}, where the only order~$\frac{1}{2}$ term in~$t-t_k$ from~\eqref{dertria1} has vanished, together with the same arguments as before and dominated convergence in~$R$ yields
\begin{equation}\label{cttk}
\mathbb{E}\bigg[\Big(\Big(\Big(\hat{b}^*(Y_t^{\theta},Y_{t_k}^{\theta},Z_t^{\theta,k}) \cdot \nabla\Big) \bar{X}_{t,T}^{Y_{t_k}^{\theta}}  \Big) \cdot \nabla\Big) f\Big(\bar{X}_{t,T}^{Y_{t_k}^{\theta}}\Big) \bigg] \leq C (t-t_k)
\end{equation}
for all~$t\in[t_k,t_{k+1})$,~$\theta\in\Theta$. Gathering the arguments from~\eqref{ttkbd} onwards, the integrals involving~$b$ in~\eqref{weakexp} have been shown to be of order~$t-t_k$. For the integrals involving~$\sigma$ in~\eqref{weakexp}, after rewriting
\begin{align*}
&\sigma(Y_t^{\theta})\sigma(Y_t^{\theta})^{\top} -\mathds{1}_{D_{\abs{\theta}}}(Y_{t_k}^{\theta} ) (\sigma(Y_{t_k}^{\theta}) + \sigma^*(Y_{t_k}^{\theta},Z_t^{\theta,k}))(\sigma(Y_{t_k}^{\theta}) + \sigma^*(Y_{t_k}^{\theta},Z_t^{\theta,k}))^{\top}\\
&\quad = \Big(\sigma(Y_t^{\theta}) - \mathds{1}_{D_{\abs{\theta}}}(Y_{t_k}^{\theta} ) (\sigma(Y_{t_k}^{\theta}) + \sigma^*(Y_{t_k}^{\theta},Z_t^{\theta,k}))\Big)\sigma(Y_t^{\theta})^{\top}\\
&\qquad + \mathds{1}_{D_{\abs{\theta}}}(Y_{t_k}^{\theta} ) (\sigma(Y_{t_k}^{\theta}) + \sigma^*(Y_{t_k}^{\theta},Z_t^{\theta,k})) \Big(\sigma(Y_t^{\theta})^{\top} - (\sigma(Y_{t_k}^{\theta}) + \sigma^*(Y_{t_k}^{\theta},Z_t^{\theta,k}))^{\top}\Big)
\end{align*}
and similarly
\begin{align*}
&\Big(\Big(\partial_i \bar{X}_{t,T}^{Y_t^{\theta}} \otimes \partial_j \bar{X}_{t,T}^{Y_t^{\theta}} \Big) : D^2\Big) f\Big(\bar{X}_{t,T}^{Y_t^{\theta}}\Big) + \Big(\partial_{ij}^2 \bar{X}_{t,T}^{Y_t^{\theta}} \cdot \nabla\Big) f\Big(\bar{X}_{t,T}^{Y_t^{\theta}}\Big)\\
&\quad = \Big(\Big(\Big(\partial_i \bar{X}_{t,T}^{Y_t^{\theta}} - \partial_i \bar{X}_{t,T}^{Y_{t_k}^{\theta}}\Big) \otimes \partial_j \bar{X}_{t,T}^{Y_t^{\theta}} \Big) : D^2\Big) f\Big(\bar{X}_{t,T}^{Y_t^{\theta}}\Big)\\
&\qquad + \Big(\Big(\partial_{ij}^2 \bar{X}_{t,T}^{Y_t^{\theta}} - \partial_{ij}^2 \bar{X}_{t,T}^{Y_{t_k}^{\theta}} \Big)\cdot \nabla\Big) f\Big(\bar{X}_{t,T}^{Y_t^{\theta}}\Big)\\
&\qquad +  \Big(\Big(\Big(\partial_i \bar{X}_{t,T}^{Y_t^{\theta}} - \partial_i \bar{X}_{t,T}^{Y_{t_k}^{\theta}}\Big) \otimes \partial_j \bar{X}_{t,T}^{Y_{t_k}^{\theta}} \Big) : D^2\Big) f\Big(\bar{X}_{t,T}^{Y_t^{\theta}}\Big)\\
&\qquad  + \Big( \partial_{ij}^2 \bar{X}_{t,T}^{Y_{t_k}^{\theta}} \cdot \nabla\Big) \Big(f\Big(\bar{X}_{t,T}^{Y_t^{\theta}}\Big)\! -\! f\Big(\bar{X}_{t,T}^{Y_{t_k}^{\theta}}\Big) \Big)\\
&\qquad + \Big(\Big( \partial_i \bar{X}_{t,T}^{Y_{t_k}^{\theta}} \otimes \partial_j \bar{X}_{t,T}^{Y_{t_k}^{\theta}} \Big) : D^2\Big) \Big( f\Big(\bar{X}_{t,T}^{Y_t^{\theta}}\Big) - f\Big(\bar{X}_{t,T}^{Y_{t_k}^{\theta}}\Big)\Big) \\
&\qquad + \Big( \partial_{ij}^2 \bar{X}_{t,T}^{Y_{t_k}^{\theta}} \cdot \nabla\Big)  f(\bar{X}_{t,T}^{Y_{t_k}^{\theta}}) + \Big(\Big( \partial_i \bar{X}_{t,T}^{Y_{t_k}^{\theta}} \otimes \partial_j \bar{X}_{t,T}^{Y_{t_k}^{\theta}} \Big) : D^2\Big)  f\Big(\bar{X}_{t,T}^{Y_{t_k}^{\theta}}\Big),
\end{align*}
the same bound as~\eqref{cttk} holds for all of~\eqref{weakexp} by the same treatment as for~\eqref{cttk}.
\end{proof}

\begin{remark}\label{remprof}
In this remark, it is sketched that 
some relatively abstract weakening of Assumption~\ref{A4lya}, which is made to include the case of globally Lipschitz coefficients, is sufficient for Theorem~\ref{weakconvlya} to hold. Here, the main issues 
are that we would like to include~$U$ growing logarithmically (for polynomial Lyapunov functions) instead of assuming the lower bound~$U\geq c'\abs{1+x}^{\gamma'}$ and also to have a bound in place of~\eqref{bubound} that serves the same purpose as~\eqref{bubound}. 
There are two uses of these conditions in the proofs that require particular attention. The first is the exponential integrability property of the discretization given by~\cite[Theorem~2.9]{MR3766391}, which uses only~\eqref{bubound} out of the two conditions. The second is in obtaining a good enough order for the last term on the right-hand side of~\eqref{dertria1}, which uses the derivation for~(47)-(48) in~\cite{MR4079431}. 
Moreover, for the first point about exponential integrability, an inspection of the proofs of Theorem~2.9 and Lemma~2.8 both in~\cite{MR3766391} shows that~\eqref{bubound} (with~$\abs{\alpha} = 1$) is only strictly required for estimates of the form
\begin{equation}\label{leest}
\abs{b(x)} + \|\sigma(x)\| \leq Cs^{-\alpha \gamma} \qquad \forall s\in (0,\sup_k t_{k+1}-t_k]
\end{equation}
for~$x\in\{ x\in\mathbb{R}^n : U(x) \leq C(\sup_k t_{k+1}-t_k)^{-\alpha} \}$ and appropriately small~$\alpha>0$.\\
To resolve the issues, assume for all~$k>1$ that there exists~$U_k$ satisfying the assumptions on~$U$, except, in place of the corresponding parts in Assumption~\ref{A4lya}, that
\begin{itemize}
\item $U_k(x)\geq c'\log(1+\abs{x})$ for all~$x$,
\item inequality~\eqref{leest} holds for~$x\in D'$ for some~$D' = D'(\theta)\subset\mathbb{R}^n$ replacing the set appearing in~\eqref{stotamapp} and satisfying
\begin{itemize}
\item for any~$\alpha>0$ and~$k>1$, there exists~$c>0$ such that~$D' \in \mathcal{B}( \{ U_k \leq c (\sup_{k'} t_{k'+1}-t_{k'})^{-\alpha}\})$ and
\item it holds that~$\mathbb{P}(Y_T^{\theta} \in D'(\theta))=O(\sup_{k'} t_{k'+1}-t_{k'})$ as~$\sup_{k'} t_{k'+1}-t_{k'}\rightarrow 0$,
\end{itemize}
\item there exists~$K>1$ such that the inequalities $\abs{\partial^{\alpha}b} + \|\partial^{\alpha}\sigma\|\leq Ce^{\frac{U_k}{ke^{\rho T}}}$,~$\abs{\bar{U}} \leq c(1+U_k)^{\gamma}$ hold for~$k>K$ and~$\rho$ independent of~$k$.
\end{itemize}

In Assumption~\ref{A4lya}, these conditions are satisfied by taking~$U=U_k$ for all~$k>1$. In the globally Lipschitz (with polynomial growing second derivatives) case, one can take~$D':=\{ y : \abs{y} < (\sup_{k'} t_{k'+1}-t_{k'})^{-\epsilon} \}$ with small~$\epsilon>0$ independent of~$k$, in which case, for~$U_k:\mathbb{R}^n\rightarrow\mathbb{R}$ given by $U_k(y) = (k^2\tilde{c}e^{\rho T})^2+k^2\tilde{c}e^{\rho T}\log (1+\abs{y}^2)$,~$\bar{U} = 0$ and some large enough~$\rho,\tilde{c} >0$ depending only on the global Lipschitz constant (for~$b,\sigma$) and the degree of the polynomial bound of the second derivatives respectively and not on~$k$, it holds that
\begin{itemize}
\item for any~$\alpha>0$ and~$k>1$, there exists~$c>0$ such that~$D' \in \mathcal{B}(\{ U_k \leq c (\sup_{k'} t_{k'+1}-t_{k'})^{-\alpha}\})$,
\item $\abs{b(x)} + \|\sigma(x)\| \leq C(1+\abs{x}) \leq C(1+ (\sup_{k'} t_{k'+1}-t_{k'})^{-\epsilon})$
for all~$x\in D'$ (attaining~\eqref{leest} and playing the role of~(54),~(85) both in~\cite{MR3766391}),
\item with the exponential integrability of~$U_k$ given by the last two points and following the approach of~$(47)$-$(48)$ in~\cite{MR4079431},
\begin{align*}
&\mathbb{P}\bigg(\abs{Y_T^{\theta}}\geq \bigg(\sup_{k'} t_{k'+1}-t_{k'}\bigg)^{-\epsilon} \bigg)\\
&= \mathbb{P}\bigg( e^{\frac{U_k(Y_T^{\theta})}{ke^{\rho T}}} \geq e^{k^3\tilde{c}^2e^{\rho T}}\bigg(\bigg(\sup_{k'} t_{k'+1}-t_{k'}\bigg)^{-2\epsilon} +1\bigg)^{k\tilde{c}}\ \bigg) \\
&\leq C e^{-k^3\tilde{c}^2e^{\rho T}} \bigg(\bigg(\sup_{k'} t_{k'+1}-t_{k'}\bigg)^{-2\epsilon} +1\bigg)^{-k\tilde{c}},
\end{align*}
which is arbitrary order in~$\sup_{k'} t_{k'+1}-t_{k'}$ for large enough~$k$.
\end{itemize}

\end{remark}

\section{Examples}\label{lyaexamples}
In this section, specific examples are provided where the results presented above are applicable. As stated in the introduction, most of the examples in~\cite{https://doi.org/10.48550/arxiv.1309.5595,MR3766391} are viable and many Lyapunov functions have already been given in these references (applicable here after a simple transformation, see Remark~\ref{before51}). A notable exception is the stochastic SIR model, where the Lipschitz constant of the diffusion coefficients grow too quickly for the Lyapunov functions given there (besides, the domain in that example is not~$\mathbb{R}^n$ as assumed for the main results in the present work). Here, the focus is placed on two particular examples differing in some considerable way to analogies in the aforementioned references. In Section~\ref{langwithvarlya}, our results are applied to the underdamped Langevin dynamics with variable friction, which in general (for example as soon as friction depends on position) does not have globally Lipschitz (nor monotone) coefficients; this is motivated by the work~\cite{optfric}. In Section~\ref{stochdufflya}, a Lyapunov function ($V_0$ satisfying~$LV_0\leq CV_0$) is given for the Stochastic Duffing-van der Pol equation with parameter values not accounted for in previous works mentioned above. 
\subsection{Langevin equation with variable friction}\label{langwithvarlya}
Here, the backward Kolmogorov equation and Poisson equation associated with the Langevin equation are shown to hold even in cases where the friction matrix depends on both position and velocity variables. The pointwise solution to the backward Kolmogorov equation may be used to obtain a distributional solution to the associated Poisson equation and in doing so, comprises a first step towards a gradient formula for the asymptotic variance as in~\cite{optfric}. In addition and perhaps more importantly, solutions to the Poisson equation allows one to obtain central limit theorems for additive functionals~\cite[Section~3]{MR3069369}. The results here give a rigorous way to derive distributional solutions to the Poisson equation in the proof of Proposition~3.10 in~\cite{MR3069369} (in particular, it is not clear that the domain of~$L^*$ includes~$C_c^{\infty}$, given the interpretation of~$L$ as a limit in~$L^2$ earlier in the same section). 
In this case, hypoellipticity is required to complete the argument to obtain the central limit theorem, which means that Proposition~4.18 in~\cite{MR3305998} may also be used in the case of continuous bounded observables; the results here extend the space of observables beyond that of continuous bounded functions at the cost of stronger assumptions on the coefficients of the SDE.\\
\begin{assumption}\label{langa}
The function $U\in C^3(\mathbb{R}^n)$ is such that there exists $\tilde{k},\tilde{K}>0$ with 
$\nabla U(q)\cdot q \geq \tilde{k}\abs*{x}^2-\tilde{K}$ for all $q\in\mathbb{R}^n$. The friction matrix $\Gamma\in C^\infty(\mathbb{R}^{2n}, \mathbb{R}^{n\times n})\cap L^{\infty}$ is symmetric positive definite everywhere such that there exist\footnote{It is possible to allow for $\beta_1 = 1$, but at the cost of more stringent bounds on the coefficients.} $\beta_1<1$, $\tilde{m}, \tilde{M}>0$ with $\abs*{\nabla\!_p\cdot\Gamma(q,p)}
< \tilde{M} (1+\abs*{q}^{\beta_1} + \abs*{p}^{\beta_1})$ and $\Gamma(q,p)\geq \tilde{m}I$ for all $q,p\in\mathbb{R}^n$.
\end{assumption}
Note Assumption~\ref{langa} implies that for $R>1$, $q\in\mathbb{R}^n$ with $\abs*{q} = 1$,
\begin{align*}
U(Rq) - U(q) &= \int_1^R\nabla U(\lambda q)\cdot\frac{\lambda q}{\lambda} d\lambda \geq \int_1^R (\tilde{k}\abs{\lambda q}^2 - \tilde{K})\lambda^{-1}d\lambda\\
& = \frac{\tilde{k}(R^2-1)}{2} - \tilde{K}\log R,
\end{align*}
which yields $U(q) \geq \frac{\tilde{k}}{4}\abs*{q}^2 - C$ for all $q\in\mathbb{R}^n$ and some constant $C>0$. Consider $\mathbb{R}^{2n}$-valued solutions $(q_t,p_t)$ to
\begin{subequations}\label{lvf}
\begin{align}
dq_t &= p_tdt\\
dp_t &= -\nabla U(q_t)dt + \nabla\!_p\cdot\Gamma(q_t,p_t) dt - \Gamma(q_t,p_t)p_tdt + \sqrt{\Gamma}(q_t,p_t)dW_t,
\end{align}
\end{subequations}
where $\sqrt{\Gamma}$ denotes some matrix satisfying $\sqrt{\Gamma}\sqrt{\Gamma}^{\top} = \Gamma$ and $(\nabla\!_p\cdot\Gamma)_i = \sum_j\nabla_{p_j}\Gamma_{ij}$. Equation~\eqref{lvf} is~\eqref{sde0} with~$b(t,(q,p))=(p,-\nabla U(q) + \nabla\!_p\cdot\Gamma(q,p) - \Gamma(q,p)p)$ and~$\sigma_{i,j}=0$ for~$(i,j)\in\{(i',j'):i'\in[1,n]\cap\mathbb{N}\}\cup\{(i',j'):j'\in[1,n]\cap\mathbb{N}\}$,~$(\sigma(t,(q,p)))_{n+i,n+j}=(\sqrt{\Gamma}(q,p))_{i,j}$ for~$i,j\in[1,n]\cap\mathbb{N}$. 
For~$b = \min(\tilde{k}^{-1}(\sup_{\mathbb{R}^{2n}}\abs*{\Gamma})^{-1},\tilde{m},\tilde{k}^{\frac{1}{2}})$,~$a = \frac{1}{4}\min(\frac{b}{\tilde{k}},\tilde{m})$, let
\begin{equation}\label{friclya}
V_{\gamma}(q,p) = e^{\gamma(U(q) + a\abs*{q}^2 + bq\cdot p + \abs*{p}^2)}.
\end{equation}
In the following,~$\abs*{M}$ denotes the operator norm of~$M\in\mathbb{R}^{n\times n}$. Proposition~\ref{Langevinprop} shows that the assumptions on the coefficients~$b,\sigma$ of~\eqref{sde} in Theorems~\ref{alters2},~\ref{alters3},~\ref{alters4} are satisfied.
\begin{prop}\label{Langevinprop}
Under Assumption~\ref{langa}, there exists constants $c_1, c_2, c_3>0$ such that for all $\gamma$ satisfying
\begin{equation}\label{gammabs}
0<\gamma \leq \gamma^* := \frac{1}{8}\min\bigg(\bigg(\tilde{k}b\sup_{\mathbb{R}^{2n}}\abs*{\Gamma}\bigg)^{-1}, \tilde{m}\bigg(4\sup_{\mathbb{R}^{2n}}\abs*{\Gamma}\bigg)^{-1}\bigg),
\end{equation}
it holds that
\begin{equation}\label{lhslang}
LV_{\gamma}(q,p) \leq (c_1 - c_2\abs*{q}^2 - c_3\abs*{p}^2)\gamma V_{\gamma}(q,p)
\end{equation}
for all $(q,p)\in\mathbb{R}^{2n}$, where $L$ is the generator~\eqref{jen} associated with~\eqref{lvf}, explicitly given by
\begin{align*}
LV(q,p) &= p\cdot\nabla\!_q V(q,p) - \nabla U(q)\cdot\nabla\!_p V(q,p) + (\nabla\!_p\cdot\Gamma(q,p))\cdot\nabla\!_p V(q,p) \\
&\quad- (\Gamma(q,p)p)\cdot\nabla\!_p V(q,p) + (1/2)\Gamma(q,p):D_p^2V(q,p).
\end{align*}
If in addition there exist $0<\beta_2<1$, $\bar{M}>0$ such that 
\begin{align*}
\abs*{\partial_i ( \nabla\!_p\cdot\Gamma(q,p) - \nabla U(q))} &\leq \bar{M}(1 - \inf U + U(q) + \abs*{p}^2)^{\beta_2}\\
\abs*{\partial_i \Gamma(q,p)} &\leq \bar{M}(1 - \inf U^{\frac{1}{2}} + U(q)^{\frac{1}{2}} + \abs*{p})^{\beta_2}\\
\abs*{\partial_i\partial_j(\nabla\!_p\cdot\Gamma(q,p) - \nabla U(q))} + \abs*{\partial_i\partial_j\Gamma(q,p)} &\leq \bar{M}(1+e^{(U(q) + \abs*{p}^2)^{\beta_2}})
\end{align*}
for all $q,p\in\mathbb{R}^n$, $i,j\in\{1,\dots, 2n\}$, then Assumptions~\ref{A1prime} and~\ref{A2prime} (with~$p=2$) are satisfied with $V_i=\hat{V}_k^{s,T} = V_{\gamma}$ with any $\gamma$ satisfying~\eqref{gammabs}, $G(q,p) = C(1 - \inf U + U(q) + \abs*{p}^2)^{\beta_3}$ for some constants $C>0$ and $\beta_2 < \beta_3 < 1$.
\end{prop}

\begin{proof}
The left-hand side of~\eqref{lhslang} calculates as
\begin{align}
&(p\cdot\nabla\!_q - \nabla\!_q U(q) \cdot \nabla\!_p + (\nabla\!_p\cdot\Gamma(q,p))\cdot\nabla\!_p - (\Gamma(q,p)p)\cdot\nabla\!_p + \Gamma(q,p):D^2) V_{\gamma}(q,p)\nonumber\\
&\quad = (2aq\cdot p + b\abs*{p}^2 - b\nabla\!_q U(q)\cdot q + (\nabla\!_p\cdot\Gamma(q,p) - \Gamma(q,p)p)\cdot(bq+2p)\nonumber\\
&\qquad + 2\textrm{Tr}\Gamma(q,p) + \gamma\Gamma(q,p): (b^2qq^{\top} + 4pp^{\top}) ) \gamma V_{\gamma}(q,p)\nonumber\\
&\quad\leq \bigg( \bigg( a - \frac{b}{\tilde{k}} + \frac{1}{2} b^2\abs*{\Gamma} + b^2\gamma \abs*{\Gamma} \bigg) \abs*{q}^2 + \bigg(a + b + \frac{1}{2}\abs*{\Gamma} -2\tilde{m} + 4\gamma\abs*{\Gamma}\bigg)\abs*{p}^2 \nonumber\\
&\qquad + \tilde{M}(1+\abs*{q}^{\beta_1} + \abs*{p}^{\beta_2})\abs*{bq + 2p} + b\tilde{K} + 2\textrm{Tr}\Gamma\bigg) \gamma V_{\gamma}(q,p)\nonumber\\
&\quad\leq \bigg( c - \frac{b}{16\tilde{k}}\abs*{q}^2 - \frac{\tilde{m}}{16}\abs*{p}^2\bigg)\gamma V_{\gamma}(q,p)\label{friclanglya}
\end{align}
for some constant $c>0$. The last assertion follows by straightforward applications of Young's inequality.
\end{proof}

For $U$ with locally Lipschitz third derivatives and by Theorems~\ref{alters2},~\ref{alters3}, the associated Poisson equation with suitable right-hand side $\hat{f} = f-\int_{\mathbb{R}^{2n}}fd\mu \in L^2(\mu)$ holds in the distributional sense as in the proof of Proposition~3.1 in (the arXiv version\footnote{The published version uses a different approach, which does not generalize to~\eqref{lvf} as easily.} 
of)~\cite{optfric} if in addition
\begin{equation}\label{ptgoz}
\abs*{\mathbb{E}[\hat{f}(z_t^\cdot)]} + \abs*{\int_t^{\infty}\mathbb{E} [\hat{f}(z_s^\cdot)]ds}\rightarrow 0 \textrm{ in }L^2(\mu) \textrm{ as } t\rightarrow\infty,
\end{equation}
where for any $z\in\mathbb{R}^{2n}$, $z_t^z = (q_t,p_t)$ solves~\eqref{lvf}, $\mathbb{P}((q_0,p_0) = z) = 1$ and~$\mu(dq,dp) = Z^{-1}e^{-U(q) - \frac{p^2}{2}}dqdp$ is the invariant probability measure with normalizing constant $Z$. We obtain~\eqref{ptgoz} in the following by using the ergodicity results of \cite{MR2499863}, see alternatively Theorem~2.4 in \cite{MR1807683}. 
The proof of Proposition~1.2 in \cite{MR1807683} can be modified for~\eqref{lvf} to obtain
\begin{prop}\label{modded}
For every~$z\in\mathbb{R}^{2n}$, $t>0$, the measure~$P^t(z,\cdot):\mathcal{B}(\mathbb{R}^{2n})\rightarrow [0,1]$ given by $P^t(z,A)=\mathbb{P}(z_t^z\in A)$ admits a density $p_t(z,\cdot)$ satisfying $p_t(z,z')>0$ for Lebesgue almost every $z'\in\mathbb{R}^{2n}$ and
\begin{equation}\label{conttol1}
(z\mapsto p_t(z,\cdot) )\in C(\mathbb{R}^{2n},L^1(\mathbb{R}^{2n})).
\end{equation}
\end{prop}
\begin{proof}
For the Markov property, see the proof of Lemma~\ref{weakconvlya0} just before~\eqref{texpwp}. The proof in the aforementioned reference follows through except in the proof of Lemma~1.1 in \cite{MR1807683}, where the Lyapunov function~\eqref{friclya} is to be used in place of~$\tilde{H}(x,y) = \frac{1}{2}\abs*{y}^2+V(x) - \inf_{\mathbb{R}^n}V + 1$ and $R^2$ in the ensuing calculations is replaced as needed.
\end{proof}
Proposition~\ref{modded} implies the existence of an irreducible skeleton chain, namely, the existence of a sequence~$(P^{mk})_{k\in\mathbb{N}}$ for some $m>0$ satisfying that there exists a~$\sigma$-finite measure~$\mu$ on~$(\mathbb{R}^{2n},\mathcal{B}(\mathbb{R}^{2n}))$ for which if~$\mu(A)>0$, then for all~$z\in\mathbb{R}^n$, there exists~$k$ where~$P^{mk}(z,A) > 0$. 
Together with the following results in~\cite{MR2499863}: Theorem~3.2 (with~$\Psi = (\Psi_1,\Psi_2)$,~$\Psi_1(x) = \Psi_2(x) = (x/2)^{\frac{1}{2}}$,~$\phi(x) = x^{\frac{1}{2}}$, $V = V_{\frac{1}{2}\gamma^*}$), Theorem~3.4 (compact sets are petite by Theorem~4.1(i) in~\cite{MR1234294}, where in particular non-evanescence follows by Theorem~3.1 in~\cite{MR1234295} for which Theorem~7.4 in~\cite{MR958914} is enough to get a Borel right process) and Proposition~3.1 (with~$\phi(x) = x$,~$V = V_{\gamma^*}$), this yields~\eqref{ptgoz} for~$\hat{f}$ satisfying~$\hat{f}/V_{\frac{1}{8}\gamma^*} \in L^{\infty}$. 

\subsection{Stochastic Duffing-van der Pol equation}\label{stochdufflya}
We verify here that the Stochastic Duffing-van der Pol oscillator admits a Lyapunov function satisfying the assumptions of Theorem~\ref{seclem}. Note that in doing so, the difficult parts of Assumption~\ref{A4lya} are shown to be satisfied, so that our Theorem~\ref{weakconvlya} about weak numerical convergence rates applies. In particular, the logarithm of the Lyapunov function described below may be used for~$U$ in Assumption~\ref{A4lya}. The version of the equation considered is from \cite{MR3364862} with $\beta_2=0$, which is less general than in \cite{MR3364862} but still includes the setting of Section~13.1 in \cite{MR1214374} and \cite{MR1430980} for example. 
Specifically, for $(W^{(1)},W^{(3)}):[0,T]\times \Omega\rightarrow \mathbb{R}^2$ a standard $(\mathcal{F}_t)_{t\in [0,T]}$-adapted Brownian motion, $\alpha_1,\alpha_2,\beta_1,\beta_3\in\mathbb{R}$, $\alpha_3 > 0$, consider $\mathbb{R}^2$-valued solutions to
\begin{subequations}\label{sdvdp}
\begin{align}
dX_t^{(1)} &= X_t^{(2)}dt,\\
dX_t^{(2)} &= [\alpha_1 X_t^{(1)} - \alpha_2 X_t^{(2)} - \alpha_3 X_t^{(2)} (X_t^{(1)})^2 - (X_t^{(1)})^3]dt\\
&\quad  + \beta_1 X_t^{(1)} dW_t^{(1)} + \beta_3 dW_t^{(3)}.
\end{align}
\end{subequations}
Equation~\eqref{sdvdp} is~\eqref{sde0} with~$b(t,x) = (x_2,\alpha_1x_1-\alpha_2x_2 - \alpha_3x_2x_1^2 - x_1^3)$ and~$\sigma_{1,1}=\sigma_{1,2}=0$,~$(\sigma(t,x))_{2,1}=\beta_1x_1$,~$(\sigma(t,x))_{2,2}=\beta_3$. 
Let~$V:\mathbb{R}^2\rightarrow\mathbb{R}$ be given by
\begin{align*}
V(x_1,x_2) &= V_1(x_1,x_2)+V_2(x_1,x_2)\\
&:= (1-\eta(x_1))e^{\gamma(x_1^4 + ax_1x_2 + bx_2^2)} + e^{\gamma(-cx_1x_2 + \frac{1}{2}x_2^2)}.
\end{align*}
The following Proposition~\ref{vdpprop} verifies that~$V$ provides as~$V=e^U$ an appropriate Lyapunov function satisfying the assumptions of Theorem~\ref{weakconvlya}.
\begin{prop}\label{vdpprop}

There exists a constant $C^*>0$ such that $LV\leq C^*V$, where $L$ is the generator~\eqref{jen} associated with~\eqref{sdvdp}. Moreover, Assumptions~\ref{A1lya} and~\ref{A2lya} are satisfied with $G(t,x) = (3+2\sum_i \abs*{\alpha_i} + \beta_1^2)(1+\abs*{x_1}^3+\abs*{x_2}^{\frac{3}{2}})$ and $\hat{V}_k(t,x) = (\abs{X_t^{(1)}}^4 + 2\abs{X_t^{(2)}}^2+1)^k$ for~$t\geq 0$,~$x=(x_1,x_2)\in\mathbb{R}^n$, where~$(X_t^{(1)},X_t^{(2)})$ solves~\eqref{sdvdp} with~$(X_t^{(1)},X_t^{(2)}) = (x_1,x_2)$.
\end{prop}

\begin{proof}
The functions $V_1$ and $V_2$ satisfy
\begin{align*}
LV_1(x_1,x_2) &= \bigg[ (2\alpha_1b - \alpha_2a)x_1x_2 + (a-2\alpha_2b + 2\beta_3^2\gamma b^2)x_2^2 + (\alpha_1 a + \frac{1}{2} \beta_3^2\gamma a^2 \\
&\quad + \beta_1^2b) x_1^2+ (2\beta_1^2\gamma b^2 - 2\alpha_3b)x_1^2x_2^2  - (\alpha_3 a + 2b - 4) x_1^3x_2 \\
&\quad  + (\frac{1}{2} \beta_1^2 \gamma a^2- a)x_1^4 + b\beta_3^2 - \frac{x_2\partial_{x_1}\eta(x_1)}{1-\eta(x_1)} \bigg]\gamma V_1(x_1,x_2)\\
LV_2(x_1,x_2) &= \bigg[\bigg(\frac{1}{2}\beta_3^2\gamma - c - \alpha_2\bigg)x_2^2 + \bigg(\frac{1}{2}c^2\beta_3^2\gamma - \alpha_1 c + \frac{1}{2}\beta_1^2\bigg) x_1^2\\
&\quad + (\alpha_2 c + \alpha_1) x_1x_2 + (\alpha_3 c - 1) x_1^3 x_2 + \bigg(c + \frac{1}{2}c^2 \gamma \beta_1^2\bigg) x_1^4\\
&\quad +  \bigg(\frac{1}{2}\beta_1^2\gamma- \alpha_3\bigg) x_1^2x_2^2+ \frac{1}{2}\beta_3^2\bigg]\gamma V_2(x_1,x_2).
\end{align*}
where $\frac{1}{1-\eta(x_1)}:=0$ whenever $1-\eta(x_1) = 0$. 
In order to see~$LV\leq CV$, consider separately the set where $x_1^2\leq\frac{1+\abs*{a - 2\alpha_2 b + 2\beta_3^2\gamma b^2}}{ 2\alpha_3b - 2\beta_1^2\gamma b^2}$ and its complement in $\mathbb{R}^2$. In the former case, $V_1(x_1,x_2) = LV_1(x_1,x_2) = 0$ and by our choice of $c$ and $\gamma$, there exists a generic constant $C>0$ such that $LV_2\leq C V_2$, therefore $LV\leq CV$. Otherwise in the complementary case where $\abs*{x_1}$ is bounded below, we have $LV_1 \leq C V_1$ and when in addition $x_1\in \textrm{supp}\eta \cup B_1(0)$, it holds that $LV_2\leq CV_2$. It remains to estimate $LV_2$ when $x_1\notin \textrm{supp}\eta\cup B_1(0)$, in which case we have $\abs*{x}^i e^{\gamma(-cx_1x_2 + \frac{1}{2}x_2^2)}\leq C e^{\gamma(\frac{1}{2}x_1^4 + \frac{3}{4}x_2^2)} \leq CV_1(x_1,x_2)$ for $i\leq 4$, from which $LV_2\leq CV_1$. \\
For the second assertion, it is straightforward to see that~\eqref{a3},~\eqref{a4} hold and that the higher derivatives of the coefficients of~\eqref{sdvdp} are bounded above in terms of $\hat{V}_k$ for any $k,p$ as called-for in Assumption~\ref{A2lya}. For~\eqref{a52}, consider separately the cases $\abs*{x_1}\leq \sup\{\abs*{x}:x\in\textrm{supp}\eta\}$ and otherwise. In the former case, it holds that
\begin{equation*}
G(x_1,x_2) \leq C(1+\abs*{x_2}^{\frac{3}{2}}),
\end{equation*}
which yields that for any $m>0$, there is $M = M(m)>0$ continuous in $m$ such that
\begin{equation}\label{gabove}
G \leq m\log(V_2) + M\leq m\log(V) + M.
\end{equation}
When~$\abs*{x_1} > \sup\{\abs*{x}:x\in\textrm{supp}\eta\}$, inequalities~\eqref{gabove} continue to hold with $V_1$ replacing $V_2$ and a corresponding continuous function $m\mapsto M(m)$. 
\end{proof}

\begin{remark}
Alternative to~$V$ above, it is also possible to take the Lyapunov function given by~\cite[Section~3.1.4]{MR4079431}. For example, let~$\alpha>0$,~$U_0,U_1$ be given as in~\cite[Section~3.1.4]{MR4079431} and for~$x\in\mathbb{R}^3$, let~$Y_{\cdot}^x:\Omega\times[0,T]\rightarrow\mathbb{R}^3$ be given by~$Y_t^x = (X_t^{(1)},X_t^{(2)},X_t^{(3)})$, where~$(X_t^{(1)},X_t^{(2)})$ is given by~\eqref{sdvdp} and~$X_t^{(3)}$ satisfies~$dX_t^{(3)} = U_1(X_t^{(1)},X_t^{(2)})e^{-\alpha t}dt$ with~$Y_0^x=x$~$\mathbb{P}$-almost surely. The derivation in~\cite[inequality~(57)]{MR4079431} implies that there exists~$\alpha$ such that the function~$V:\Omega\times[0,T]\times\mathbb{R}^3\rightarrow(0,\infty)$ given by~$V(t,x) = \exp(U_0(t,(Y_t^x)_1,(Y_t^x)_2)e^{-\alpha t} + (Y_t^x)_3)$ is a Lyapunov function. 
\end{remark}

\begin{appendix}
\section{Auxiliary results}
Just as in the case of globally Lipschitz coefficients in~\cite[Lemma~5.10]{MR1731794}, the regularity of an extended system and the harmonic property of the expectation~\eqref{actualetstw} are required. These properties are established for our setting in the following. 

Throughout the section, we assume~$O = \mathbb{R}^n$ and~$b,\sigma,f,c,g$ are nonrandom functions. Moreover, we suppose all of the assumptions in Theorem~\ref{fk2} hold (including those in the last statement~\ref{tstwiii}). In particular,~$f:[0,\infty)\times \mathbb{R}^n\rightarrow\mathbb{R}$,~$c:[0,\infty)\times \mathbb{R}^n\rightarrow[0,\infty)$ and $g:\mathbb{R}^n\rightarrow\mathbb{R}$ are Borel functions satisfying that~$f(t,\cdot),c(t,\cdot),g(\cdot)$ are continuous for every~$t\in[0,T]$,~$\int_0^T \sup_{x\in B_R}(\abs{c(t,x)} + \abs{f(t,x)})dt <\infty$ for every~$R>0$ and such that for~$h\in\{f,c,g\}$,~$R>0$, there exists $C \geq 0$,~$0<\bar{l}\leq 1$, Lyapunov functions~$V^{s,T}$, locally bounded~$\tilde{x}$ for which for any~$s\in[0,T]$ it holds~$\mathbb{P}$-a.s. that 
\begin{subequations}\label{xxgg0}
\begin{align}
\abs*{h(s+t,X_t^{s,x})}&\leq C(1+V^{s,T}(t,\tilde{x}(x)))^{\bar{l}},\\
\abs*{h(s+t,y)-h(s+t,y')} &\leq C\abs*{y-y'}\label{xxgg02}\\
\textrm{and if~$h\in\{f,g\}$,}\quad V^{s+\tau,T}(0,\tilde{x}(X_{\tau}^{s,x}))^{\bar{l}} &\leq C(1+V^{s,T}(\tau,\tilde{x}(x)))\label{xxgg03}
\end{align}
\end{subequations}
for all~$t\leq T-s$, 
stopping times~$\tau \leq T$,~$x\in\mathbb{R}^n$ and~$y,y'\in B_R$. 

For any~$s\geq 0$,~$T>0$,~$x\in\mathbb{R}^n$,~$x',x''\in\mathbb{R}$, consider solutions $X_t^{s,x}$ to~\eqref{sde} appended with the corresponding $\mathbb{R}$-valued solutions~$X_t^{(n+1),s,x'}$ and~$X_t^{(n+2),s,x'}$ to
\begin{subequations}\label{addingonapp}
\begin{align}
X_t^{(n+1),s,x'} &= x' + \int_0^t c(s+r,X_r^{s,x})dr,\\
X_t^{(n+2),s,x''} &= x'' + \int_0^t f(s+r,X_r^{s,x})e^{-X_r^{(n+1),s,x'}}dr
\end{align}
\end{subequations}
on~$[0,T]$, denoted $\bar{X}_t^{s,y} = (X_t^{s,x},X_t^{(n+1),s,x'},X_t^{(n+2),s,x''})$, $y = (x,x',x'')$. 
Let $\bar{X}_t^{s,y}(I)$ be the corresponding Euler approximation analogous to~\eqref{euap} with~$I$ as in the beginning of Lemma~\ref{eulercan}.
\begin{lemma}\label{regsde}
Under the assumptions of this section, for every~$R,T>0$, it holds that
\begin{equation*}
\sup_{s\in[0,T]}\sup_{y\in B_R} \mathbb{P}\bigg( \sup_{t\in[0,T]}\abs*{\bar{X}_t^{s,y} - \bar{X}_t^{s,y}(I)} > \epsilon \bigg) \rightarrow 0
\end{equation*}
as $\sup_k t_{k+1} - t_k \rightarrow 0$.
\end{lemma}
\begin{proof}
For any~$R'>0$, let~$R_X^{s,x}(I,R')\in\mathcal{F}$ denote the event
\begin{equation*}
R_X^{s,x}(I,R') = \bigg\{ \sup_{t\in[0,T]} \abs*{X_t^{s,x}}\leq R' \bigg\} \cap \bigg\{ \sup_{t\in[0,T]} \abs*{X_t^{s,x}(I)}\leq R' \bigg\}.
\end{equation*}
For any $\epsilon, R'>0$, it holds that
\begin{align*}
&\mathbb{P}\bigg(\sup_{t\in[0,T]}\abs*{\bar{X}_t^{s,y} - \bar{X}_t^{s,y}(I) } > \epsilon\bigg)\\ &\quad\leq \mathbb{P}\bigg(\sup_{t\in[0,T]}\abs*{X_t^{s,x}} > R'\bigg) + \mathbb{P}\bigg(\sup_{t\in[0,T]}\abs*{X_t^{s,x}(I)} > R'\bigg)\\
&\qquad + \mathbb{P}\bigg(\sup_{t\in [0,T]}\abs*{\bar{X}_t^{s,y} - \bar{X}_t^{s,y}(I) } > \epsilon\  \bigg\vert R_X^{s,x}(I,R') \bigg).
\end{align*}
Fix~$\epsilon'>0$. For any~$T,R>0$, we may choose~$R' = R^*$ so that, by Lemma~2.2 in~\cite{MR1731794}, the sum of the first and second term on the right-hand side is bounded above by~$\epsilon'/2$ uniformly in~$s\in[0,T]$ and~$x\in B_R$. For the last term on the right, note that by our assumptions on $c$, there exists locally bounded $\tilde{G}:\mathbb{R}^n\rightarrow[0,\infty)$ such that
\begin{align}
&\sup_{t\in[0,T]}\abs*{c(s+t,X_t^{s,x}) - c(s+t,X_t^{s,x}(I))}\nonumber\\
&\quad\leq \sup_{t\in[0,T]}\abs*{X_t^{s,x} - X_t^{s,x}(I)}(\tilde{G}(X_t^{s,x}) + \tilde{G}(X_t^{s,x}(I)))\label{xxgg}
\end{align}
and such that for~$I_r := \max\{t_k:r\geq t_k\}$,
\begin{align}
&\mathbb{E}\bigg[\sup_{t\in [0,T]}\bigg\vert \int_0^t c(s+r,X_r^{s,x}(I)) - c(s+r,X_{I_r}^{s,x}(I)) dr\bigg\vert \mathds{1}_{R_X^{s,x}(I,R')} \bigg] \nonumber\\
&\quad \leq \mathbb{E}\bigg[\sup_{t\in [0,T]}\bigg\vert\int_0^t (X_r^{s,x}(I) - X_{I_r}^{s,x}(I)) (\tilde{G}(X_r^{s,x}(I)) + \tilde{G}(X_{I_r}^{s,x}(I))) dr \bigg\vert \mathds{1}_{R_X^{s,x}(I,R')} \bigg] \nonumber\\
&\quad =  \mathbb{E}\bigg[\sup_{t\in [0,T]} \bigg\vert\int_0^t \bigg(\int_{I_r}^r b(s+r',X_{I_{r'}}^{s,x}(I))dr' + \int_{I_r}^r \sigma(s+r', X_{I_{r'}}^{s,x}(I)) dW_{r'} \bigg)\nonumber\\
&\qquad\cdot(\tilde{G}(X_r^{s,x}(I)) + \tilde{G}(X_{I_r}^{s,x}(I)) dr \bigg\vert \mathds{1}_{R_X^{s,x}(I,R')} \bigg] \label{xxgg2}
\end{align}
for all~$s\in[0,T]$,~$y=(x,x',x'')\in\mathbb{R}^{n+2}$, where we have used~\eqref{xxgg02}. 
By~\eqref{xxgg}, it holds that
\begin{align*}
&\bigg\{\sup_{t\in[0,T]} \abs*{X_t^{s,x} - X_t^{s,x}(I)} \leq \frac{\epsilon}{12\sqrt{3}T\sup_{z\in B_{R^*}}\tilde{G}(z)} \bigg\} \cap R_X^{s,x}(I,R^*)\\
&\quad\subset \bigg\{\sup_{t\in[0,T]}\abs*{c(s+t,X_t^{s,x}) - c(s+t,X_t^{s,x}(I))}\leq \frac{\epsilon}{6\sqrt{3}T}\bigg\} \cap R_X^{s,x}(I,R^*)\\
&\quad\subset \bigg\{ \int_0^T \abs*{c(s+u,X_u^{s,x}) - c(s+u,X_u^{s,x}(I))}du \leq \frac{\epsilon}{6\sqrt{3}} \bigg\} \cap R_X^{s,x}(I,R^*).
\end{align*}
This yields
\begin{align*}
&\mathbb{P}\bigg( \bigg\{\sup_{t\in[0,T]} \abs*{X_t^{(n+1),s,x'} - X_t^{(n+1),s,x'}(I)} > \frac{\epsilon}{3\sqrt{3}}\bigg\} \cap  R_X^{s,x}(I,R^*) \bigg)\nonumber\\
&\quad \leq \mathbb{P}\bigg( \bigg\{\sup_{t\in[0,T]} \abs*{ \int_0^t (c(s+r,X_r^{s,x}) - c(s+r,X_r^{s,x}(I)) ) dr } > \frac{\epsilon}{6\sqrt{3}}\bigg\} \nonumber\\
&\qquad \cap  R_X^{s,x}(I,R^*) \bigg) + \mathbb{P}\bigg( \bigg\{\sup_{t\in[0,T]} \abs*{ \int_0^t (c(s+r,X_r^{s,x}(I)) + c(s+r,X_{I_r}^{s,x}(I)) ) dr } \nonumber\\
&\qquad > \frac{\epsilon}{6\sqrt{3}}\bigg\} \cap  R_X^{s,x}(I,R^*) \bigg)\nonumber\\
&\quad \leq \mathbb{P} \bigg( \bigg\{\sup_{t\in[0,T]} \abs*{X_t^{s,x} - X_t^{s,x}(I)} > \frac{\epsilon}{12\sqrt{3}T\sup_{z\in B_{R^*}}\tilde{G}(z)} \bigg\} \cap R_X^{s,x}(I,R^*) \bigg)\nonumber\\
&\qquad + \frac{6\sqrt{3}}{\epsilon}\mathbb{E} \bigg[ \sup_{t\in[0,T]} \bigg\vert \int_0^t c(s+r,X_r^{s,x}(I)) - c(s+r,X_{I_r}^{s,x}(I)) dr\bigg\vert \mathds{1}_{ R_X^{s,x}(I,R^*)} \bigg],
\end{align*}
which converges to zero as~$\sup_k t_{k+1} - t_k \rightarrow 0$ by~\eqref{xxgg2} and Lemma~\ref{eulercan}. In particular, 
there exists~$0<\delta\leq \delta^*$ such that for~$I$ satisfying~$\sup_{k\geq 0} t_{k+1} - t_k \leq \delta$, it holds that
\begin{align}
&\mathbb{P}\bigg( \sup_{t\in[0,T]} \abs*{X_t^{(n+1),s,x'} - X_t^{(n+1),s,x'}(I)} > \frac{\epsilon}{3\sqrt{3}} \bigg\vert R_X^{s,x}(I,R^*) \bigg)\nonumber\\
&\quad \leq \mathbb{P}\bigg(\sup_{t\in [0,T]}\abs*{X_t^{s,x} - X_t^{s,x}(I)} > \frac{\epsilon}{12\sqrt{3}T\sup_{z\in B_{R^*}}\tilde{G}(z)}\bigg) \leq \frac{\epsilon'}{6}\label{xxgg3}
\end{align}
for all~$s\in[0,T]$ and~$y = (x,x',x'')\in B_R \subset \mathbb{R}^{n+2}$. By a similar argument and using the above,~\eqref{xxgg3} holds with~$n+1$ replaced by~$n+2$ and~$x'$ by~$x''$. Together with Lemma~\ref{eulercan}, the lemma is proved.
\end{proof}
Next, the harmonic property (see~\cite[Definition~3.1]{MR1731794}) of~\eqref{actualetstw} is shown. Let~$\bar{g}$ given by~$\bar{g}(y) = x'' + g(x) e^{-x'}$ for all~$y = (x,x',x'')\in\mathbb{R}^{n+2}$ and for~$T>0$,~$s\in[0,T]$, let~$\bar{v}:[0,\infty)\times\mathbb{R}^{n+2}\rightarrow\mathbb{R}$ be given by
\begin{equation}\label{altue}
\bar{v}(s,y) = \mathbb{E}[\bar{g}(\bar{X}_{T-s}^{s,y})]= \mathbb{E}\Big[X_{T-s}^{(n+2),s,x''} + g(X_{T-s}^{s,x})e^{-X_{T-s}^{(n+1),s,x'}}\Big].
\end{equation}
In addition for a bounded subset~$Q\subset (0,T)\times \mathbb{R}^{n+2}$, let~$\tau$ be the stopping time
\begin{equation}\label{vbarst}
\tau := \inf\{u\geq 0:(s+u,\bar{X}_u^{s,y})\notin Q\}.
\end{equation}
The next lemma establishes the equality~$\bar{v}(s,y) = \mathbb{E}[\bar{v}(s+(\tau\wedge t),\bar{X}_{\tau\wedge t}^{s,y})]$ under our setting.
\begin{lemma}\label{tstwalem}
Under the assumptions of this section, for any~$T>0$, any bounded subset~$Q\subset (0,T)\times\mathbb{R}^{n+2}$,~$(s,y)\in Q$,~$t\in[0,T-s]$, it holds that
\begin{equation*}
\mathbb{E}[\bar{g}(\bar{X}_{T-s}^{s,y})] = \int \int \bar{g}\bigg(\bar{X}_{T-s-(\tau(\omega)\wedge t)}^{s+(\tau(\omega)\wedge t),\bar{X}_{\tau(\omega)\wedge t}^{s,y} (\omega)}(\omega')\bigg) d\mathbb{P}(\omega')d\mathbb{P}(\omega),
\end{equation*}
where~$\tau$ is defined by~\eqref{vbarst}.
\end{lemma}
\begin{proof}
For any~$R,T>0$,~$t\in[0,T]$,~$(s,y)\in Q$ with~$y = (x,x',x'')$, by Theorem~2.13 in~\cite{MR1731794} together with Lemma~\ref{regsde}, it holds for~$\mathbb{P}$-a.a.~$\omega$ that
\begin{equation}\label{scuff}
\mathbb{E}[(\mathds{1}_{B_R}\bar{g})(\bar{X}_{T-s}^{s,y}) \vert \mathcal{F}_{\tau\wedge t} ] = \int (\mathds{1}_{B_R}\bar{g})\bigg(\bar{X}_{T-(s+(\tau(\omega)\wedge t))}^{s+(\tau(\omega)\wedge t), \bar{X}_{\tau(\omega)\wedge t}^{s,y}(\omega)}(\omega')\bigg)d\mathbb{P}(\omega'),
\end{equation}
so that the right-hand side is~$\mathcal{F}_{\tau\wedge t}$-measurable. Moreover for~$\mathbb{P}$-a.a.~$\omega$, by~\eqref{xxgg0}, the absolute value of the integrand in the right-hand side is bounded independently of~$R$ as
\begin{align}
&(\mathds{1}_{B_R}\abs*{\bar{g}})\bigg(\bar{X}_{T-(s+(\tau(\omega)\wedge t))}^{s+(\tau(\omega)\wedge t), \bar{X}_{\tau(\omega)\wedge t}^{s,y}(\omega)}(\omega')\bigg) - \abs{X_{\tau(\omega) \wedge t}^{(n+2),s,x''}(\omega)} \nonumber\\
&\quad\leq  \int_0^{T-s-(\tau(\omega)\wedge t)}\bigg\vert f\bigg(s+(\tau(\omega)\wedge t)+r, X_r^{s+(\tau(\omega)\wedge t), X_{\tau(\omega) \wedge t}^{s,x}(\omega)}(\omega')\bigg)\bigg\vert dr \nonumber\\
&\qquad + \bigg\vert g\bigg(X_{T-(s+(\tau(\omega)\wedge t))}^{s+(\tau(\omega)\wedge t), X_{\tau(\omega)\wedge t}^{s,x}(\omega)}(\omega')\bigg)\bigg\vert\nonumber\\
&\quad \leq  C\bigg(\int_0^{T-s-(\tau(\omega)\wedge t)}\Big(1+V^{s+(\tau(\omega)\wedge t),T}(\omega',r,\tilde{x}(X_{\tau(\omega)\wedge t}^{s,x}(\omega)))\Big)^{\bar{l}} dr\nonumber\\
&\qquad + \Big(1+V^{s+(\tau(\omega)\wedge t),T} (\omega',T-s-(\tau(\omega)\wedge t), \tilde{x}(X_{\tau(\omega)\wedge t}^{s,x}(\omega)))\Big)^{\bar{l}}\bigg),\label{scuff1}
\end{align}
where we have abused the notation~$V^{\cdot,\cdot}$ to refer to Lyapunov functions for both~$f$ and~$g$. 
Since~$\bar{l}^{\textrm{th}}$-powers of Lyapunov functions are still Lyapunov functions (but with different auxiliary processes), the expectation in~$\omega'$ of the right-hand side of this is bounded by 
Corollary~\ref{huddecor0} 
and
~\eqref{xxgg03} as in
\begin{align}
&\int \bigg( \int_0^{T-s-(\tau(\omega)\wedge t)}\Big(1+V^{s+(\tau(\omega)\wedge t),T}(\omega',r,\tilde{x}(X_{\tau(\omega)\wedge t}^{s,x}(\omega)))\Big)^{\bar{l}} dr\nonumber\\
&\quad + \Big(1+V^{s+(\tau(\omega)\wedge t),T} (\omega',T-s-(\tau(\omega)\wedge t), \tilde{x}(X_{\tau(\omega)\wedge t}^{s,x}(\omega)))\Big)^{\bar{l}} \bigg) d\mathbb{P}(\omega')\nonumber\\
&\quad\leq 
C\bigg\|e^{\int_0^T \abs*{\alpha_u^{s+(\tau(\omega)\wedge t),T}(\omega')}du}\bigg\|_{L^{\frac{p^{s+(\tau(\omega)\wedge t),T}}{p^{s+(\tau(\omega)\wedge t),T} - 1}}(d\mathbb{P}(\omega'))}\nonumber\\
&\qquad \cdot  \int \Big(1+V^{s+(\tau(\omega)\wedge t),T}(\omega',0,\tilde{x}(X_{\tau(\omega)\wedge t}^{s,x}(\omega)))\Big)^{\bar{l}} d\mathbb{P}(\omega')\nonumber\\
&\quad \leq C
 \bigg( 1+ \int V^{s,T}(\omega',\tau(\omega)\wedge t,\tilde{x}(x)) d\mathbb{P}(\omega') \bigg)\nonumber\\
&\quad\leq C
 (1+ V^{s,T}(0,\tilde{x}(x))  )\nonumber\\
&\quad  < \infty,\label{scuff2}
\end{align}
where~$\alpha_{\cdot}^{\cdot,T}$ and~$p^{\cdot,T}$ are the obvious objects associated with~$(1+V^{\cdot,T})^{\bar{l}}$. 
Therefore by dominated convergence, the right-hand side of~\eqref{scuff} converges to the same expression but without~$\mathds{1}_{B_R}$ for~$\mathbb{P}$-a.a.~$\omega$. 
Moreover, by~\eqref{xxgg0} and Corollary~\ref{huddecor0}, 
\begin{align*}
\mathbb{E}[\abs{X_{\tau\wedge t}^{(n+2),s,x''}}] -\abs{x''} &\leq  \mathbb{E}\bigg[\int_0^{\tau\wedge t}\Big\vert f\Big(s + r,X_r^{s,x''}\Big)\Big\vert dr\bigg]\\
& \leq C\int_0^T \mathbb{E}\Big[1+V^{s,T}(r, \tilde{x}(x''))\Big] dr\\
& \leq C \Big(1 + V^{s,T}(0,\tilde{x}(x''))\Big).
\end{align*}
Consequently, together with~\eqref{scuff1},~\eqref{scuff2} 
and dominated convergence (in~$\omega$), it holds that
\begin{align*}
&\int \int (\mathds{1}_{B_R}\bar{g})\bigg(\bar{X}_{T-(s+(\tau(\omega)\wedge t))}^{s+(\tau(\omega)\wedge t), \bar{X}_{\tau(\omega)\wedge t}^{s,y}(\omega)}(\omega')\bigg)d\mathbb{P}(\omega') d\mathbb{P}(\omega) \\
&\quad\rightarrow \int \int \bar{g}\bigg(\bar{X}_{T-(s+(\tau(\omega)\wedge t))}^{s+(\tau(\omega)\wedge t), \bar{X}_{\tau(\omega)\wedge t}^{s,y}(\omega)}(\omega')\bigg)d\mathbb{P}(\omega') d\mathbb{P}(\omega)
\end{align*}
as~$R\rightarrow\infty$. On the other hand, by a similar argument as above, the expectation of the left-hand side of~\eqref{scuff} has the limit
\begin{equation*}
\mathbb{E}[\mathbb{E}[(\mathds{1}_{B_R}\bar{g})(\bar{X}_{T-s}^{s,y}) \vert \mathcal{F}_{\tau\wedge t} ]] = \mathbb{E}[(\mathds{1}_{B_R}\bar{g})(\bar{X}_{T-s}^{s,y}) ] \rightarrow \mathbb{E}[\bar{g}(\bar{X}_{T-s}^{s,y})]
\end{equation*}
as~$R\rightarrow\infty$.
\end{proof}

\end{appendix}


\bibliography{document}


\end{document}